\documentclass[reqno]{amsart}
\usepackage{amssymb}
\usepackage{amsmath}
\usepackage{amsthm}
\usepackage[all]{xy}
\usepackage{mathtools}
\usepackage{stmaryrd}
\usepackage{bm}

\usepackage{enumerate}
\usepackage{setspace}
\usepackage{comment}

\theoremstyle{plain}
\newtheorem{thm}{Theorem}[section]
\newtheorem{prop}[thm]{Proposition}
\newtheorem{cor}[thm]{Corollary}
\newtheorem{dfn}[thm]{Definition}
\newtheorem{lem}[thm]{Lemma}

\newcommand{\suspend}{}

\newcommand{\cc}{\mathbb{C}}
\newcommand{\zz}{\mathbb{Z}}

\newcommand{\irr}[1]{\mathrm{Irr}(#1)}

\newcommand{\triv}{\mathrm{triv}}
\newcommand{\sgn}{\mathrm{sgn}}

\newcommand{\htriv}{\mathrm{htriv}}
\newcommand{\hsgn}{\mathrm{hsgn}}

\newcommand{\Hom}{\mathrm{Hom}}
\newcommand{\End}{\mathrm{End}}
\newcommand{\Ext}{\mathrm{Ext}}
\newcommand{\im}{\mathrm{Im}}

\newcommand{\gs}[1]{\langle #1\rangle}
\newcommand{\pd}{\mathrm{pd\,}}
\newcommand{\gch}{\mathrm{gch\,}}
\newcommand{\gdim}{\mathrm{gdim\,}}
\newcommand{\tr}{\mathrm{Tr\,}}

\newcommand{\gep}[2]{\mathbf{gEP}(#1,#2)}
\newcommand{\hgmod}{\mathrm{\mathchar`-gmod}}

\newcommand{\C}{C}

\newcommand{\filt}{\blacktriangleright}

\mathtoolsset{showonlyrefs}
\allowdisplaybreaks

\begin{document}

\title{Representation theoretic interpretation of the Springer correspondence for dihedral groups}
\author{Sususmu Higuchi}
\date{\today}
\begin{abstract}
    The Lusztig–Shoji algorithm is generalized to a complex reflection group $W$ and give us a version of the Springer correspondence of $W$. 
    We show that the combinatorics of generalized Springer correspondences of dihedral groups of order $2n$ exhibit the Brauer-Humphreys type reciprocity as in the case of Weyl groups for odd $n$, and these constitute a major portion of the stratification of the natural module categories attached to them.
\end{abstract}

    \maketitle

\section{Introduction}
Green polynomials, introduced by Green \cite{Green1955}, is a family of polynomials indexed by a pair of partitions that describes a part of the character table of a general linear group $\mathop{GL} (n, \mathbb F_q)$.
They are transformations of the Kostka polynomials, which are $q$-analogues of the Kostka numbers that appear in representation theory of $\mathop{GL} (n, \mathbb C)$ (see e.g.~\cite{Fulton2004}). As such, Kostka polynomials satisfies the orthogonality relation arising from the orthogonality of characters of $\mathop{GL} (n, \mathbb F_q)$. 
In more concrete terms, we have a $\mathbb Z [\![q]\!]$-valued matrix equality
\begin{equation}\label{eq:LS-orth}P \cdot \Lambda \cdot P^{\mathtt t} = \Omega,
\end{equation}
where $P = (P_{\lambda, \mu})$ is the square matrix recording Kostka polynomials, $\Lambda$ is a diagonal matrix, $\Omega = (\Omega_{\lambda, \mu})$ is the matrix defined as
$$\Omega_{\lambda, \mu} = \frac{1}{n!} \sum_{\sigma \in \mathfrak S_n} \frac{\chi_\lambda ( \sigma) \chi_{\mu} ( \sigma )}{\det (1_n - q \cdot \rho (\sigma))} \in \mathbb Z [\![q]\!],$$
where $\chi _\lambda (\sigma)$ is the character value of the $\mathfrak S_n$-representation corresponding to the partition $\lambda$ evaluated at $\sigma \in \mathfrak S_n$, and $\rho$ is the permutation representation of $\mathfrak S_n$. In conjunction with the triangularity relation
\begin{equation}\label{eq:tri-rel}P_{\lambda, \mu} = \begin{cases} 0 & \lambda \not\ge \mu \\ 1 & \lambda = \mu \end{cases},
\end{equation}
the equation (\refeq{eq:LS-orth}) has a unique solution valued in $\mathbb Z_{\ge 0} [q]$.

The definition of Green polynomials is extended to an arbitrary Weyl group $W$, and they are also connected with the representation theory of the corresponding finite groups of Lie types (\cite{Deligne-Lusztig1976, Lusztig1984, Lusztig1990}).
There, we have a preorder on the set of (isomorphism classes of) irreducible representations $\irr{W}$ of $W$ coming from the generalized Springer correspondence, and we have an appropriate generalization of the matrix $\Omega$ for $W$.
It is an observation of Shoji that (\refeq{eq:LS-orth}) and (\refeq{eq:tri-rel}) also determine such generalized Green/Kostka polynomials.

Achar-Aubert \cite{Achar2007} generalized the definition of Green polynomials to the case of dihedral groups $W = D_{n}$, that are Coxeter groups but not necessarily a Weyl group. In their situation, the preorders on $\irr{W}$ are given by using partitions arising from special representations of $W$.
Their main result is that a suitable class of preorder yields a unique solution of analogues of (\refeq{eq:LS-orth}) and (\refeq{eq:tri-rel}) which are valued in $\mathbb Z_{\ge 0} [q]$. 

In \cite{kato2013homological}, the row vectors of the matrix $P$ are interpreted as a collection of modules over an algebra that satisfies an orthogonality relation with respect to the $\mathrm{Ext}$-functor. This fact can be understood as a version of the Brauer-Humphreys type reciprocity.
The goal of this paper is to prove that all the preorders on $\irr{W}$ considered in \cite{Achar2007} admit such interpretations, and they constitute a major part of all the preorders with this property when $n$ is odd.

To state our results precisely, we introduce some more notation: Let $S := \mathbb C [X,Y]$ be the polynomial ring on which $W = D_n$ ($n$ is odd) acts as the reflection representation. The algebra $(S * W)$ is graded, and hence we have the category $(S*W)\mathchar`-\mathrm{gmod}$ of finitely generated graded $(S*W)$-modules. This category admits autoequivalences $\left< i \right>$ ($i \in \zz$) that shift the grading.
The set of isomorphism classes of simple graded modules of $(S * W)$ is indexed by $\irr{W}$ (up to $\left<\bullet\right>$). Let $L_\lambda$ be the simple graded $(S*W)$-module corresponding to $\lambda \in \irr{W}$, and let $P_\lambda$ be its projective cover.

In accordance with \cite{kato2013homological}, we define two modules $\widetilde{K}_\lambda, K_\lambda \in (S*W)\mathchar`-\mathrm{gmod}$ determined by a given preorder $\precsim$ on $\irr{W}$ (and indexed by $\irr{W}$) as follows:
\[\widetilde{K}_\lambda = P_\lambda /\sum_{\lambda \prec \mu,d>0, f:P_\mu\gs{d}\rightarrow P_\lambda} \im f\]
\[K_\lambda = P_\lambda /\sum_{\lambda \precsim \mu,d>0, f:P_\mu\gs{d}\rightarrow P_\lambda} \im f.\]

Note that an analogue of (\refeq{eq:tri-rel}) for $K_\lambda$'s, that is the multiplicity formula
$$[K_{\lambda} : L_{\mu}] = \begin{cases}0 & (\lambda \prec \mu) \\ 1 & (\lambda = \mu) \end{cases},$$
holds by its definition. 

We have the following analogue of the Brauer-Humphreys type reciprocity:
\begin{prop}[$=$ Proposition \ref{prop:reciprocity} + Proposition \ref{prop:verify-orth}]
    \label{BHR}\hbox{}
    Suppose $P_{\mu}$ $(\mu \in \irr{W})$ admits a finite filtration by graded shifts of $\{ \widetilde{K}_\lambda\}_{\lambda \in \irr{W}}$.
    Let $(P_{\mu} : \widetilde{K}_\lambda)_q$ be the graded multiplicity of $\widetilde{K}_\lambda$ in the filtration pieces of $P_\mu$.
    
    The collection of modules $\{ \widetilde{K}_\lambda,K_\lambda\}_{\lambda \in \irr{W}}$ satisfies the orthogonality relation
    \begin{equation}
    \Ext^i_{(S*W)\mathchar`-\mathrm{mod}}(\widetilde{K}_\lambda,K_\mu^*)_d= 
        \begin{cases}
            \cc &(\lambda = \mu \textrm{ and } i=d=0),\\
            0   &(otherwise)
        \end{cases}
            \label{eq:Ext-orth}
    \end{equation}
    if and only if 
    $$(P_{\mu} : \widetilde{K}_\lambda)_q = [K_{\lambda} : L_{\mu}]_q $$
    holds for $\lambda,\mu \in \irr{W}$.
\end{prop}
Our main result is stated as follows:

\begin{thm}[$=$ Theorem \ref{thm:mainthm} + Corollary \ref{cor:modified-springerorder} + Corollary \ref{cor:preorder-uniqueness}]\label{fmain}\hbox{}
The collection of modules $\{ \widetilde{K}_\lambda,K_\lambda\}_{\lambda \in \irr{W}}$ satisfies the one of the equivalent conditions of Proposition \ref{BHR}
   if and only if the preorder $\precsim$ is one of the following:
\begin{enumerate}[(i)]
    \item 
        $\triv\sim\sgn\sim\chi_1\sim\chi_2\sim\chi_3\sim \cdots \sim\chi_{\frac{n-1}{2}-1}\sim\chi_{\frac{n-1}{2}}$
    \item \label{item:secondcase}
        $\triv \left\{\begin{matrix}\prec\\\sim\end{matrix}\right\} \chi_1 \left\{\begin{matrix}\prec\\\sim\end{matrix}\right\} \chi_2 \left\{\begin{matrix}\prec\\\sim\end{matrix}\right\} \chi_3 \left\{\begin{matrix}\prec\\\sim\end{matrix}\right\} \cdots \left\{\begin{matrix}\prec\\\sim\end{matrix}\right\}\chi_{\frac{n-1}{2}-1} \left\{\begin{matrix}\prec\\\sim\end{matrix}\right\} \chi_{\frac{n-1}{2}}\prec\sgn$
    \item 
        $\sgn \left\{\begin{matrix}\prec\\\sim\end{matrix}\right\} \chi_1 \left\{\begin{matrix}\prec\\\sim\end{matrix}\right\} \chi_2 \left\{\begin{matrix}\prec\\\sim\end{matrix}\right\} \chi_3 \left\{\begin{matrix}\prec\\\sim\end{matrix}\right\} \cdots \left\{\begin{matrix}\prec\\\sim\end{matrix}\right\}\chi_{\frac{n-1}{2}-1} \left\{\begin{matrix}\prec\\\sim\end{matrix}\right\} \chi_{\frac{n-1}{2}}\prec\triv$.
    \end{enumerate}
\end{thm}
Note that the preorders considered in \cite{Achar2007} under the name of Springer correspondences match the case with $\triv \prec \chi_1$ in our setting.
All of them appear in the above list and form half of the preorders of case \eqref{item:secondcase}.




The organization of this paper is as follows: In Section \ref{sec:preliminaries}, we present basic material. 
In Section \ref{sec:filtofgradedmodules}, we review some of the basics of filtrations of graded modules.
In Section \ref{sec:skewgrpalg}, we consider the smash product algebra $(S * W)$ for a finite group $W$, and present its basic properties.
In Section \ref{sec:reciprocity}, we describe some properties of the Brauer-Humphreys type reciprocity.
Finally in Section \ref{sec:dihedralgroup}, we concentrate on the case $W$ is a dihedral group of odd order and present a proof of Theorem \ref{fmain}.

\subsection*{Acknowledgements}
    The author wishes to thank his supervisor Syu Kato for his patient guidance and support.    
    This work was supported by JST, the establishment of university fellowships towards the creation of science technology innovation, Grant Number JPMJFS2123.

    \section{Preliminaries}\label{sec:preliminaries}
        In this section, we introduce some notations and recall some basic results on homological algebra (cf. \cite{hazrat_2016,Rotman2009}).
        
        Let $A$ be a $\mathbb{C}$-algebra.
        We denote by $A$-mod the category of finitely generated left $A$-modules.
        All modules in this paper are assumed to be finitely generated.
        
        A $\cc$-algebra $A$ is called a graded algebra (or more precisely, a $\zz$-graded algebra) if $A = \bigoplus_{i\in\zz}A_i$, where each $A_i$ is a subspace of $A$ and $A_iA_j\subset A_{i+j}$ for all $i,j\in \zz$.
        We say $A$ is non-negatively graded if $A_i = 0$ for all $i<0$.
        We call $A_i\,(i\in\zz)$ the homogeneous part of $A$ of degree $i$.

        Let $A$ be a $\zz$-graded algebra.
        A (left) $A$-module $M$ is said to be graded if $M$ is decomposed into $M = \bigoplus_{i\in\zz}M_i$, where each $M_i$ is a subspace of $M$ and $A_iM_j \subset M_{i+j}$ hold for all $i,j\in\zz$.
        For a graded $A$-module $M$ and $d\in\zz$, we define the graded module $M\gs{d}$ by $(M\gs{d})_i\coloneqq M_{i-d}$ for each $i\in\zz$.
        We call $M\gs{d}$ the grade $d$ shift of $M$.
        We say $M$ is bounded below if there exists $n\in\zz$ such that $M_i = 0$ for all $i<n$.
        We say $M$ is locally finite if $\dim_\cc M_i<\infty$ for all $i\in\zz$.
        A graded algebra $A$ is called locally finite if $A$ is locally finite as graded $A$-module. 
        We define the graded dimension of a locally finite module $M$ by
        \[
            \gdim M \coloneqq \sum_{d\in\zz}q^d\dim M_d.   
        \]
        For finitely generated graded $A$-modules $M$ and $N$, a graded homomorphism of degree $n$ is an $A$-module homomorphism $f:M\rightarrow N$ such that $f(M_i)\subset N_{i+n}$ for all $i\in\zz$.
        We denote by $\Hom_A(M,N)_n$ the additive group consisting of all graded $A$-module homomorphisms of degree $n$.
        We define a category $A$-gmod whose the objects are finitely generated graded $A$-modules and whose morphisms are graded $A$-module homomorphisms of degree $0$.
        It is clear that $\Hom_A(M,N)_n$ is a subspace of $\Hom_A(M,N)$.
        Moreover, we have a natural decomposition
        \[
            \Hom_A(M,N) = \bigoplus_{i\in\zz}\Hom_A(M,N)_i
        \]
        since $M$ is finitely generated (see e.g.~\cite[Theorem 1.2.6]{hazrat_2016}).
        By this decomposition, we can equip $\End_A(M)$ and $\Hom_A(M,N)$ with gradings.

        \medskip
        Let $\mathcal{C}$ be an abelian category with enough projectives.
        For $M\in \mathcal{C}$, its projective resolution is an exact complex
        \[\xymatrix{
            \cdots \ar[r]^{p_{d+1}} &P_d \ar[r]^{p_d} &\cdots \ar[r]^{p_3} &P_2 \ar[r]^{p_2} &P_1 \ar[r]^{p_1} &P_0 \ar[r]^{p_0} &M \ar[r] &0.
        }\]
        such that $P_d$ is a projective object for each $d\geq 0$.
        A projective resolution is said to be finite length if $P_d=0$ for $d\gg 0$ and to be infinite length if $P_d \neq 0$ for infinitely many $d\geq 0$.
        We define the length of a resolution of finite length as the maximum index $d$ such that $P_d \neq 0$.
        We define the projective dimension of $M$ as the minimum length of a projective resolution of $M$ and denote it by $\pd M$.
        The global dimension of $\mathcal{C}$ is defined as 
        \[
            \mathrm{gl.dim}\,\mathcal{C} \coloneqq \sup \pd M,
        \] 
        where $M$ ranges over all objects of $\mathcal{C}$.
        We define the $i$-th extension group $(i\in\zz_{>0})$ by
        \[
            \Ext^i_\mathcal{C}(M,N) \coloneqq \mathbb{R}^i\Hom_\mathcal{C}(-,N)(M) = \frac{\mathrm{Ker}(\Hom_\mathcal{C}(p_{i+1},N))}{\mathrm{Im}(\Hom_\mathcal{C}(p_{i},N))}.
        \]
        The extension group $\Ext^i_\mathcal{C}(M,N)$ is independent of the choice of a projective resolution of $M$ (see e.g.~\cite[Corollary 6.57]{Rotman2009}).

        For a graded Noetherian algebra $A$, categories $A$-mod and $A$-gmod are abelian categories with enough projectives.
        We denote $\Ext_{A\textrm{-gmod}}^i(M,N\gs{-d})$ by $\Ext^i_A(M,N)_d$ for $M,N \in A$-gmod.
        Since a projective object in $A$-gmod is a projective module as an $A$-module (see e.g.~\cite[Proposition 1.2.15]{hazrat_2016}), we have a natural isomorphism 
        \begin{align}\label{eq:extgrd}
            \Ext_A^i(M,N)   &\simeq \frac{\textrm{Ker}(\Hom_A(p_{i+1},N))}{\mathrm{Im}(\Hom_A(p_{i},N))}\\
                            &\simeq \bigoplus_{n\in\zz} \frac{\mathrm{Ker}(\Hom_{A\textrm{-gmod}}(p_{i+1},N\gs{-n}))}{\mathrm{Im}(\Hom_{A\textrm{-gmod}}(p_{i},N\gs{-n}))}\\
                            &\simeq \bigoplus_{n\in\zz} \Ext^i_{A\textrm{-gmod}}(M,N\gs{-n}) = \bigoplus_{n\in\zz} \Ext^i_{A}(M,N)_n
        \end{align}
        where
        \[\xymatrix{
            \cdots \ar[r]^{p_{d+1}} &P_d \ar[r]^{p_d} &\cdots \ar[r]^{p_3} &P_2 \ar[r]^{p_2} &P_1 \ar[r]^{p_1} &P_0 \ar[r]^{p_0} &M \ar[r] &0
        }\]
        is a graded projective resolution of $M$.

        We record a simple lemma from linear algebra:
        \begin{lem}\label{lem:detinjective}
        Let $S$ be an integral commutative ring.
        Let $f:S^M\rightarrow S^N$ be a $S$-module homomorphism between free $S$-modules and $(a_{i,j})\in \mathrm{Mat}_{N,M}(S)$ be a matrix representation of $f$.
        If $M\leq N$ and $(a_{i,j})$ admits $M\times M$ submatrix $A$ such that $det\,A\neq 0$, then $f$ is injective. $\hfill\square$
        \end{lem}
        
        \medskip
        Let $X$ be a set.
        A binary relation on $X$ is a subset of $X\times X$.
        Let $B\subset X \times X$ be a binary relation on $X$.
        For each $x,y \in X$, the notation $x\mathrel{B} y$ means $(x,y)\in \mathrel{B}$, and $x\not\mathrel{B} y$ means $(x,y)\notin\mathrel{B}$.
        A binary relation $\mathrel{B}$ is said to be:
            \begin{itemize}
                \item \textbf{reflexive} if $x\mathrel{B}x$ for all $x\in X$.
                \item \textbf{symmetric} if $x\mathrel{B}y$ implies $y\mathrel{B}x$ for all $x,y \in X$.
                \item \textbf{antisymmetric} if $x\mathrel{B}y$ and $y\mathrel{B}x$ implies $x=y$ for all $x,y\in X$
                \item \textbf{transitive} if $x\mathrel{B}y$ and $y\mathrel{B}z$ implies $x\mathrel{B}z$ for all $x,y,z \in X$.
            \end{itemize}
        
        We call $\mathrel{B}$ a preorder if $\mathrel{B}$ is reflexive and transitive.
        A preorder $\mathrel{B}$ is called a partial order if $\mathrel{B}$ is antisymmetric.
        A binary relation $\mathrel{B}$ is said to be equivalence if $\mathrel{B}$ is symmetric, reflexive, and transitive.

    \section{Filtrations of graded modules}
        \label{sec:filtofgradedmodules}
        Let $A$ be a graded algebra.
        Let $M$ be a finitely generated graded $A$-module.
        Since $M$ is finitely generated, there is a surjection 
        \[\xymatrix{
           \bigoplus_{i=1}^n A\gs{d_i} \ar@{->>}[r] &M
        }\]
        for some $d_1,d_2,\ldots,d_n\in\zz $.
        Therefore, if $A$ is locally finite, then we have $\dim M_d < \sum_{i=1}^n (A\gs{d_i})_d = \sum_{i=1}^n A_{d-d_i} < \infty$ for every $d\in\zz$.
        In particular, $M$ is locally finite if $A$ is locally finite.
        If $A$ is non-negatively graded, then $M_d = 0$ for $d<\min_i\{d_i\}$ and $M$ is bounded below.
        In the rest of this section, we fix a non-negatively graded locally finite algebra $A$.
        In particular, a finitely generated $A$-module $M \in A\hgmod$ is locally finite and bounded below.

        Let $M$ be a finitely generated graded $A$-module and $F = \{F_\alpha\}_{\alpha\in\Lambda}$ be a set of finitely generated graded $A$-modules indexed by $\Lambda$ such that $F_\alpha \simeq F_{\alpha'}\gs{d}$ for $\alpha,\alpha'\in\Lambda$ and $d\in\zz$ only if $\alpha=\alpha'$ and $d=0$. 
        A sequence of submodules of M
        \[
            M = M^0 \supseteq M^1 \supseteq M^2 \supseteq \cdots \supseteq M^i \supseteq \cdots
        \]
        is called a filtration of $M$.
        We say a filtration $\{M^i\}_{i=0}^\infty$ is separable if $\bigcap_{i=0}^{\infty}M^i = 0$.
        If $M^i = 0$ for $i\gg 0$, then we say that this filtration is finite.
        For each $i\geq 1$, the quotient module $M^{i-1}/M^i$ is called the $i$-th factor of the filtration $\{M^i\}_{i=0}^\infty$.
        Clearly, a separable filtration is finite if and only if the number of its nonzero factors is finite.
        A separable filtration is called an $F$-filtration if its every factor is isomorphic to zero or $F_\alpha\gs{d}$ for some $\alpha\in \Lambda$ and $d\in\zz$.
        If there exists an $F$-filtration (resp. finite $F$-filtration) of $M$, we say $M$ is filtered by $F$ (resp. finitely filtered by $F$).
        For each $\alpha\in \Lambda$, we define the multiplicity of $F_\alpha$ in a $F$-filtration $\{M^i\}_{i=0}^\infty$ by $f_\alpha = \sum_{d\in\zz}m_dq^d$ where $m_d$ is the number of factors isomorphic to $F_\alpha\gs{d}$.
        Since $M$ is locally finite and bounded below, the multiplicity of $F_\alpha$ belongs to $\zz(\!(q)\!)$.
        Note that the multiplicity $f_\alpha$ depends on the choice of filtrations.
        When $M$ has a $F$-filtration in which the multiplicity of $F_\alpha$ is $f_\alpha$ for $\alpha\in\Lambda$, we use the notation
        \begin{equation}
            \label{eq:denotefilt}
            M \filt \sum_{\alpha\in \Lambda}f_\alpha\cdot[F_\alpha]
        \end{equation} 
        to indicate the multiplicities of factors in the filtration.
        We also use this notation without specifying a filtration to imply that there exists a $F$-filtration of $M$ in which the multiplicity of $F_\alpha$ is $f_\alpha$.
        In the right hand side of (\refeq{eq:denotefilt}), the formal sum $\sum_{\alpha\in \Lambda}f_\alpha\cdot[F_\alpha]$ belongs to  $\prod_{\alpha\in \Lambda} \zz(\!(q)\!)\cdot[F_\alpha]$.
        If a filtration $\{M_i\}_{i=0}^\infty$ is finite, then $f_\alpha\neq 0$ for only finitely many $\alpha\in\Lambda$ and $f_\alpha$ is a Laurent polynomial for every $\alpha\in\Lambda$.
        Therefore, we have $\sum_{\alpha\in \Lambda}f_\alpha\cdot[F_\alpha] \in \bigoplus_{\alpha\in \Lambda} \zz[q,q^{-1}]\cdot[F_\alpha]$.
        Conversely, $M \filt h$ for some $h \in \bigoplus_{\alpha\in \Lambda} \zz[q,q^{-1}]\cdot[F_\alpha]$ implies $M$ is finitely filtered by $F$.

        We prove some lemmas needed later.

        Let $M\in A$-gmod have a separable filtration
        \[
            M = M^0 \supseteq M^1 \supseteq M^2 \supseteq \cdots \supseteq M^i \supseteq \cdots.
        \]
        For each $i>0$, we assume that $i$-th factor $M_{i-1}/M_i$ have a $F$-filtration
        \[
            M^{i-1}/M^i = U^{i,0} \supset U^{i,1} \supset \cdots \supset U^{i,j} \supset \cdots 
        \]
        which have the multiplicities of factors $M^{i-1}/M^i \filt \sum_{\alpha \in \Lambda}f_{i,\alpha}\cdot[F_\alpha]$.
        We give some sufficient conditions to be able to merge these filtrations, that is, to construct $F$-filtration of $M$ with the multiplicity $M \filt \sum_{\alpha \in \Lambda}\sum_{i=1}^\infty f_{i,\alpha}\cdot[F_\alpha]$.
        Note that since $M$ is locally finite, $F_\alpha\gs{d}$ appears at most finitely many times in above filtrations for each $\alpha \in \Lambda$ and $d\in\zz$.
        Since $M$ is bounded below, there exists some integer $N$ for each $\alpha\in \Lambda$ such that $F_\alpha\gs{d}$ does not appear in above filtrations if $d<N$.
        Therefore, the sum $\sum_{i=1}^\infty f_{i,\alpha}$ is well-defined as a formal Laurent series.
        
        \begin{lem}
            \label{lem:cnctfilts}    
            Keep the settings as above.
            We assume one of the following conditions:
            \begin{enumerate}[\upshape(i)\itshape]
                \item \label{item:fin-filts} the filtration $\{U^{i,j}\}_{j=0}^\infty$ is a finite $F$-filtration for every $i\geq 1$,
                \item \label{item:findimfacts} the graded $A$-module $F_\alpha$ is finite dimensional for every $\alpha \in \Lambda$.
            \end{enumerate}
            Then, the graded $A$-module $M$ admits a $F$-filtration which have multiplicities of factors $ M \filt  \sum_{\alpha \in \Lambda}\sum_{i=1}^\infty f_{i,\alpha}\cdot[F_\alpha]$.
        \end{lem}

        \begin{proof}
            In the case \eqref{item:fin-filts}, we fix a number $l_i$ with $U^{i,l_i} = 0$ for each $i>0$.
            Then $\{U^{i,j}\}_{j=0}^{l_i}$ has finitely many factors $U^{i,0}/U^{i,1}, U^{i,1}/U^{i,2},\ldots , U^{i,l_i-1}/U^{i,l_i}$. 
            Let $p_i:M\rightarrow M/M^i$ be the natural quotient morphism.
            By taking preimages under $p_i$, we get a sequence of submodules of $M^{i-1}$
            \[
                M^{i-1} = p_i^{-1}(U^{i,0}) \supset p_i^{-1}(U^{i,1}) \supset \cdots \supset p_i^{-1}(U^{i,l_{\alpha_i}}) = M^i.
            \]
            Since this sequence is finite for each $i$, we can connect sequences and obtain a filtration 
            \begin{equation}\label{eq:filtofM}\begin{array}{rll}
                M   &= M^0 = p_1^{-1}(U^{1,0}) \supset p_1^{-1}(U^{1,1}) \supset \cdots \supset p_1^{-1}(U^{1,l_{\alpha_1}})\\
                    &= M^1 = p_2^{-1}(U^{2,0}) \supset p_2^{-1}(U^{2,1}) \supset \cdots \supset p_2^{-1}(U^{2,l_{\alpha_2}})\\
                    &\qquad\vdots\\
                    &= M^{i-1} = p_i^{-1}(U^{i,0}) \supset p_i^{-1}(U^{i,1}) \supset \cdots \supset p_i^{-1}(U^{i,l_{\alpha_i}})\\
                    &\qquad\vdots&.\\                    
            \end{array}\end{equation}
            
            Since $\ker p_i = M^i$, we have $\ker p_i\cap p_i^{-1}(U^{i,j-1})\subset M^i \subset p_i^{-1}(U^{i,j})$.
            Hence, a surjection $p_i$ induces an isomorphism 
            \begin{align}
                p_i^{-1}(U^{i,j-1})/p_i^{-1}(U^{i,j}) &= p_i^{-1}(U^{i,j-1})/(p_i^{-1}(U^{i,j}) + \ker p_i\cap p_i^{-1}(U^{i,j-1})) \\
                    &\simeq U^{i,j-1}/U^{i,j}.
            \end{align}
            Therefore, a filtration \eqref{eq:filtofM} is a $F$-filtration.
            By construction, this filtration has factors 
            \begin{align}
                &U^{1,0}/U^{1,1}, U^{1,1}/U^{1,2},\ldots , U^{1,l_1-1}/U^{1,l_1},\\
                &U^{2,0}/U^{2,1}, U^{2,1}/U^{2,2},\ldots , U^{2,l_2-1}/U^{2,l_2},\\
                &U^{3,0}/U^{3,1},\ldots
            \end{align}
            as desired.

            Next, we prove case \eqref{item:findimfacts}.
            In this case, we cannot simply concatenate filtrations because each filtration might be infinite.
            We begin with the case of $M^2 = 0$.
            In this case, the second factor of a filtration $\{M^i\}$ is $M^1/M^2 = M^1$.
            Since $F_\alpha$ is finite dimensional for all $\alpha \in \Lambda$, the graded $A$-module $M^1/U^{2,i}$, which admits a finite $F-filtration$, is also finite dimensional for each $i\geq 0$.
            Therefore, we can fix $d_i\in\zz$ such that $(M^1/U^{2,i})_{\geq d_i} = 0$ for each $i\geq 0$.
            Since the $F$-filtration $\{U^{1,j} \}_{j=0}^\infty$ is separable, there exists $a_i \geq 0$ for each $i > 1$ such that $(U^{1,a_i})_d = 0$ for $d < d_i$, that is, $U^{1,a_i}=(U^{1,a_i})_{\geq d_i}$.
            Since we can replace $a_i$ with a number greater than $a_i$, and hence there is an strictly increasing sequence $a_1 < a_2 < a_3 < \cdots $ with the condition above.
            
            Let $p:M^0\rightarrow M^0/M^1$ be the natural quotient morphism.
            We construct a sequence of submodules of $M$ by
            \begin{equation}\label{eq:mixedfilt}\begin{array}{llll}
                M = M^0 
                        &\supseteq p^{-1}(U^{1,a_1})_{\geq d_1} + U^{2,0} &\supseteq p^{-1}(U^{1,a_1})_{\geq d_1} + U^{2,1}\\
                        &\supseteq p^{-1}(U^{1,a_2})_{\geq d_2} + U^{2,1} &\supseteq p^{-1}(U^{1,a_2})_{\geq d_2} + U^{2,2}\\
                        &&\vdots\\
                        &\supseteq p^{-1}(U^{1,a_j})_{\geq d_j} + U^{2,j-1} &\supseteq p^{-1}(U^{1,a_j})_{\geq d_j} + U^{2,j}\\
                        &\supseteq p^{-1}(U^{1,a_{j+1}})_{\geq d_{j+1}} + U^{2,j} &\supseteq p^{-1}(U^{1,a_{j+1}})_{\geq d_{j+1}} + U^{2,{j+1}}\\ 
                        &&\vdots&.
            \end{array}\end{equation}
   
            Given $d\in\zz$, we have $(p^{-1}(U^{1,a_j})_{\geq d_j})_d$ and $ (U^{2,j})_d$ are zero for sufficiently large $j$.
            Then, we have $\bigcap_{j=1}^\infty (p^{-1}(U^{1,a_j})_{\geq d_j} + U^{2,j}) = 0$; therefore, the sequence \eqref{eq:mixedfilt} is a separable filtration of $M$.
            We consider the odd-numbered factors of this filtration.
            The first factor is
            \[
                M^0/(p^{-1}(U^{1,a_1})_{\geq d_1} + U^{2,0}). 
            \] 
            Since $\ker p = M^1 = U^{2,0} \subset (p^{-1}(U^{1,a_1})_{\geq d_1} + U^{2,0})$, a morphism $p$ induces an isomorphism
            \begin{align}
                M^0/(p^{-1}(U^{1,a_1})_{\geq d_1} + U^{2,0}) &\simeq p(M^0)/p(p^{-1}(U^{1,a_1})_{\geq d_1} + U^{2,0}) \\
                            &=(M^0/M^1)/p(p^{-1}(U^{1,a_1})_{\geq d_1}).
            \end{align}
            We have $p(p^{-1}(U^{1,a_1})_{\geq d_1}) = (U^{1,a_1})_{\geq d_1}= U^{1,a_1}$; hence, $(M^0/M^1)/p(p^{-1}(U^{1,a_1})_{\geq d_1}) = U^{1,0}/U^{1,a_1}$.
            
            The other odd-numbered factors are given by
            \[ 
                (p^{-1}(U^{1,a_j})_{\geq d_j} + U^{2,j})/(p^{-1}(U^{1,a_{j+1}})_{\geq d_{j+1}} + U^{2,j}) \qquad(j\geq 1).
            \]
            Since we have $\ker p = M^1$ and $U^{2,j}_{\geq d_j} = M^1_{\geq d_j}$, it follows that
            \[
                (p^{-1}(U^{1,a_j})_{\geq d_j} + U^{2,j})\cap \ker p \subset U^{2,j}\subset (p^{-1}(U^{1,a_{j+1}})_{\geq d_{j+1}} + U^{2,j}).
            \] 
            Therefore, a surjection $p$ induces an isomorphism
            \begin{align}
                &(p^{-1}(U^{1,a_j})_{\geq d_j} + U^{2,j})/(p^{-1}(U^{1,a_{j+1}})_{\geq d_{j+1}} + U^{2,j})\\
                    &\simeq p(p^{-1}(U^{1,a_j})_{\geq d_j} + U^{2,j})/p(p^{-1}(U^{1,a_{j+1}})_{\geq d_{j+1}} + U^{2,j})\\
                    & = (U^{1,a_j})_{\geq d_j}/(U^{1,a_{j+1}})_{\geq d_{j+1}}\\
                    &= U^{1,a_j}/U^{1,a_{j+1}}.
            \end{align}
            These factors $U^{1,0}/U^{1,a_1},U^{1,a_1}/U^{1,a_2},U^{1,a_2}/U^{1,a_3},\ldots$ have finite $F$-filtrations
            \begin{align}
                &U^{1,0}/U^{1,a_1} \supseteq U^{1,1}/U^{1,a_1} \supseteq U^{1,2}/U^{1,a_1}\supseteq\\
                    &\qquad\qquad\qquad\cdots\supseteq U^{1,a_1-1}/U^{1,a_1} \supseteq U^{1,a_1}/U^{1,a_1} = 0
            \end{align}
            and 
            \begin{align}
                &U^{1,a_j}/U^{1,a_{j+1}} \supseteq U^{1,a_{j}+1}/U^{1,a_{j+1}} \supseteq U^{1,a_{j}+2}/U^{1,a_{j+1}}\supseteq\\
                    &\qquad\qquad\cdots\supseteq U^{1,a_{j+1}-1}/U^{1,a_{j+1}} \supseteq U^{1,a_{j+1}}/U^{1,a_{j+1}} = 0 \qquad\qquad(j\geq 1).    
            \end{align}
            Consequently, the factors of these $F$-filtrations coincide with the factors of $\{U^{1,k}\}_{k=0}^\infty$. 
            In particular, the sum of the multiplicities of factors is equal to $\sum_{\alpha \in \Lambda}f_{1,\alpha}\cdot[F_\alpha]$.
            
            
            Next we consider the even-numbered factors.
            For each $j\geq 1$, the $(2j)$-th factor is
            \begin{align}
                &(p^{-1}(U^{1,a_j})_{\geq d_j} + U^{2,j-1} )/( p^{-1}(U^{1,a_j})_{\geq d_j} + U^{2,j})\\  
                    &\simeq U^{2,j-1}/(U^{2,j-1}\cap (p^{-1}(U^{1,a_j})_{\geq d_j} + U^{2,j}) ).
            \end{align}
            Since $U^{2,j-1}_{\geq d_j} = U^{2,j}_{\geq d_j} = M^1_{\geq d_j}$, we have 
            \begin{align}
                U^{2,j-1}\cap (p^{-1}(U^{1,a_j})_{\geq d_j} + U^{2,j}) &= U^{2,j-1}\cap (p^{-1}(U^{1,a_j})_{\geq d_j} + U^{2,j})\\ 
                    &=U^{2,j-1}\cap U^{2,j}\\
                    &=U^{2,j}. 
            \end{align}
            Thus we have
            \[
                U^{2,j-1}/(U^{2,j-1}\cap (p^{-1}(U^{1,a_j})_{\geq d_j} + U^{2,j}) )\simeq U^{2,j-1}/U^{2,j};
            \]
            hence, the even-numbered factors of a filtration \eqref{eq:mixedfilt} coincide with the factors of a filtration $\{U^{2,j}\}_{j=0}^\infty$.
            Applying the result of the case \eqref{item:fin-filts}, a filtration \eqref{eq:mixedfilt} is a $F$-filtration with the multiplicity of factors $M\filt \sum_{\alpha\in\Lambda}f_{1,\alpha}\cdot[F_\alpha] + \sum_{\alpha\in\Lambda}f_{2,\alpha}\cdot[F_\alpha]$.
            This completes the proof of the case $M^2 = 0$.
            

            Next we consider the case where $M^2\neq 0$.
            Set $N^0 = M^0$ to inductively construct a $F$-filtration $\{N^i\}_{i=0}^\infty$ of $M$.
            Suppose that we have already defined $M= N^0\supset N^1 \supset N^2 \supset \cdots \supset N^l$ with the following properties:
            \begin{enumerate}[(a)]
                \item $N^i$ contains $M^i$ for $0\leq i \leq l$, and $(N^i/M^i)_d=0$ holds for $d<i$,
                \item for $0 < i \leq l$, a graded module $N^{i-1}/N^i$ admits a finite $F$-filtration,
                \item a graded module $N^i/M^i$ admits a $F$-filtration,
                \item the sum of the multiplicities of factors in the $F$-filtrations of $N^{i-1}/N^i$ and $N^i/M^i$ is equal to the sum of the ones of $N^{i-1}/M^{i-1}$ and $M^{i-1}/M^i$.
            \end{enumerate}

            Applying the result of the case $M^2 = 0$ to a filtration $N^l/M^{l+1} \supset M^l/M^{l+1} \supset M^{l+1}/M^{l+1}=0$, we obtain a $F$-filtration of $N^l/M^{l+1}$ such that its multiplicity of factors is equal to the sum of ones of the $F$-filtrations for $N^l/M^l$ and $M^l/M^{l+1}$.
            By taking the preimage under the quotient morphism $p_{l+1}:M\rightarrow M/M^{l+1}$, this filtration is lifted up to a sequence of graded modules which are submodules of $N^l$ and containing $M^{l+1}$.
            We choose a graded module $N^{l+1}$ from this sequence, or equivalently select a graded module from a $F$-filtration of $N^l/M^{l+1}$ as the image $p_{l+1}(N^{l+1})=N^{l+1}/M^{l+1}$.
            Since a $F$-filtration is separable, we can choose $N^{l+1}$ such that $(N^{l+1}/M^{l+1})_d=0$ for $d<l+1$ holds.
            By construction, $N^l/N^{l+1}$ and $N^{l+1}/M^{l+1}$ admit $F$-filtrations such that the sum of the multiplicities of factors in them is equal to the sum of the ones in the $F$-filtrations for $N^l/M^l$ and $M^l/M^{l+1}$.
            In addition, the $F$-filtration of $N^l/N^{l+1}$, which has the first finitely many factors of the $F$-filtration of $N^l/M^{l+1}$, is a finite filtration.
            Consequently, we can construct a filtration $\{N^i\}_{i=0}^\infty$ satisfying the conditions (a), (b), (c), and (d) for every $l\geq 1$ by repeating the same processes.
            Since $N^{i} \subset M_{\geq i}+M^{i}$ for each $i$, we have $\{N^i\}_{i=0}^\infty$ is separable.
            The $i$-th factor $N^{i-1}/N^i$ admits a finite $F$-filtration for each $i\geq 1$.
            Hence we can apply the result of the case \eqref{item:fin-filts} and obtain a $F$-filtration of $M$.
            
            Finally we check the multiplicity of factors of this filtration.
            Given $k\in\zz$ and $\alpha\in \Lambda$, we count the number of factors which is isomorphic to $F_\alpha\gs{k}$.
            We fix $d\in\zz$ to be $(F_\alpha\gs{k})_d\neq 0$.
            Since $\{N^i\}_{i=0}^\infty$ is separable, we can fix $i$ to be $(N^i)_d = 0$.
            Then $F_\alpha\gs{k}$ does not appear in $F$-filtrations of $N^i/N^{i+1}, N^{i+1}/N^{i+2}, N^{i+2}/N^{i+3}, \ldots $ because they are subquotients of $N^i$ and do not have the homogeneous part of degree $d$.
            Therefore, all the factors isomorphic to $F_\alpha\gs{k}$ in the (infinitely many) $F$-filtrations of $N^0/N^1, N^1/N^2, \ldots$ are contained in the filtrations of the first $i$ modules $N^0/N^1, N^1/N^2, \ldots, N^{i-1}/N^i$.
            For the same reason, $F$-filtrations of modules $M^i/M^{i+1}, M^{i+1}/M^{i+2}, M^{i+2}/M^{i+3}, \ldots $ have no factor $F_\alpha\gs{k}$, and the number of the factors isomorphic to $F_\alpha\gs{k}$ in the $F$-filtrations of $M^0/M^1, M^1/M^2, \ldots $ is equal to the one in the $F$-filtrations of the first $i$ modules $M^0/M^1, M^1/M^2,\ldots, M^{i-1}/M^i$.
            Using the condition (d) repeatedly, it follows that $F_\alpha\gs{k}$ appears in the $F$-filtrations of $N^0/M^1 = M^0/M^1, M^1/M^2,\ldots, M^{i-1}/M^i$ the same number of times as in the filtrations of $N^0/N^1, N^1/N^2, \ldots, N^{i-1}/N^i, N^i/M^i$.
            Since $(N^i/M^i)_d=0$, the $F$-filtration of $N^i/M^i$ does not have a factor which is isomorphic to $F_\alpha\gs{k}$.
            Consequently, a $F$-filtration of $M$ obtained by applying the result of the case \eqref{item:fin-filts} to $\{N^i\}_{i=0}^\infty$ has the required number of factors $F_\alpha\gs{k}$.
            This completes the proof.

        \end{proof}

        \begin{lem}
            \label{lem:repeatingfilt}
            Let $M$ be a finitely generated graded $A$-module.
            Let $n$ be a positive integer and let $f:M\gs{n}\rightarrow M$ be an injective graded $A$-module homomorphism.
            Then we have \[M\filt \frac{1}{1-q^n}\cdot[M/\im f].\]
        \end{lem}
        \begin{proof}
            For a positive integer $i$, we denote by $f^i$ the composition of homomorphisms
            \[
                \xymatrix{
                    M\gs{in} \ar[r]^f &M\gs{(i-1)n} \ar[r]^f &\cdots \ar[r]^f &M\gs{2n} \ar[r]^f &M\gs{n} \ar[r]^f &M.
                }  
            \]
            Then, $f^i:M\gs{in}\rightarrow M$ is an injective homomorphism.
            Since $M$ is bounded below, we have $\bigcap_{i=1}^\infty \im f^i = 0$.
            Thus $M$ has a filtration
            \[
                M \supset \im f^1 \supset \im f^2 \cdots \supset \im f^{i-1} \supset \im f^{i} \supset \cdots.    
            \]
            The restriction of $f^i$ gives an isomorphism between $(\im f)\gs{in}$ and $\im f^{i+1}$ for each $i>0$.
            Therefore, we obtain an isomorphism $(M/\im f)\gs{in}\simeq M\gs{in}/\im f\gs{in}\simeq \im f^i/\im f^{i+1}$.
            In other words, the $i$-th factor is isomorphic to $(M/\im f)\gs{in}$.
            Consequently, we have
            \[
                M\filt \sum_{i=0}^\infty q^{in} [M/\im f] = \frac{1}{1-q^n} \cdot [M/\im f].
            \]

        \end{proof}

    \section{Skew group algebras}
        \label{sec:skewgrpalg}
        Keep the setting of the previous section.
        Let $S = \cc[X_1,X_2,\ldots,X_n]$ be a polynomial ring in $n$ variables.
        Let $W$ be a finite group acting on $S$ from the left such that its action preserves the grading of $S$.
        We define the skew group algebra $(S*W)$ as the vector space $S \otimes_\cc \mathbb{C}W$ with its multiplication 
        \[(s\otimes w)(s' \otimes w') \coloneqq sw(s')\otimes ww'\] 
        for any $s,s'\in S$ and $w,w' \in W$.
        Since $S= \bigoplus_{i\in\mathbb{Z}}S_i$ is graded, \[S*W = \bigoplus_{i\in\mathbb{Z}} S_i\otimes \mathbb{C}W\] is also a graded algebra.
        The following is well known:
        \begin{prop}\label{prop:opiso}
            The opposite algebra $(S*W)^{op}$ of $S*W$ is also a graded algebra, and two graded algebras $S*W$ and $(S*W)^{op}$ are isomorphic.
        \end{prop}
        \begin{proof}
            There exists a graded algebras isomorphism $\theta :(S*W)^{op} \rightarrow S^{op}*W$ given by
            \[\theta :s\otimes w \mapsto w^{-1}(s)\otimes w^{-1} \qquad s\in S,\,w\in W\]
            (see e.g.~\cite{Auslander_Reiten_Smalo_1989, Marcos2003HochschildCO}).
            Since $S$ is commutative, the graded algebras $S$ and $S^{op}$ are isomorphic.
            This completes the proof.
        \end{proof}

        We denote by $\irr{W}$ the set of isomorphism classes of irreducible $W$-representations over $\cc$.
        Let $V$ be a $W$-representation and let $M$ be a graded $(S*W)$-module.
        Then $M\otimes_\cc V$ acquires a $(S*W)$-action given by 
        \[(s\otimes w)(m\otimes v)\coloneqq (s\otimes w)m\otimes w(v) \]
        for each $s\in S, w\in W, m\in M$ and $v\in V$.
        We also define a $(S*W)$-module $P_V$ by $P_V \coloneqq S\otimes_\cc V$.
        The module $P_V$ acquires a graded structure by $(P_V)_i = S_i\otimes V \,(i\in\zz).$
        Since we have a direct sum decomposition $V \simeq \bigoplus_{\lambda\in\irr{W}}\lambda^{\oplus n_\lambda}$ by Maschke's theorem (see~\cite[Theorem 6.1]{Lam2001}), we find that $P_V$ is decomposed as $V\simeq\bigoplus_{\lambda\in\irr{W}}P_\lambda^{\oplus n_\lambda}$ as $S*W$-modules.
        In paticular, $S*W\simeq P_{\cc W}$ is also decomposed into 
        \[S*W \simeq \bigoplus_{\lambda\in\irr{W}}P_\lambda^{\oplus\dim_\cc \lambda}.\]
        Therefore, $P_\lambda \,(\lambda\in\irr{W})$ are (graded) projective $S*W$-modules and $P_V$ is also projective for every $W$-representation $V$.
        Note that, for a basis $v_1,v_2,\ldots,v_m$ of $V$ over $\cc$, the elements $1\otimes v_1,1\otimes v_2,\ldots, 1\otimes v_m$ form a basis of free $S$-module $P_V$.

        Let $V$ be a $W$-representation.
        We define the graded $(S*W)$-module $L_V$ by 
        \[
            (L_V)_i \coloneqq 
            \left\{ \begin{array}{ll}   V    &(i=0)   \\
                                        0          &(i\neq0)\end{array}\right.
        \]
        as a graded $W$-module such that $S_{>0}$ acts as zero.
        Since $\lambda\in\irr{W}$ is a simple $(S*W)_0$-module, $L_\lambda$ is obviously a simple $S*W$-module.
        \begin{lem}
            \label{lem:gsimple}
            The set $\{L_\lambda\gs{m}\}_{\lambda \in \irr{W}, m \in \zz}$ is the complete collection of isomorphism classes of simple graded $(S * W)$-modules.
        \end{lem} 
        \begin{proof}
            It is well known (see e.g.~\cite[Proposition 2.7.1, 5.2.4]{Oystaeyen_Nastasescu}).
        \end{proof}

        For each $M\in (S*W)$-gmod and $i\in\zz$, \,$M_i$ is a $W$-representation.
        We define 
        \[ [M:L_\lambda]_q\coloneqq \sum_{i\in \zz} q^i[M_i:\lambda] \]
        where $[M_i:\lambda]$ is the multiplicity of $\lambda$ in $M_i$ as $W$-representations.
        Since $M$ is bounded below and locally finite, then $[M:L_\lambda]_q$ is a formal Laurent series.
        Clearly, we have
        \[\gdim M = \sum_{\lambda\in\irr{W}} [M:L_\lambda]_q\dim_\cc \lambda.\]
        We also define 
        \begin{equation}\label{eq:gch}\gch M \coloneqq \sum_{\lambda\in\irr{W}} [M:L_\lambda]_q\cdot[L_\lambda] \qquad\in\bigoplus_{\lambda \in \irr{W}} \zz(\!(q)\!)\cdot[L_\lambda].\end{equation}

        Since $M$ is bounded below, there is an integer $d\in \zz$ such that $M=M_{\geq d}$.
        Then $\{M_{\geq i}\}_{i=d}^\infty$ is a separable filtration of $M$.
        Its factor $M_{\geq i}/M_{\geq i+1}$ is semisimple for each $i\geq d$, and hence $M$ has $\{L_\lambda\}_{\lambda\in\irr{W}}$-filtration by Lemma \ref{lem:cnctfilts}.  
        In particular, we have
        \[
            M\filt \sum_{\lambda\in\irr{W}} [M:L_\lambda]_q\cdot[L_\lambda] = \gch M. 
        \]

        \begin{prop}
            \label{prop:ghom-multi}
            For $M\in (S*W)\hgmod$ and $\lambda \in \irr{W}$, there is an isomorphism between graded vector spaces 
            \[
                \Hom_{\cc W}(L_\lambda,M) \simeq    \Hom_{S*W}(P_\lambda,M).
            \]
        In particular, we have $\gdim \Hom_{S*W}(P_\lambda,M) = [M:L_\lambda]_q$.
        \end{prop}
        \begin{proof}
            We have a graded $(S*W)$-module isomorphism
            \[
                (S*W)\otimes_{\cc W}L_\lambda \simeq S\otimes_{\cc} \cc W\otimes_{\cc W}L_\lambda \simeq S\otimes_{\cc} L_\lambda \simeq P_\lambda. 
            \]
            By the Frobenius reciprocity (see e.g.~\cite[Corollary 2.4.10]{Oystaeyen_Nastasescu}), we obtain an isomorphism between graded vector spaces
            \[
                \Hom_{\cc W}(L_\lambda, M)\simeq \Hom_{S*W}((S*W)\otimes_{\cc W} L_\lambda, M)=\Hom_{S*W}(P_\lambda, M).
            \]
            By Schur's lemma, we have $\gdim \Hom_{\cc W}(L_\lambda,M)=[M:L_\lambda]_q$.
            This completes the proof.
        \end{proof}

        Let $M\in (S*W)$-gmod such that $\dim_\cc M<\infty$.
        We define the dual $M^* \coloneqq \Hom_\cc (M,\cc)$ as graded right $(S*W)$-module.
        By Proposition \ref{prop:opiso}, $M^*$ admits a left $(S*W)$-module structure.
        More precisely, $s\otimes w \in S*W$ acts on $M^*$ from the left as 
        \[((s\otimes w)f)(m) = f((w^{-1}(s)\otimes w^{-1})m)\] for $f \in M^*$ and $m\in M$.
        Note that $M^*$ is graded as $(M^*)_i = \Hom_\cc(M_{-i},\cc)$.
        
        
        The following is well known:
        \begin{prop}
            \label{prop:gldim-skewgrpalg}
            The skew group algebra $S*W$ has global dimension $n$.
        \end{prop}
        \begin{proof}
            A polynomial ring $S$ has global dimension $n$ (see~\cite[Theorem 8.37]{Rotman2009}).
            By \cite{li_2017}, the skew group algebra $(S*W)$ has the same global dimension.
        \end{proof}
        We set $A=S*W$. 
        \begin{dfn}
            Let $M,N$ be a finitely generated graded $A$-modules.
            The graded Euler-Poincar\'e pairing of $M$ and $N$ is
            \[
                \gep{M}{N}\coloneqq \sum_{d\in\zz} q^d \gep{M}{N}_d ,
            \]
            where
            \[
                \gep{M}{N}_d\coloneqq \sum_{i=0}^n (-1)^i\dim_\cc \Ext_A^i(M,N)_d \quad\in\zz
            \]
            for each $d \in \zz$.
        \end{dfn}
        By definition, we have 
        \[\gep{M}{N} = \sum_{i=0}^n (-1)^i \gdim \Ext_A^i(M,N)\]
        for $M,N\in A$-gmod.
        Since $\Ext_A^i(M,N)$ is locally finite and bounded below for $i\geq 0$, then $\gep{M}{N}$ is a formal Laurent series.   

        \begin{lem}
            \label{lem:sumofgep}
            Let $V$ be a graded $A$-module.
            If there is a short exact sequence in $A\textrm{-gmod}$
            \[
                \xymatrix{ 0 \ar[r] &L \ar[r] &M \ar[r] & N \ar[r] &0}.
            \]
            ,then we have 
            \[
                \gep{V}{L} + \gep{V}{N} = \gep{V}{M}
            \]
            and
            \[
                \gep{L}{V} + \gep{N}{V} = \gep{M}{V}.
            \].
        \end{lem}
        \begin{proof}

            Applying $\Hom_{A}(V,-)$ to a given short exact sequence, we get a long exact sequence
            \[
                \xymatrix{
                    0   \ar[r] &\Hom_A(V,L) \ar[r] &\Hom_A(V,M) \ar[r] &\Hom_A(V,N)\\
                    \ar[r] &\Ext_A^1(V,L) \ar[r] &\Ext_A^1(V,M) \ar[r] &\Ext_A^1(V,N)\\
                    \ar[r] &\Ext_A^2(V,L) \ar[r] &\Ext_A^2(V,M) \ar[r] &\Ext_A^2(V,N) \ar[r] &0.\\
                }
            \](see~\cite[Theorem 6.43]{Rotman2009}).

            Therefore, we obtain
            \begin{multline*}
                \gdim\Hom_A(V,L) - \gdim\Hom_A(V,M) + \gdim\Hom_A(V,N)\\
                -\gdim\Ext_A^1(V,L) + \gdim \Ext_A^1(V,M) - \gdim \Ext_A^1(V,N)\\
                +\gdim\Ext_A^2(V,L) - \gdim\Ext_A^2(V,M)  + \gdim \Ext_A^2(V,N) = 0.   
            \end{multline*}
            Hence, it follows that
            \begin{align*}
                \gep{V}{L} + \gep{V}{N} &= \gdim\Hom_A(V,L) + \gdim\Hom_A(V,N)\\
                &-\gdim\Ext_A^1(V,L) - \gdim \Ext_A^1(V,N)\\
                &+\gdim\Ext_A^2(V,L) + \gdim\Ext_A^2(V,N)\\
                &= \gdim\Hom_A(V,M)- \gdim \Ext_A^1(V,M)+ \gdim\Ext_A^2(V,M)\\
                &= \gep{V}{M}.
            \end{align*}
            The second equation is shown in a similar fashion.

        \end{proof}

        \begin{cor}
            \label{cor:epdectosimple}
            For finite dimensional graded $A$-modules $M,N$, we have
            \[
                \gep{M}{N} = \sum_{\lambda,\mu\in \irr{W}} \gep{L_\lambda}{L_\mu}\cdot[M^*:L_\lambda]_{q}\cdot[N:L_\mu]_q.
            \]
            
        \end{cor}
        \begin{proof}
           Since $M$ is a finite dimensional graded $A$-module, $M$ admits a finite $\{L_\lambda\}_{\lambda\in\irr{W}}$-filtration
           \[M = M^0\supsetneq M^1\supsetneq M^2 \supsetneq\cdots \supsetneq M^m = 0.\]
           For each $i\hspace{1.2mm}(1\leq i\leq m)$, we have a short exact sequence
           \[\xymatrix{0\ar[r] & M^i \ar[r] &M^{i-1} \ar[r] &M^{i-1}/M^i \ar[r] &0. }\]
           By using Lemma \ref{lem:sumofgep} repeatedly, we have
            \begin{align} 
                \gep{M}{N}  &= \gep{M^0}{N} \\
                            &= \gep{M^0/M^1}{N} + \gep{M^1}{N}\\ 
                            &\vdots\\
                            &=\sum_{i=1}^{m} \gep{M^{i-1}/M^i}{N}.
            \end{align}
           Clearly, the number of factors isomorphic to $L_\lambda\gs{d}$ is equal to $[M_d:\lambda]$ for each $\lambda \in \irr{W}$.
           Therefore, we have
           \begin{align}
                \gep{M}{N} &= \sum_{\lambda\in\irr{W}}\sum_{d\in\zz} [M_d:\lambda]\cdot\gep{L_\lambda\gs{d}}{N}\\
                        &= \sum_{\lambda\in\irr{W}}\sum_{d\in\zz} [M_d:\lambda]\cdot q^{-d}\,\gep{L_\lambda}{N}\\
                        &= \sum_{\lambda\in\irr{W}}[M^*:L_\lambda]_{q}\cdot\gep{L_\lambda}{N}.
        \end{align}
           Similarly, we can decompose $N$ to simple modules and obtain that
           \begin{align}
            \gep{M}{N} &= \sum_{\mu\in\irr{W}}\sum_{d\in\zz} [N_d:\mu]\cdot \gep{M}{L_\mu\gs{d}}\\
                        &= \sum_{\mu\in\irr{W}}\sum_{d\in\zz} [N_d:\mu]\cdot q^d\,\gep{M}{L_\mu}\\
                        &= \sum_{\mu\in\irr{W}} [N:L_\mu]_q\cdot\gep{M}{L_\mu}.\\
           \end{align}
            Consequently, we have
            \[
                \gep{M}{N} = \sum_{\lambda\in\irr{W}}\sum_{\mu\in\irr{W}}[M^*:L_\lambda]_{q}\cdot[N:L_\mu]_q \cdot \gep{L_\lambda}{L_\mu} 
            \]
            as desired.                       
        \end{proof}


    \section{Brauer-Humphreys type reciprocity for preorders}\label{sec:reciprocity}
        Keep the setting of the previous section.
        Let $S = \cc[X_1,X_2,\ldots,X_n]$ be a polynomial ring in $n$ variables. 
        Let $W$ be a finite group such that every finite dimensional $W$-representation $V$ is self-dual, that is, $V\simeq V^*$.
        Note that a representation $V$ is self-dual if and only if the character of $V$ is real-valued (see e.g.~\cite[13.2]{Serre1977}).
        Fix a $W$-action on $S$ which preserves the grading of $S$.
        We set $A=S*W$ and $R = S^W$.
        We assume that $S/I$ is finite dimensional, where $I=S\cdot R_{>0}$ is the ideal of $A$ generated by the homogeneous invariants of positive degree.

        \begin{prop}
            \label{prop:adj-sym}
            Let $V$ be a $W$-representation.
            For $\lambda,\mu\in\irr{W}$, we have $[V\otimes\lambda:\mu] = [V\otimes\mu:\lambda]$.
        \end{prop}
        \begin{proof}
            By Schur's Lemma, we have 
            \[[V\otimes\lambda:\mu] = \dim \Hom_{\cc W}(\mu,V\otimes\lambda) = \dim \Hom_{\cc}(\mu,V\otimes\lambda)^W.\]
            There is a natural isomorphism $\Hom_\cc(\mu,V\otimes\lambda)\simeq \mu^*\otimes V\otimes\lambda \simeq \Hom_\cc(\lambda^*,V\otimes\mu^*)$, so we obtain $\dim \Hom_\cc(V\otimes\lambda,\mu)^W = \dim \Hom_\cc(V\otimes\mu^*,\lambda^*)^W = [V\otimes\mu^*:\lambda^*]$.
            Since $\lambda$ and $\mu$ are self-dual, we have $ [V\otimes\mu^*:\lambda^*] =  [V\otimes\mu:\lambda]$.
            This completes the proof.

        \suspend
        \end{proof}
        \begin{cor}
            \label{cor:proj-sym}
            For $\lambda,\mu\in\irr{W}$, we have\[ [P_\lambda:L_\mu]_q = [P_\mu:L_\lambda]_q \].
        \end{cor}
        \begin{proof}
            For each $d\geq 0$, the degree $d$ parts of $P_\lambda$ and $P_\mu$ are $(P_\lambda)_d = S_d\otimes\lambda$ and $(P_\mu)_d = S_d\otimes\mu$ respectively.
            By Proposition \ref{prop:adj-sym}, we obtain
            \begin{align}
                [P_\lambda:L_\mu]_q &=\sum_{d=0}^\infty [(P_\lambda)_d:\mu]\cdot q^d=\sum_{d=0}^\infty [(P_\mu)_d:\lambda]\cdot q^d=[P_\mu:L_\lambda]_q. 
            \end{align} 
        \end{proof}

        Let $\precsim$ be a preorder on $\irr{W}$.
        We define 
        \[\begin{aligned}
            \lambda \sim \mu \text{ if and only if } \lambda \precsim \mu \text{ and } \mu \precsim \lambda\\
            \lambda \prec \mu \text{ if and only if } \lambda \precsim \mu \text{ and } \lambda \nsim \mu.
        \end{aligned}\]
        for each $\lambda,\mu\in\irr{W}$.
        Then $\sim$ is an equivalence relation and $(\irr{W}/\!\sim ,\precsim)$ is partially ordered set.
        We define graded $(S*W)$-modules $K_\lambda$ and $\widetilde{K}_\lambda$ depending on a preorder $\precsim$ as follows:
        
        \begin{dfn}\label{dfn:K-Ktilde}
            Let $\precsim$ be a preorder on $\irr{W}$.
            For each $\lambda \in \irr{W}$, we define submodules of $P_\lambda$ by
            \[
                I_\lambda \coloneqq \sum_{\mu\succsim\lambda,f\in\Hom_{A}(P_\mu,P_\lambda)_{>0}}\mathrm{Im}\,f
                \quad\textrm{and}\quad
                \widetilde{I}_\lambda\coloneqq \sum_{\mu\succ\lambda,f\in\Hom_{A}(P_\mu,P_\lambda)_{>0}}\mathrm{Im}\,f.
            \]
            We also define quotient modules of $P_\lambda$ by 
            \[
                K_\lambda \coloneqq P_\lambda /I_\lambda  
                \quad\textrm{and}\quad
                \widetilde{K}_\lambda \coloneqq P_\lambda / \widetilde{I}_\lambda.
            \]
        \end{dfn}
            We call $\{K_\lambda\}_{\lambda\in\irr{W}}$ (resp. $\{\widetilde{K}_\lambda\}_{\lambda\in\irr{W}}$) the trace quotient modules (resp. the strict trace quotient modules) with respect to the preorder $\precsim$.

        We remark that, by Proposition \ref{prop:ghom-multi}, the $S*W$-submodule $\widetilde{I}_\lambda$ is generated by the isotypic components of $(P_\lambda)_{>0}$ associated to $\mu$ with $\mu\succ \lambda$.
        Similarly, $I_\lambda$ is generated by the isotypic components associated to $\mu$ with $\mu\succsim \lambda$.
         
        Since $\lambda \precsim \lambda$ and $R\subseteq \End_{S*W}(P_\lambda)$, we have
        \[\begin{aligned}
            \dim K_\lambda  &\leq \dim P_\lambda/(R_{>0}\cdot P_\lambda)\\
                            &=\dim (S/R_{>0}\cdot S)\otimes_\cc \lambda.\\
        \end{aligned}\]   
        By assumption, we have $\dim S/(R_{>0}\cdot S)<\infty$.
        Therefore, a trace quotient module $K_\lambda$ is finite dimensional and its dual module $K_\lambda^*$ is well-defined object in $(S*W)$-gmod.
        We denote the sets $\{\widetilde{K}_\lambda\}_{\lambda\in\irr{W}}$ and $\{K_\lambda\}_{\lambda\in\irr{W}}$ by $\widetilde{K}$ and $K$ respectively.
        By Proposition \ref{prop:ghom-multi}, we have
        \[\dim \Hom_{A\text{-gmod}} (P_\lambda\gs{d},L_\mu) = \begin{cases}1 &  (\lambda = \mu \text{ and } d=0),\\0 &(otherwise).\end{cases}\]
        Since we have $\widetilde{I}_\lambda,I_\lambda\subset (P_\lambda)_{>0}$ and $(L_\mu)_{>0} = 0$, we obtain
        \[\dim \Hom_{A\text{-gmod}} (\widetilde{K}_\lambda\gs{d},L_\mu)= \begin{cases}1 & (\lambda = \mu \text{ and } d=0),\\0 &(otherwise)\end{cases}\]
        and
        \[\dim \Hom_{A\text{-gmod}} (K_\lambda\gs{d},L_\mu)= \begin{cases}1 &(\lambda = \mu \text{ and } d=0),\\0 &(otherwise).\end{cases}\]
        
        We consider the case $\widetilde{K}$ and $K$ satisfy some conditions.
        \begin{dfn}
            \label{dfn:BHR-conditions}
            Let $\precsim$ be a preorder on $\irr{W}$.
            Let $\{\widetilde{K}_\lambda\}_{\lambda\in\irr{W}}$ and $\{K_\lambda\}_{\lambda\in\irr{W}}$ be the strict trace quotient modules and the trace quotient modules with respect to $\precsim$ respectively.
            We define following conditions:
            \begin{itemize}
                \item \text{(The orthogonality)} For each $\lambda,\mu\in\irr{W}$ , we have
                    \begin{equation}\label{eq:orthogonality}\Ext_{A}^i(\widetilde{K}_\lambda,K_\mu^*)_d = 
                        \begin{cases}
                            \cc & (i=d=0 \textrm{ and } \lambda = \mu),\\
                            0   & (otherwise).
                        \end{cases}
                    \end{equation}
                \item (The existence of $\widetilde{K}$-filtration) For each $\lambda\in\irr{W}$, the projective module $P_\lambda$ is filtered by the strict trace quotient modules $\{\widetilde{K}_\mu\}_{\mu\in\irr{W}}$. 
                \item (The existence of $K$-filtration) For each $\lambda\in\irr{W}$, the module $\widetilde{K}_\lambda$ is filtered by the trace quotient modules $\{K_\mu\}_{\mu\in\irr{W}}$.
            \end{itemize}
        \end{dfn}
        Assume that a projective module $P_\lambda$ is filtered by $\widetilde{K}$.
        We denote by $(P_\lambda:\widetilde{K}_\mu\gs{d})$ the number of factors of a $\widetilde{K}$-filtration of $P_\lambda$ isomorphic to $\widetilde{K}_\mu\gs{d}$ in $(S*W)$-gmod, and we define the graded version of multiplicity of $\widetilde{K}_\mu$ by
        \[
            (P_\lambda:\widetilde{K}_\mu)_q \coloneqq \sum_{d\in\zz} q^d(P_\lambda:\widetilde{K}_\mu\gs{d}).
        \]    
        
        If the orthogonality condition \eqref{eq:orthogonality} holds, then we can count the number of graded factors in a $\widetilde{K}$-filtration as follows:
        \begin{prop}\label{prop:reciprocity}
            Let $M\in A$-gmod be filtered by $\widetilde{K}$.
            If the condition \eqref{eq:orthogonality} in Definition \ref{dfn:BHR-conditions} holds, then every $\widetilde{K}$-filtration of $M$ has the same multiplicities of factors:
            \begin{align*}
                M\filt \sum_{\mu\in\irr{W}} \sum_{d\in\zz} q^d\dim\Hom_{A}(M,K_\mu^*)_{-d}\cdot[\widetilde{K}_\mu].
            \end{align*}
            In particular, for each $\lambda,\mu\in\irr{W}$ such that $P_\lambda$ is filtered by $\widetilde{K}$, we have
            \begin{align*}
                (P_\lambda:\widetilde{K}_\mu)_q = [K_\mu:L_\lambda]_q.
            \end{align*}
        \end{prop}

        \begin{proof}
            Given a $\widetilde{K}$-filtration of $M$ 
            \[
                M = M^0 \supseteq M^1 \supseteq M^2 \supseteq \cdots \supseteq M^i \supseteq \cdots,
            \]
            it suffices to show that the number of factors isomorphic to $\widetilde{K}_\mu\gs{d}$ is equal to $\dim\Hom_{A}(M,K_\mu^*)_{-d}$ for each $d\in\zz$ and $\mu\in\irr{W}$.
            Since $M$ is locally finite and bounded below, there is a number $l$ such that $M^l$ have no homogeneous part of degree less than or equal to $d$; that is, $M^l=(M^l)_{>d}$.
            Clearly, the degree $d$ part of $\widetilde{K}_\mu\gs{d}$ is nonzero, and hence $i$-th factor $M^{i-1}/M^i$ is not isomorphic to $\widetilde{K}_\mu\gs{d}$ for any $i>l$.
            In other word, every factor which isomorphic to $\widetilde{K}_\mu\gs{d}$ appears in the first $l$ factors.
            Then, a $\widetilde{K}$-filtration  $\{M^i\}_{i=0}^\infty$ of $M$ and a finite $\widetilde{K}-$filtration $\{M^i/M^l\}_{i=0}^l$ of $M/M_l$ have the same number of factors isomorphic to $\widetilde{K}_\mu\gs{d}$.
            Since the homogeneous part of $K_\mu^*\gs{d}$ of degree greater than $d$ is zero, then we have $\Hom_{A\text{-gmod}}(M^l,K_\mu^*\gs{d}) \simeq \Hom_{A}(M^l,K_\mu^*)_{-d} = 0$. 
             
            For each $i\geq 1$, we have a short exact sequence
            \[\xymatrix{ 0 \ar[r] & M^{i-1}/M^i  \ar[r] & M/M^i \ar[r] &M/M^{i-1} \ar[r] &0.}\]
            This sequence induces a long exact sequence
            \begin{align}\label{eq:exseq}
                \vcenter{\xymatrix@R=10pt{
                 0 \ar[r] &\Hom_{A}(M/M^{i-1},K_\mu^*)_{-d} \ar[r] &\Hom_{A}(M/M^i ,K_\mu^*)_{-d}\\
                \ar[r] &\Hom_{A}(M^{i-1}/M^i,K_\mu^*)_{-d} \ar[r]   &\Ext^1_{A}(M/M^{i-1},K_\mu^*)_{-d} \\
                \ar[r] &\Ext^1_{A}(M/M^i ,K_\mu^*)_{-d} \ar[r]      &\Ext^1_{A}(M^{i-1}/M^i,K_\mu^*)_{-d} .}}
            \end{align}
            Since $M^{i-1}/M^i$ is isomorphic to zero or $\widetilde{K}_{\alpha_i}\gs{d_i}$ for some $\alpha_i\in\irr{W}$ and $d_i\in\zz$, we have 
            \[\Ext^1_{A}(M^{i-1}/M^i,K_\mu^*)_{-d} = 0\] 
            by \eqref{eq:orthogonality}.
            By \eqref{eq:exseq}, we have 
            \begin{align}
                \dim \Ext^1_A(M/M^i,K_\mu^*)_{-d} &\leq \dim \Ext^1_A(M/M^{i-1},K_\mu^*)_{-d} + \dim\Ext^1_{A}(M^{i-1}/M^i,K_\mu^*)_{-d}\\
                    &= \dim\Ext^1_A(M/M^{i-1},K_\mu^*)_{-d}.
            \end{align}
            From this, we obtain 
            \[\dim \Ext^1_A(M/M^i,K_\mu^*)_{-d} \leq \dim \Ext^1_A(M/M^0,K_\mu^*)_{-d} = 0\]
            by induction on $i$.
            By \eqref{eq:orthogonality}, we have $\dim_\cc \Hom_{A}(M^{i-1}/M^i,K_\mu^*)_{-d} = \delta_{\alpha_i,\mu}\delta_{d_i,d}$.
            Therefore, \eqref{eq:exseq} yields 
            \begin{align}
                \dim_\cc &\Hom_{A}(M/M^i F,K_\mu^*)_{-d} \\
                    &= \dim_\cc \Hom_{A}(M/M^{i-1},K_\mu^*)_{-d} + \dim_\cc \Hom_{A}(M^{i-1}/M^i,K_\mu^*)_{-d}\\
                    &=\dim_\cc \Hom_{A}(M/M^{i-1},K_\mu^*)_{-d}+\delta_{\alpha_i,\mu}\delta_{d_i,d}.
            \end{align}
            Hence, we can show inductively that $\dim_\cc \Hom_{A}(M/M^i,K_\mu^*)_{-d}$ is equal to the number of indices $j$ such that $1\leq j \leq i$ and $j$-th factor $M^j/M^{j-1}$ is isomorphic to $\widetilde{K}_{\mu}\gs{d}$.
            In particular, $\dim_\cc \Hom_{A}(M/M^l,K_\mu^*)_{-d}$ is equal to the number of factors in a filtration $\{M^i/M^l\}_{i=0}^l$ isomorphic to $\widetilde{K}_\mu\gs{d}$.
            Applying $\Hom_{A}(-,K_\mu^*)_{-d}$ to the following short exact sequence: 
            \[\xymatrix{ 0 \ar[r] & M^l  \ar[r] & M \ar[r] &M/M^l \ar[r] &0,}\]
            we obtain a long exact sequence
            \[\xymatrix{ 0 \ar[r] & \Hom_A(M/M^l,K_\mu^*)_{-d}  \ar[r] & \Hom_A(M,K_\mu^*)_{-d} \ar[r] & \Hom_A(M^l,K_\mu^*)_{-d}=0.}\]
            Consequently, we have $\dim_\cc \Hom_{A}(M/M^l,K_\mu^*)_{-d} = \dim \Hom_A(M,K_\mu^*)_{-d}$.
            This completes the proof of the first assertion.

            In the case $M=P_\lambda$, we obtain
            \[
                (P_\lambda:\widetilde{K}_\mu)_q = \sum_{d\in\zz} q^d\dim\Hom_{A}(P_\lambda,K_\mu^*)_{-d}.
            \]
            By Proposition \ref{prop:ghom-multi}, we have
            \[
                \sum_{d\in\zz} q^d\dim\Hom_{A}(P_\lambda,K_\mu^*)_{-d} = \sum_{d\in\zz} q^d [(K_\mu^*)_{-d}:\lambda].\\
            \]
            Since every $W$-representation is self-dual, we have
            \begin{align}
                \sum_{d\in\zz} q^d [(K_\mu^*)_{-d}:\lambda] &= \sum_{d\in\zz} q^d [(K_\mu)_d:\lambda]= [K_\mu:L_\lambda]_q.
            \end{align}
            Consequently, we obtain 
            \[(P_\lambda:\widetilde{K}_\mu)_q = [K_\mu:L_\lambda]_q\]
            as desired.
        \end{proof}
        Thanks to Proposition \ref{prop:reciprocity}, the multiplicity $(P_\lambda:\widetilde{K}_\mu)_q$ is independent of the choice of the $\widetilde{K}$-filtration of $P_\lambda$.
        An equation of multiplicities $(P_\lambda:\widetilde{K}_\mu)_q = [K_\mu:L_\lambda]_q$ of this type is found in many settings and called Brauer-Humphreys type reciprocity (see e.g.~\cite{Cline1988,HOLMES1991117,IRVING1990363}).
        We will show later that the converse of Proposition \ref{prop:reciprocity} also holds under proper conditions.
        
        Let $\mathrm{Mat} (\irr{W},\zz[\![q]\!])$ be the ring of square matrices with entries in $\zz[\![q]\!]$ whose rows and columns are indexed by the finite set $\irr{W}$.
        We define matrices in $\mathrm{Mat} (\irr{W},\zz[\![q]\!])$ as follows:
        \[
            [P:L] \coloneqq \biggl([P_\lambda:L_\mu]_q\biggr)_{\lambda,\mu\in\irr{W}},\quad[\widetilde{K}:L] \coloneqq \biggl([\widetilde{K}_\lambda:L_\mu]_q\biggr)_{\lambda,\mu\in\irr{W}},
        \]
        \[
            [K:L] \coloneqq \biggl([K_\lambda:L_\mu]_q\biggr)_{\lambda,\mu\in\irr{W}},\quad \textrm{and}\quad (P:\widetilde{K}) \coloneqq \biggl((P_\lambda:\widetilde{K}_\mu)_q\biggr)_{\lambda,\mu\in\irr{W}}.
        \] 
        \begin{dfn}\label{dfn:Ktilde-K-mult}
            Assume that the three conditions in Definition \ref{dfn:BHR-conditions} are satisfied.
            For each $\lambda,\mu\in\irr{W}$, we fix a $K$-filtration of $\widetilde{K}_\lambda$ and denote the multiplicity of $K_\mu$ in a $K$-filtration of $\widetilde{K}_\lambda$ by $(\widetilde{K}_\lambda:K_\mu)_q$.
            We define a matrix $(\widetilde{K}:K)$ by
            \[
                (\widetilde{K}:K)\coloneqq \biggl( (\widetilde{K}_\lambda:K_\mu)_q \biggr)_{\lambda,\mu\in\irr{W}}.
            \]
        \end{dfn}
        \begin{prop}\label{prop:Ktilde-K-indep}
            The matrix $(\widetilde{K}:K)$ is independent of the choice of $K$-filtration.
        \end{prop}
        \begin{proof}
        We have a following equation over $\mathrm{Mat}(\irr{W},\zz[\![q]\!])$:
        \begin{align}
            [P:L]   &= (P:\widetilde{K}) \cdot [\widetilde{K}:L] \\
                    &= (P:\widetilde{K}) \cdot (\widetilde{K}:K) \cdot [K:L].
        \end{align}
        By Proposition \ref{prop:reciprocity}, we have 
        \[
            (P:\widetilde{K}) = [K:L]^t.
        \]
        Hence, we have
        \begin{align}
           \label{eq:dec-gch}
           [P:L] = [K:L]^t \cdot (\widetilde{K}:K) \cdot [K:L].
        \end{align}
        Since $(P_\lambda)_0\simeq \lambda$, we have 
        \begin{align}
            [P_\lambda:L_\mu ]_q \in
            \begin{cases}
                1+q\zz[\![q]\!] &(\lambda=\mu)\\
                q\zz[\![q]\!] &(\lambda\neq\mu)
            \end{cases}
        \end{align}
        for $\lambda,\mu\in\irr{W}$.
        Therefore, $\det [P:L]\in 1+q\zz[\![q]\!] \subset \zz[\![q]\!]^\times $.
        It follows that all matrices appearing in \eqref{eq:dec-gch} are invertible.
        Hence, we obtain 
        \[
            (\widetilde{K}:K) = ([K:L]^t)^{-1} \cdot [P:L] \cdot [K:L]^{-1}.
        \]
        Since the right hand side of this equality is independent of the choice of $K$-filtrations, then $(\widetilde{K}:K)$ is also independent.
    \end{proof}
    \begin{cor}
        \label{cor:Ktilde-K-sym}
        The matrix $(\widetilde{K}:K)$ is symmetric.
    \end{cor}
    \begin{proof}
        We have already seen in the proof of Proposition \ref{prop:Ktilde-K-indep} that
        \[
            (\widetilde{K}:K) = ([K:L]^t)^{-1} \cdot [P:L] \cdot [K:L]^{-1}.
        \]
        Then, we obtain
        \begin{align}
            (\widetilde{K}:K)^t = ([K:L]^{-1})^t \cdot [P:L]^t \cdot [K:L]^{-1}.
        \end{align}
        By Corollary \ref{cor:proj-sym}, the matrix $[P:L]$ is symmetric.
        Therefore, we obtain
        \[
            (\widetilde{K}:K)^t = ([K:L]^{-1})^t \cdot [P:L] \cdot [K:L]^{-1} = (\widetilde{K}:K).
        \]
        This completes the proof.
        \end{proof}

        We continue to assume that the three conditions in Definition \ref{dfn:BHR-conditions} are satisfied.
        Since $K$ be a family of finite dimensional graded modules, there exists a $K$-filtration of $P_\lambda$ with the multiplicity of factors
        \[
            P_\lambda\filt \sum_{\mu,\nu \in\irr{W}} (P_\lambda:\widetilde{K}_\nu )_q (\widetilde{K}_\nu:K_\mu)_q\cdot[K_\mu]  
        \]
        by Lemma \ref{lem:cnctfilts}. 
        We show that every $K$-filtration of $P_\lambda$ has the same multiplicity of factors.
        We fix a $K$-filtration of $P_\lambda$ for each $\lambda\in\irr{W}$ and denote the multiplicity of $K_\mu$ in filtration of $P_\lambda$ by $(P_\lambda:K_\mu)_q$.
        We define the matrix $(P:K)$ by 
        \[
            (P:K) \coloneqq \biggl((P_\lambda:K_\mu)_q\biggr)_{\lambda,\mu\in\irr{W}}.
        \]
        \begin{prop}
            \label{prop:reciprocity-K}
            Suppose that the three conditions in Definition \ref{dfn:BHR-conditions} are satisfied.
            We have
            \[(P:K) = [\widetilde{K}:L]^t.\]
            In particular, $(P:K)$ is independent of the choice of $K$-filtration.
        \end{prop}
        \begin{proof}
            We have an equation
            \begin{align}
                [P:L] = (P:K)\cdot [K:L].
            \end{align}
            By \eqref{eq:dec-gch}, we have
            \begin{align}
                (P:K) = [P:L]\cdot [K:L]^{-1} = [K:L]^t \cdot (\widetilde{K}:K).
            \end{align}
            By Corollary \ref{cor:Ktilde-K-sym}, the matrix $(\widetilde{K}:K)$ is symmetric.
            Therefore, we have $[K:L]^t \cdot (\widetilde{K}:K) =  \bigl((\widetilde{K}:K)\cdot [K:L]\bigr)^t$.
            Consequently, we obtain
            \begin{align}
                (P:K) = \bigl((\widetilde{K}:K)\cdot [K:L]\bigr)^t = [\widetilde{K}:L]^t
            \end{align} 
            as desired.
        \end{proof}

        Our purpose is to classify the structures of $\widetilde{K}$ and $K$ arising from a preorder relation.
        In some case different preorders have the same trace quotient modules and strict trace quotient modules.
        We will show that there is the preorder which have the minimum number of relations in such preorders if they satisfy the Brauer-Humphreys type reciprocity relation.
        
        Let $\precsim$ be a preorder on $\irr{W}$.
        Let $\{\widetilde{K}_\lambda\}_{\lambda\in\irr{W}}$ and $\{K_\lambda\}_{\lambda\in\irr{W}}$ be the strict trace quotient modules and the trace quotient modules with respect to $\precsim$, respectively.
        We define a preorder $\precsim_K$ on $\irr{W}$ as follows.
        \begin{dfn}\label{dfn:precsimK}
            For $\lambda,\lambda'\in\irr{W}$, we write as $\lambda\precsim_K \lambda'$ if and only if there exists a sequence $\lambda = \lambda_0,\lambda_1,\ldots,\lambda_l=\lambda'$ for some $l\geq 0$ such that
            \[
                \Ext^1_A(K_{\lambda_{i-1}},L_{\lambda_i})\neq 0
            \]
            for $1\leq i\leq l$.
        \end{dfn}
        It is straightforward to see the relation $\precsim_K$ is reflexive and transitive.
        In particular, this binary relation is a preorder.
        We remark that there exists a surjective morphism of graded $A$-modules
        \[
            \bigoplus_{i=1}^l P_{\mu_i}\gs{d_i} \longrightarrow \sum_{\mu\succsim\lambda,f\in\Hom_{A}(P_\mu,P_\lambda)_{>0}}\mathrm{Im}\,f = I_\lambda
        \]
        where $\mu_i\succsim\lambda$ and $d_i>0$ for each $i$.
        Since we have an exact sequence
        \[
            \xymatrix{
                \Hom_A(I_\lambda,L_\mu) \ar[r] &\Ext_A^1(K_\lambda,L_\mu) \ar[r] & \Ext_A^1(P_\lambda,L_\mu)=0,
            }
        \]
        then $\Ext^1_A(K_\lambda,L_\mu)\neq 0$ implies $\Hom_A(I_\lambda,L_\mu)\neq 0$ and there exists a nonzero morphism from $\bigoplus_{i=1}^l P_{\mu_i}\gs{d_i}$ to a grading shift of $L_\mu$.
        Hence, we have $\mu = \mu_i$ for some $i$.
        Consequently, we have $\lambda \precsim \mu$ if $\Ext^1_A(K_\lambda,L_\mu)\neq 0$.
        Similarly, we can show that $\Ext^1_A(\widetilde{K}_\lambda,L_\mu)\neq 0$ implies $\lambda\prec\mu$. 

        \begin{prop}\label{prop:lessrelations}
            Let $\precsim_K$ be the preorder defined in Definition \ref{dfn:precsimK}.
            We have $\lambda\precsim\lambda'$ if $\lambda\precsim_K \lambda'$.
        \end{prop}
        \begin{proof}
            By the definition of the preorder $\precsim_K$, the relation $\lambda\precsim_K \lambda'$ implies that there exists a sequence $\lambda = \lambda_0,\lambda_1,\ldots,\lambda_l=\lambda'$ for some $l\geq 0$ such that
            \[
                \Ext^1_A(K_{\lambda_{i-1}},L_{\lambda_i})\neq 0
            \]
            for $1\leq i\leq l$.
            Therefore, we obtain
            \[
                \lambda = \lambda_0\precsim \lambda_1\precsim \lambda_2 \precsim \cdots \precsim \lambda_l=\lambda'
            \]
            as required.
        \end{proof}

        \begin{lem}
            \label{lem:extpath}
            Let $F=\{F_\lambda\}_{\lambda\in\irr{W}}$ be a family of graded $A$-modules such that $F_\lambda$ is a nonzero graded quotient module of $P_\lambda$ for each $\lambda\in\irr{W}$.
            Let $d_1, d_2, \ldots, d_l$ be integers and let $\lambda_1, \lambda_2, \ldots, \lambda_l \in \irr{W}$.
            Assume $M\in A\hgmod$ has a $F$-filtration
            \[
                M = M^0 \supset M^1 \supset M^2 \supset \cdots \supset M^{l-1} \supset M^l \supset \cdots
            \]
            such that the $i$-th factor $M^{i-1}/M^i$ is isomorphic to $F_{\lambda_i}\gs{d_i}$ for $1\leq i \leq l$.
            There exists an integer $k\geq 0$ and an increasing sequence $1\leq a_0 < a_1 < a_2 \cdots < a_{k-1} < a_k=l$ with the following properties:
            \begin{itemize}
                \item $\Hom_A(M,L_{\lambda_{a_0}})\neq 0$,
                \item $\Ext_A^1(F_{\lambda_{a_{j-1}}},L_{\lambda_{a_j}})$ for $j = 1, 2, \ldots, k$.
            \end{itemize}
        \end{lem}

        \begin{proof}
            We have surjections $M\rightarrow F_{\lambda_1}\gs{d_1} \rightarrow L_{\lambda_1}\gs{d_1}$.
            In particular, we have $\Hom_A(M,L_{\lambda_1})\neq 0$.
            We will prove the lemma by induction on the length $l$ of sequence.
            If $l=1$, we set $k=0$ and $a_0 = 1 = l$.
            It is clear that there is no $j<k$, and the properties of the lemma are satisfied.
            Next, we assume $l>1$.
            By applying the induction hypothesis to $M^1$, we obtain an increasing sequence $2 \leq a_1 < a_2 < \ldots < a_k = l$ such that $\Hom_A(M^1,L_{\lambda_{a_1}})\neq 0$ and $\Ext_A^1(F_{\lambda_{a_{j-1}}},L_{\lambda_{a_j}})$ for $2 \leq j \leq k$.
            If $\Hom_A(M,L_{\lambda_{a_1}})\neq 0$, then this sequence satisfies the properties of the lemma.
            We assume $\Hom_A(M,L_{\lambda_{a_1}}) = 0$ and set $a_0 = 1$.
            It suffices to show that $\Ext_A^1(F_{\lambda_{a_0}},L_{\lambda_{a_1}})\neq 0$.
            A short exact sequence
            \[\xymatrix{
                0 \ar[r] &M^{1} \ar[r] & M \ar[r] & F_{\lambda_1}\gs{d_1} \ar[r] &0
            }\]
            induces a long exact sequence
            \[\xymatrix{
                \Hom_A(M,L_{\lambda_{a_1}})\ar[r]&\Hom_A(M^1,L_{\lambda_{a_1}})\ar[r]&\Ext_A^1(F_{\lambda_1}\gs{d_1},L_{\lambda_{a_1}}).
            }\]
            Since we have $\Hom_A(M^1,L_{\lambda_{a_1}})\neq 0$ and $\Hom_A(M,L_{\lambda_{a_1}}) = 0$, we obtain $\Ext_A^1( F_{\lambda_1}\gs{d_1},L_{\lambda_{a_1}})\neq 0$.
            This completes the proof.
        \end{proof}

        \begin{prop}
            \label{prop:KtildeK-precsimK}
            Suppose that the third condition in Definition \ref{dfn:BHR-conditions} is satisfied; that is, for each $\lambda \in\irr{W}$, a graded $A$-module $\widetilde{K}_\lambda$ admits a $K$-filtration.
            For $\lambda\in\irr{W}$, the first nonzero factor in a $K$-filtration of $\widetilde{K}_\lambda$ is isomorphic to $K_\lambda$.
            The other nonzero factors are isomorphic to $K_\mu\gs{d}$ for some $\mu\succsim_K\lambda$ and $d>0$. 
        \end{prop}
        \begin{proof}
            We have $\Hom_A(\widetilde{K}_\lambda,L_{\lambda'}\gs{d})_0$ is nonzero if and only if $\lambda=\lambda'$ and $d=0$.
            In particular, the first nonzero factor is isomorphic to $K_\lambda$.
            Let $f:\widetilde{K}_\lambda\rightarrow K_\lambda$ be a quotient map.
            Then we have $(\ker f)_d=0$ for $d\leq 0$.
            Hence for $i\geq 2$ the $i$-th factor is of the form $K_{\mu_i}\gs{d_i}$ for $d_i\geq 1$ and $\mu_i\in\irr{W}$.
            By Lemma \ref{lem:extpath}, we obtain $\lambda\precsim_K \mu_i$.
            This completes the proof. 
        \end{proof}

        \begin{cor}
            \label{cor:Ktilde-relation}
            Suppose that the third condition in Definition \ref{dfn:BHR-conditions} is satisfied.
            For $\lambda,\mu\in\irr{W}$, $\Ext^1_A(\widetilde{K}_\lambda,L_\mu)\neq 0$ implies $\lambda\precsim_K\mu$.
        \end{cor}
        \begin{proof}
            Fix $d\in\zz$ such that $\Ext^1_{A}(\widetilde{K}_\lambda,L_\mu\gs{d})_0 \neq 0$.
            A short exact sequence
            \[
                \xymatrix{ 0 \ar[r] & \widetilde{I}_\lambda \ar[r] & P_\lambda \ar[r] & \widetilde{K}_\lambda \ar[r] &0}
            \]
            induces a long exact sequence
            \[
                \xymatrix{ \Hom_{A}(\widetilde{I}_\lambda,L_\mu\gs{d})_0 \ar[r] & \Ext^1_{A}(\widetilde{K}_\lambda,L_\mu\gs{d})_0 \ar[r] & \Ext^1_{A}(P_\lambda,L_\mu\gs{d})_0=0.}
            \]
            We have $d>0$ because $\dim \Hom_{A}(\widetilde{I}_\lambda,L_\mu\gs{d})_0\geq \dim \Ext^1_{A}(\widetilde{K}_\lambda,L_\mu\gs{d})_0 > 0$.
            
            Fix a $K$-filtration of $\widetilde{K}_\lambda$
            \[
                \widetilde{K}_\lambda = M^0 \supset M^1 \supset M^2 \supset \cdots \supset M^{i-1} \supset M^i \supset \cdots.
            \]
            Since $\{M^i\}$ is a separable filtration, there is $i>0$ such that $M^i\subset (\widetilde{K}_\lambda)_{>d}$.
            Then we have $\Ext^1_{A}(M^i,L_\mu\gs{d})_0 = 0$ because $A$ is non-negatively graded.
            Without loss of generality, we can assume that $j$-th factor $M^{j-1}/M^j$ is nonzero for $1\leq j\leq i$.
            For each $j$, we have an isomorphism $M^{j-1}/M^j\simeq K_{\mu_j}\gs{d_j}$ for some $\mu_j\in\irr{W}$ and $d_j\in\zz$.
            By Proposition \ref{prop:KtildeK-precsimK}, we have $\mu_j\succsim_K \lambda$ and $d_j\geq 0$.  
            Since there is an exact sequence of graded $A$-modules
            \[
                \xymatrix{ 0\ar[r] &K_{\mu_j}\gs{d_j} \ar[r] &M^0/M^j \ar[r] & M^0/M^{j-1} \ar[r] &0 }  
            \]
            for each $j$, then we have 
            \begin{multline}
                \dim \Ext^1_{A}(M^0/M^j,L_{\mu}\gs{d})_0\\
                    \leq \dim \Ext^1_{A}(M^0/M^{j-1},L_{\mu}\gs{d})_0+\dim \Ext^1_{A}(K_{\mu_j}\gs{d_j},L_{\mu}\gs{d})_0.
            \end{multline}
            Therefore, we obtain
            \begin{align}
                \dim \Ext^1_{A}&(M^0/M^{i},L_{\mu}\gs{d})_0\\
                    &= \sum_{j=1}^{i} \dim \Ext^1_{A}(M^0/M^j,L_{\mu}\gs{d})_0 - \dim \Ext^1_{A}(M^0/M^{j-1},L_{\mu}\gs{d})_0\\
                    &\leq \sum_{j=1}^{i} \dim \Ext^1_{A}(K_{\mu_j}\gs{d_j},L_{\mu}\gs{d})_0.
            \end{align}
            We have an exact sequence
            \[
                \xymatrix{\Ext^1_{A}(\widetilde{K}_\lambda/M^i,L_\mu\gs{d})_0\ar[r] & \Ext^1_{A}(\widetilde{K}_\lambda,L_\mu\gs{d})_0   \ar[r] &\Ext^1_{A}(M^i,L_\mu\gs{d})_0 = 0,}
            \]
            and hence, we obtain 
            \[\dim \Ext^1_{A}(\widetilde{K}_\lambda/M^i,L_\mu\gs{d})_0\geq\dim \Ext^1_{A}(\widetilde{K}_\lambda,L_\mu\gs{d})_0 >0 .\]
            Therefore, we obtain $\sum_{j=1}^{i} \dim \Ext^1_{A}(K_{\mu_j}\gs{d_j},L_{\mu}\gs{d})_0\neq 0$.
            In particular, there is some $j$ such that $\Ext^1_{A}(K_{\mu_j}\gs{d_j},L_{\mu}\gs{d})_0\neq 0$.
            Consequently, we obtain $\lambda \precsim_K \mu_j \precsim_K \mu$ as desired.
        \end{proof}

        \begin{cor}
            \label{cor:KtildeK-equiv}
            Suppose that the third condition in Definition \ref{dfn:BHR-conditions} is satisfied.
            For each $\lambda\in\irr{W}$, all nonzero factors of a $K$-filtration of $\widetilde{K}_\lambda$ is isomorphic to $K_\mu\gs{d}$ for some $\mu\sim\lambda$ and $d\geq 0$. 
        \end{cor}
        \begin{proof}
            By Proposition \ref{prop:KtildeK-precsimK} and Proposition \ref{prop:lessrelations}, every nonzero factor of a $K$-filtration of $\widetilde{K}_\lambda$ is isomorphic to $K_\lambda$ or isomorphic to $K_\mu\gs{d}$ for some $\mu\succsim \lambda$ and $d>0$.
            If there is a factor isomorphic to $K_\mu\gs{d}$ where $\mu\succ \lambda$ and $d>0$, then we have $[(\widetilde{K}_\lambda)_d:\mu]\neq 0$.
            However, we have $[(\widetilde{K}_\lambda)_d:\mu] = 0$ by the definition of $\widetilde{K}_\lambda$.
            This is a contradiction.
            Consequently, the all nonzero factors are isomorphic to $K_\mu\gs{d}$ for some $\mu\sim\lambda$ and $d\geq 0$.
        \end{proof}

        \begin{prop}
            \label{prop:PKtilde-precsimK}
            Suppose that the second and the third conditions in Definition \ref{dfn:BHR-conditions} are satisfied.
            For $\lambda\in\irr{W}$, the first nonzero factor in $\widetilde{K}$-filtration of $P_\lambda$ is isomorphic to $\widetilde{K}_\lambda$.
            The other nonzero factors are isomorphic to $\widetilde{K}_\mu\gs{d}$ for some $\mu\succsim_K\lambda$ and $d>0$. 
        \end{prop}
        \begin{proof}
            By Proposition \ref{prop:ghom-multi}, $\Hom_A(P_\lambda,L_{\lambda'}\gs{d})_0 \neq 0$ if and only if $\lambda = \lambda'$ and $d=0$.
            Therefore, the first nonzero factor is isomorphic to $\widetilde{K}_\lambda$.
            By Lemma \ref{lem:extpath} and Corollary \ref{cor:Ktilde-relation}, the other factors are of the form $\widetilde{K}_\mu\gs{d}$ for some $\mu\succsim_K \lambda$ and $d\in\zz$.
        \end{proof}

        \begin{cor}\label{cor:PK-precsimK}
            Suppose that the second and the third conditions in Definition \ref{dfn:BHR-conditions} are satisfied.
            For $\lambda\in\irr{W}$, a projective module $P_\lambda$ is filtered by $\{K_\mu\}_{\mu\succsim_K\lambda}$.    
        \end{cor}         
        \begin{proof}
            Combine Propositions \ref{prop:KtildeK-precsimK}, \ref{prop:PKtilde-precsimK}, and Lemma \ref{lem:cnctfilts}.
        \end{proof}

        \begin{cor}
            \label{cor:higherK-relation}
            Suppose that the second and the third conditions in Definition \ref{dfn:BHR-conditions} are satisfied.
            Let $l$ be a positive integer and $\lambda,\mu\in\irr{W}$.
            Then $\Ext_A^l(K_\lambda,L_\mu)\neq 0$ implies $\lambda\precsim_K \mu$.
        \end{cor}
        \begin{proof}
            We prove this by induction on $l$.
            When $l=1$, by definition of the preorder $\precsim_K$, we have $\lambda\precsim_K \mu$ if $\Ext_A^1(K_\lambda,L_\mu)\neq 0$.
            When $l>2$, we can fix an integer $d\in\zz$ such that $\Ext^{l}_{A}(K_\lambda,L_\mu\gs{d})_0\neq 0$.
            By Corollary \ref{cor:PK-precsimK}, $P_\lambda$  admits a $\{K_\mu\}_{\mu\succsim_K \lambda}$-filtration
            \[
                P_\lambda = M^0 \supset M^1 \supset M^2 \supset \cdots \supset M^{i-1} \supset M^i \supset \cdots.
            \]
            Without loss of generality, we can assume that $M^0/M^1\simeq K_\lambda$. 
            Since $\{M^i\}_i$ is a separable filtration, there exists a number $k>0$ such that $M^k\subset (P_\lambda)_{>d}$.
            Since $A$ is non-negatively graded, a graded module $M^k$ admits a projective resolution such that its projective terms does not have the homogeneous part of degree $d$. 
            Hence, we have $\Ext_{A}^{l-1}(M^k,L_\mu\gs{d})_0=0$.
            There is an exact sequence
            \[
                \xymatrix{ 0 \ar[r] & M^k \ar[r] &P_\lambda \ar[r] &P_\lambda/M^k \ar[r] &0},
            \]
            and it induces a long exact sequence 
            \[
                \xymatrix{ 0=\Ext_{A}^{l-1}(M^k,L_\mu\gs{d})_0 \ar[r] & \Ext_{A}^l(P_\lambda/M^k,L_\mu\gs{d})_0 \ar[r] &\Ext_{A}^{l}(P_\lambda,L_\mu\gs{d})_0.  }
            \]
            Since $P_\lambda$ is projective, we have $\Ext_{A}^{l}(P_\lambda,L_\mu\gs{d})_0=0$, and hence $\Ext_{A}^l(P_\lambda/M^k,L_\mu\gs{d})_0 = 0$.
            For each $1\leq i\leq k$, we have an exact sequence
            \[
                \xymatrix{ 0 \ar[r] & M^i/M^k \ar[r] &M^{i-1}/M^k \ar[r] &M^{i-1}/M^i \ar[r] &0}.
            \] 
            Then, we obtain long exact sequence
            \[\xymatrix@=10pt{
                &\Ext_{A}^{l-1}(M^{i-1}/M^i,L_\mu\gs{d})_0 \ar[r] &\Ext_{A}^{l-1}(M^{i-1}/M^k,L_\mu\gs{d})_0 \ar[r] &\Ext_{A}^{l-1}(M^i/M^k,L_\mu\gs{d})_0 \\
                \ar[r] &\Ext_{A}^l    (M^{i-1}/M^i,L_\mu\gs{d})_0 \ar[r] &\Ext_{A}^l (M^{i-1}/M^k,L_\mu\gs{d})_0 .
            }\]
            Hence we obtain inequalities
            \begin{multline}
                \dim \Ext_{A}^l(M^0/M^1,L_\mu\gs{d})_0 \\
                \leq \dim \Ext_{A}^{l-1}(M^1/M^k,L_\mu\gs{d})_0 + \dim\Ext_{A}^l(M^0/M^k,L_\mu\gs{d})_0
            \end{multline}
            and 
            \begin{multline}
                \dim \Ext_{A}^{l-1}(M^{i-1}/M^k,L_\mu\gs{d})_0 \\
                    \leq \dim \Ext_{A}^{l-1}(M^{i-1}/M^i,L_\mu\gs{d})_0 + \dim \Ext_{A}^{l-1}(M^i/M^k,L_\mu\gs{d})_0 \qquad(2\leq i\leq k).
            \end{multline}
            
            Since we have $M^0/M^1\simeq K_\lambda$ and $\Ext_{A}^l(M^0/M^k,L_\mu\gs{d})_0 = \Ext_{A}^l(P_\lambda/M^k,L_\mu\gs{d})_0=0$, we obtain
            \begin{align}
                0 &< \dim \Ext_{A}^l(K_\lambda,L_\mu\gs{d})_0 \\
                &\leq \dim \Ext_{A}^{l-1}(M^1/M^k,L_\mu\gs{d})_0\\
                    &=  \sum_{i=2}^k \dim \Ext_{A}^{l-1}(M^{i-1}/M^k,L_\mu\gs{d})_0 - \dim \Ext_{A}^{l-1}(M^i/M^k,\L_\mu\gs{d})_0\\
                    &\leq \sum_{i=2}^k \dim \Ext_{A}^{l-1}(M^{i-1}/M^i,L_\mu\gs{d})_0.
            \end{align}
            Therefore, there is some integer $i\geq 2$ such that $\Ext_{A}^{l-1}(M^{i-1}/M^i,L_\mu\gs{d})_0\neq 0$.
            Since the filtration $\{M^i\}_i$ is given by Corollary \ref{cor:PK-precsimK}, the $i$-th factor $M^{i-1}/M^i$ is isomorphic to a grading shift of $K_{\mu'}$ for some $\mu'\succsim_K \lambda$.
            By the induction hypothesis, we have $\mu'\precsim_K \mu$.
            Consequently, we obtain $\lambda \precsim_K \mu' \precsim_K \mu$ as desired.
        \end{proof}

        \begin{cor}
            \label{cor:ext-eliminated}
            Suppose that the second and the third conditions in Definition \ref{dfn:BHR-conditions} are satisfied.
            Let $\lambda,\mu \in\irr{W}$.
            If $\Ext^{>0}_A(K_\lambda,L_\mu)\neq 0$, then we have $[(K_\lambda)_{>0}:L_\mu]_q =0$.
        \end{cor}
        \begin{proof}
            This is immediate from Corollary \ref{cor:higherK-relation}.
        \end{proof}

        \begin{cor}
            \label{cor:K-socle-sgn-relation}
            Suppose that the second and the third conditions in Definition \ref{dfn:BHR-conditions} are satisfied.
            Let $\lambda,\alpha,\beta\in\irr{W}$ and $a>0$.
            If $(K_\lambda)_{>0}$ admits a graded simple submodule $L_\alpha\gs{a}$ and $\Ext_A^n(L_\alpha,L_\beta) \neq 0$ holds, we have $\lambda\precsim_K\beta$.
        \end{cor}
        \begin{proof}
            Since the global dimension of $A$ equals $n$, an injective homomorphism $L_\alpha\gs{a}\rightarrow K_\lambda$ induces a surjection
            \[
                \xymatrix{
                    \Ext_A^n(K_\lambda,L_\beta) \ar[r] &\Ext_A^n(L_\alpha\gs{a},L_\beta) \ar[r] &\Ext_A^{n+1}(K_\lambda/L_\alpha\gs{a},L_\beta) = 0.
                }  
            \]
            Therefore, we obtain $\Ext_A^n(K_\lambda,L_\beta)\neq 0$.
            By Corollary \ref{cor:higherK-relation}, we have $\lambda\precsim_K \beta$.
        \end{proof}
        
        \begin{cor}
            \label{cor:appear-gep-vanish}
            Suppose that the second and the third conditions in Definition \ref{dfn:BHR-conditions} are satisfied.
            Let $\lambda,\mu \in\irr{W}$.
            If $[(K_\lambda)_{>0}:L_\mu]_q\neq 0$, then we have $\gep{K_\lambda}{L_\mu} = 0$.
        \end{cor}
        \begin{proof}
            By the construction of $K_\lambda$ in Definition \ref{dfn:K-Ktilde}, we have $\lambda \not \precsim \mu$.
            In particular, we have $\lambda \neq \mu$. 
            Hence, $\Hom_A(K_\lambda,L_{\mu}) = 0$.
            By Corollary \ref{cor:ext-eliminated}, we have $\Ext_A^{>0}(K_\lambda,L_{\mu}) =  0$.
            Therefore, we obtain 
            \[
                \gep{K_\lambda}{L_{\mu}} =  \gdim \Hom_A(K_\lambda,L_{\mu}) + \sum_{i=1}^n (-1)^i \gdim \Ext_A^i(K_\lambda,L_\mu) = 0.
            \]
        \end{proof}

        \begin{lem}\label{lem:factorgenerating}
            Let $D$ be a subset of $\irr{W}$. 
            Let $F_\lambda$ be a quotient module of $P_\lambda$ for $\lambda\in D$.
            If $M\in A\textrm{-gmod}$ is filtered by $\{F_\lambda\}_{\lambda\in D}$, then $M$ is generated by the isotypic components corresponding to $\lambda\in D$. 
        \end{lem}
        \begin{proof}
            By Proposition \ref{prop:ghom-multi}, the submodule of $M$ generated by the isotypic components corresponding to $\lambda\in D$ coincides with
            \[
                \sum_{\lambda \in D,f\in\Hom_A(P_\lambda,M)} \im\, f.  
            \]
            Suppose first that $M$ admits a finite $\{F_\lambda\}_{\lambda\in D}$-filtration
            \[
                M=M^0\supset M^1 \supset M^2 \supset \cdots \supset M^l =0.
            \]
            For each $i$, we have an isomorphism $M^{i-1}/M^i\simeq F_{\lambda_i}\gs{d_i}$ for $\lambda_i\in D$ and $d_i\in \zz$. 
            It suffices to show that there exists a surjection $P\rightarrow M$, where $P$ is a direct sum of projective modules of the form $P_\lambda\gs{d}$ with $\lambda\in D$ and $d\in \zz$. 
            We proceed by induction on $l$.
            When $l=1$, we have a surjection $P_{\lambda_1}\gs{d_1}\rightarrow F_{\lambda_1}\gs{d_1}\simeq M^0/M^1 =M$ , as required.
            If $l>1$, we obtain a surjection $p:P\rightarrow M^1$ by applying the induction hypothesis.
            Since $P_{\lambda_1}\gs{d_1}$ is projective, the natural quotient morphism $\pi:M\rightarrow M^0/M^1\simeq P_{\lambda_1}\gs{d_1}$ can be lifted to a homomorphism $p':P_{\lambda_1}\gs{d_1} \rightarrow M$.
            A lifted homomorphism $p'$ satisfies $\im\, p' + M^1 = M$.
            Then $(p,p'):P \oplus P_{\lambda_1}\gs{d_1}\rightarrow M$ is surjective, and so $(p,p')$ is the required surjection.

            To complete the proof, we consider the case where $M$ admits an infinite $\{F_\lambda\}_{\lambda\in D}$-filtration $\{M^i\}_{i=0}^\infty$.
            Given $d\in\zz$, it is possible to fix $l$ such that $(M^l)_d=0$ because $\{M^i\}_{i=0}^\infty$ is separable.
            Since $M/M^l$ is finitely filtered by $\{F_\lambda\}_{\lambda\in D}$, then $M/M^l$ is generated by the isotypic components of $M/M^l$ associated to $\lambda \in D$.
            We have $M_d = (M/M^l)_d$, and so $M_d$ is contained in the submodule of $M$ generated by the isotypic components of type $\lambda \in D$.
            Since $d$ is arbitrary, we conclude the isotypic components of type $\lambda \in D$ generate the whole of $M$.

        \end{proof}

        \begin{prop}\label{prop:recoverpreorder}
            Suppose that the three conditions in Definition \ref{dfn:BHR-conditions} are satisfied.
            For each $\lambda \in \irr{W}$, we have 
            \[
                I_\lambda = \sum_{\mu\succsim_K\lambda,f\in\Hom_{A}(P_\mu,P_\lambda)_{>0}}\mathrm{Im}\,f \quad \textrm{ and } \quad \widetilde{I}_\lambda = \sum_{\mu\succ_K\lambda,f\in\Hom_{A}(P_\mu,P_\lambda)_{>0}}\mathrm{Im}\,f.
            \]
            In particular, we recover $K=\{K_\lambda\}$ and $\widetilde{K}=\{\widetilde{K}_\lambda\}$ by applying the construction of the Definition \ref{dfn:K-Ktilde} to the preorder $\precsim_K$ determined by $K$. 
        \end{prop}
        \begin{proof}
            By Corollary \ref{cor:PK-precsimK}, $P_\lambda$ admits a $\{K_\mu\}_{\mu\succsim_K\lambda}$-filtration
            \[
                P_\lambda = M^0 \supset M^1 \supset M^2 \supset \cdots \supset M^{i-1} \supset M^i \cdots. 
            \]
            Without loss of generality, we can assume that $M^0 \supsetneq M^1$.
            Then the first factor $M^0/M^1$ is isomorphic to $K_\lambda$.
            By Lemma \ref{lem:factorgenerating}, a graded module $M^1$ is generated by its isotypic components associated to $\mu\succsim\lambda$.
            Since we have $M^1\subset (P_\lambda)_{>0}$, we obtain 
            \[
                \sum_{\mu\succsim_K\lambda,f\in\Hom_{A}(P_\mu,P_\lambda)_{>0}}\im\,f \supset M^1.  
            \]
            By Proposition \ref{prop:lessrelations}, we have 
            \begin{align}
                I_\lambda = \sum_{\mu\succsim\lambda,f\in\Hom_{A}(P_\mu,P_\lambda)_{>0}}\im\,f &\supset \sum_{\mu\succsim_K\lambda,f\in\Hom_{A}(P_\mu,P_\lambda)_{>0}}\im\,f \supset M^1.
            \end{align}
            Since we have $\gdim I_\lambda = \gdim P_\lambda - \gdim K_\lambda = \gdim M^1$, we obtain
            \[
                I_\lambda = \sum_{\mu\succsim_K\lambda,f\in\Hom_{A}(P_\mu,P_\lambda)_{>0}}\im\,f = M^1.
            \]

            Next we consider the module $\widetilde{I}_\lambda$.
            By Proposition \ref{prop:PKtilde-precsimK}, we have a $\{\widetilde{K}_\mu\}_{\mu\succsim_K \lambda}$-filtration of $P_\lambda$
            \[
                P_\lambda = N^0 \supset N^1 \supset N^2 \supset \cdots  
            \]
            with $M^0/M^1\simeq \widetilde{K}_\lambda$.
            Given $i>1$, if the $i$-th factor $N^{i-1}/N^i$ is nonzero, then $N^{i-1}/N^i$ is isomorphic to $\widetilde{K}_{\mu_i}\gs{d_i}$ for some $\mu_i\succsim_K \lambda$ and $d_i>0$ by Proposition \ref{prop:PKtilde-precsimK}.
            Then we have $\mu_i\succsim \lambda$ by Proposition \ref{prop:lessrelations}.
            By Proposition \ref{prop:reciprocity}, we have $[(K_{\mu_i})_{d_i}:L_\lambda] \neq 0$, and hence we obtain $\mu_i \not\precsim \lambda$.
            By Proposition \ref{prop:lessrelations}, it follows that $\mu_i \not\precsim_K \lambda$.
            Therefore, we have $\mu_i\succ_K \lambda$ and $\mu_i\succ\lambda$.
            Consequently, we have that $N^1$ is filtered by $\{\widetilde{K}_\mu\}_{\mu\succ_K \lambda}$ and also filtered by $\{\widetilde{K}_\mu\}_{\mu\succ \lambda}$.
            Thus, we obtain
            \[
                \sum_{\mu\succ\lambda,f\in\Hom_{A}(P_\mu,N^1)}\im\,f = \sum_{\mu\succ_K\lambda,f\in\Hom_{A}(P_\mu,N^1)}\im\,f = N^1  
            \]
            by Lemma \ref{lem:factorgenerating}.
            Since we have $N^1\subset (P_\lambda)_{>0}$, then we obtain 
            \[
                \sum_{\mu\succ_K\lambda,f\in\Hom_{A}(P_\mu,P_\lambda)_{>0}}\im\,f \supset \sum_{\mu\succ_K\lambda,f\in\Hom_{A}(P_\mu,N^1)}\im\,f =N^1
            \]
            and
            \[
                \widetilde{I}_\lambda = \sum_{\mu\succ\lambda,f\in\Hom_{A}(P_\mu,P_\lambda)_{>0}}\im\,f \supset  N^1.
            \]
            
            Since we have $\gdim \widetilde{I}_\lambda = \gdim P_\lambda - \gdim \widetilde{K}_\lambda = \gdim N^1$, so we conclude that 
            \[
                \widetilde{I}_\lambda = N^1 \subset \sum_{\mu\succ_K\lambda,f\in\Hom_{A}(P_\mu,P_\lambda)_{>0}}\im\,f.
            \]
            Finally, we will show that 
            \[
                \sum_{\mu\succ_K\lambda,f\in\Hom_{A}(P_\mu,P_\lambda)_{>0}}\im\,f \subset \widetilde{I}_\lambda.
            \]
            Assume to the contrary to deduce contradiction.
            Namely, we assume 
            \[
                \sum_{\mu\succ_K\lambda,f\in\Hom_{A}(P_\mu,P_\lambda)_{>0}}\im\,f \not\subset \widetilde{I}_\lambda.
            \]
            Then, there exists $f\in\Hom_{A}(P_\mu,P_\lambda)_{>0}$ for some $\mu\succ_K\lambda$ such that
            \[
                \im\,f \not\subset \widetilde{I}_\lambda.
            \]
            It follows that $[(\widetilde{K}_\lambda)_{>0}:L_\mu]_q\neq 0$ for some $\mu\succ_K\lambda$.
            By Proposition \ref{prop:reciprocity-K}, we obtain $(P_\mu:K_\lambda)_q\neq 0$.
            By Corollary \ref{cor:PK-precsimK}, we have $\mu\precsim_K\lambda$.
            This contradicts the fact $\mu\succ_K\lambda$.
            Therefore, we obtain
            \[
                \sum_{\mu\succ_K\lambda,f\in\Hom_{A}(P_\mu,P_\lambda)_{>0}}\im\,f \subset \widetilde{I}_\lambda
            \]
            as desired.
        \end{proof}

        \begin{prop}
            \label{prop:I-Ext1}
            Let $\lambda\in\irr{W}$. We set 
            \[
                D=\{\mu\in\irr{W}\mid\Ext_A^1(K_\lambda,L_\mu)\neq 0\}.
            \]
            We have
            \[
                I_\lambda = \sum_{\mu\in D,f\in\Hom_{A}(P_\mu,P_\lambda)_{>0}}\mathrm{Im}\,f.
            \]
        \end{prop}
        \begin{proof}
            We have $\lambda\precsim_K \mu$ for $\mu\in D$.
            By Proposition \ref{prop:lessrelations}, it is clear that
            \[
                I_\lambda \supseteq \sum_{\mu\in D,f\in\Hom_{A}(P_\mu,P_\lambda)_{>0}}\mathrm{Im}\,f.    
            \]
            Assume to the contrary that 
            \[
                I_\lambda \supsetneq \sum_{\mu\in D,f\in\Hom_{A}(P_\mu,P_\lambda)_{>0}}\mathrm{Im}\,f.
            \]
            Then we have 
            \[
                M\coloneqq I_\lambda \left/ \sum_{\mu\in D,f\in\Hom_{A}(P_\mu,P_\lambda)_{>0}}\mathrm{Im}\,f \right.     
            \]
            is nonzero $A$-module.
            By the construction of $M$, we have $M=M_{>0}$.
            There exists a nonzero morphism $f:M\rightarrow L_\mu\gs{d}$ for some positive integer $d$ and $\mu\in\irr{W}$.
            Fix such $\mu$.
            Clearly, $\mu$ is not belongs to $D$ and $\lambda\neq\mu$.
            We have $\Hom_A(M,L_\mu)\neq 0$, and hence $\Hom_A(I_\lambda,L_\mu)\neq 0$.
            Applying $\Hom_A(-,L_\mu)$ to a short exact sequence
            \[
                    \xymatrix{ 0\ar[r] &I_\lambda \ar[r] &P_\lambda \ar[r] &K_\lambda \ar[r] &0}
            \]
            we obtain a long exact sequence
            \[
                \xymatrix{ \ar[r] &\Hom_A(P_\lambda,L_\mu)=0\ar[r] &\Hom_A(I_\lambda,L_\mu) \ar[r] &\Ext^1_A(K_\lambda,L_\mu) \ar[r] &\cdots.}  
            \]
            Since $\Hom_A(I_\lambda,L_\mu)\neq 0$ , we deduce $\Ext^1_A(K_\lambda,L_\mu)\neq 0$.
            Thus, we conclude that $\mu$ belongs to $D$ which is a contradiction.
            Consequently, we obtain
            \[
                I_\lambda = \sum_{\mu\in D,f\in\Hom_{A}(P_\mu,P_\lambda)_{>0}}\mathrm{Im}\,f
            \]
            as required.
        \end{proof}

        We prove the converse of Proposition \ref{prop:reciprocity}.
        \begin{prop}
            \label{prop:verify-orth}
            Let $\precsim$ be a preorder on $\irr{W}$.
            Let $\{\widetilde{K}_\lambda\}_{\lambda\in\irr{W}}$ and $\{K_\lambda\}_{\lambda\in\irr{W}}$ be the strict trace quotient modules and the trace quotient modules with respect to $\precsim$ respectively.
            Suppose that the second and the third conditions in Definition \ref{dfn:BHR-conditions} are satisfied.
            If a projective module $P_\lambda$ has a $\widetilde{K}$-filtration with multiplicity $P_\lambda\filt \sum_{\mu\in\irr{W}} [K_\mu:L_\lambda]_q[\widetilde{K}_\mu]$ for each $\lambda\in\irr{W}$, then we have
            \begin{equation}
                \gdim \Ext_{A}^i(\widetilde{K}_\lambda,K_\mu^*) = \begin{cases}
                    1 &\text{if } i=0 \text{ and } \lambda=\mu,\\
                    0 & otherwise.
                \end{cases}
            \end{equation}
        \end{prop}
        \begin{proof} 
            We denote the number of elements in $\irr{W}$ by $N$. 
            For $\lambda \in \irr{W}$, we denote by $c(\lambda)$ the number of $\mu\in\irr{W}$ such that $\lambda \precsim \mu$.
            We relabel the elements of $\irr{W}$ by $\lambda_1,\lambda_2,\ldots,\lambda_{N-1},\lambda_N$ to be $c(\lambda_1)\geq c(\lambda_2)\geq \ldots \geq c(\lambda_{N-1})\geq c(\lambda_N)$.
            Since $\lambda_i \succ \lambda_j$ implies 
            \[
                \lambda_j \not\in \left\{\mu\in\irr{W}\mid \lambda_i \precsim \mu\right\} \subsetneq \left\{\mu\in\irr{W}\mid \lambda_j \precsim \mu\right\},
            \]
            then we have $\lambda_i \not\succ \lambda_j$ if $i\leq j$.
            We show the equality
                \begin{equation}
                    \label{eq:induction-prop}
                    \gdim \Ext_{A}^i(\widetilde{K}_\lambda,K_\mu^*) = \begin{cases}
                        1 &i=0 \text{ and } \lambda=\mu\\
                        0 & otherwise.
                    \end{cases}  
                \end{equation}
            by induction on $\lambda$.
            
            First, we consider the case $\lambda = \lambda_N$.
            For every $\lambda\in\irr{W}$, we have $\lambda \not\succ \lambda_N$.
            Then $\widetilde{K}_{\lambda_N} = P_{\lambda_N}$ is projective, and hence $\Ext_{A}^{>0}(\widetilde{K}_{\lambda_N},K_\mu^*)=0$ for all $\mu \in \irr{W}$.
            We fix a $\widetilde{K}$-filtration of $P_{\lambda_n}$ with multiplicity $P_{\lambda_N}\filt \sum_{\mu\in\irr{W}} [K_\mu:L_{\lambda_N}]_q[\widetilde{K}_\mu]$.
            By Proposition \ref{prop:PKtilde-precsimK}, the first nonzero factor of this filtration is $\widetilde{K}_{\lambda_N}$,
            and the $k$-th nonzero factor for $k>1$ is isomorphic to $\widetilde{K}_{\alpha_k}\gs{d_k}$ where $\alpha_k \succsim_K \lambda_N $ and $d_k >0$.
            By Proposition \ref{prop:lessrelations}, $\alpha_k \succsim_K \lambda_N $ implies $\alpha_k \succsim \lambda_N $.
            For each $k>1$, the number of factors isomorphic to $\widetilde{K}_{\alpha_k}\gs{d_k}$ is equal to $[(K_{\alpha_k})_{d_k}:{\lambda_N}]$.
            Since we have $d_k>0$, $[(K_{\alpha_k})_{d_k}:{\lambda_N}]\neq 0$ implies $\alpha_k\not\precsim \lambda_N$.
            Then we obtain $\alpha_k \succ \lambda_N$, but this cannot be happen.
            Therefore, we conclude that there is a only one nonzero factor $\widetilde{K}_{\lambda_N}$, and hence,$[K_\mu:L_{\lambda_N}]_q = \delta_{\lambda_N,\mu}$ for $\mu\in\irr{W}$.
            By Proposition \ref{prop:ghom-multi}, we obtain
                \[
                    \gdim \Hom_{A}(\widetilde{K}_{\lambda_N},K_\mu^*) = [K_\mu^*:L_{\lambda_N}]_q  = \delta_{\lambda_N,\mu}
                \]
            as desired.
            Next, we show the case $\lambda = \lambda_i$ for $i < N$ by using the induction hypothesis: the equality \eqref{eq:induction-prop} holds for $\lambda = \lambda_j$ where $i<j\leq N$.
            Let 
            \[
                P_{\lambda_i} = M^0 \supset M^1 \supset M^2 \supset \cdots \supset M^{m-1} \supset M^m = 0
            \]
            be a filtration of $P_{\lambda_i}$ with multiplicity $P_{\lambda_i}\filt \sum_{\mu\in\irr{W}} [K_\mu:L_{\lambda_i}]_q[\widetilde{K}_\mu]$.
            Remark that $K_\mu$ is finite dimensional for every $\mu\in\irr{W}$, and hence the filtration above is finite. 
            Without loss of generality, we can suppose the $k$-th factor $M^{k-1}/M^k$ is nonzero for $k=1,2,\ldots, m$.
            By Proposition \ref{prop:PKtilde-precsimK}, the first factor $M^0/M^1 \simeq \widetilde{K}_{\lambda_i}$ and the $k$-th factor $M^k/M^{k-1}$ for $k>1$ is isomorphic to $\widetilde{K}_{\alpha_k}\gs{d_k}$ for some $\alpha_k \succsim_K \lambda_i$ and $d_k>0$.
            By Proposition \ref{prop:lessrelations}, $\alpha_k \succsim_K \lambda_i $ implies $\alpha_k \succsim \lambda_i $.
            For each $k>1$ we have $[(K_{\alpha_k})_{d_k}:{\lambda_i}]\neq 0$ because the number of factors isomorphic to $\widetilde{K}_{\alpha_k}\gs{d_k}$ is equal to $[(K_{\alpha_k})_{d_k}:{\lambda_i}]$.
            Since we have $d_k>0$, $[(K_{\alpha_k})_{d_k}:{\lambda_i}]\neq 0$ implies $\alpha_k\not\precsim \lambda_N$.
            Therefore, we obtain $\alpha_k \succ \lambda_i$ for $k>1$.
            In particular, for each $k>1$ we have $\alpha_k \in \{\lambda_{i+1},\lambda_{i+2},\ldots,\lambda_{N-1},\lambda_{N}\}$.
            We have a short exact sequence
            \[\xymatrix{
                0\ar[r] &M^k \ar[r] &M^{k-1} \ar[r] &\widetilde{K}_{\alpha_k}\gs{d_k} \ar[r] &0.
            }\]
            For each $\mu\in\irr{W}$, a functor $\Hom_{A}(-,K_\mu^*)$ induces a long exact sequence 
            \[\xymatrix@R=10pt{
                0\ar[r] &\Hom_{A}(\widetilde{K}_{\alpha_k}\gs{d_k},K_\mu^*) \ar[r] &\Hom_{A}(M^{k-1},K_\mu^*) \ar[r] &\Hom_{A}(M^k,K_\mu^*)\\
                \ar[r] &\Ext^1_{A}(\widetilde{K}_{\alpha_k}\gs{d_k},K_\mu^*) \ar[r] &\Ext^1_{A}(M^{k-1},K_\mu^*) \ar[r] &\Ext^1_{A}(M^k,K_\mu^*)\\
                \ar[r] &\Ext^2_{A}(\widetilde{K}_{\alpha_k}\gs{d_k},K_\mu^*) \ar[r] &\cdots.
            }\] 
            For each $k>1$, we have
            \[\gdim \Hom_{A}(\widetilde{K}_{\alpha_k}\gs{d_k},K_\mu^*) = q^{-d_k}\delta_{\alpha_k,\mu} \textrm{ and } \Ext^{>0}_{A}(\widetilde{K}_{\alpha_k}\gs{d_k},K_\mu^*) = 0 \]
            by induction hypothesis.
            Therefore, we obtain
            \[\gdim \Hom_{A}(M^{k-1},K_\mu^*) = \gdim \Hom_{A}(M^k,K_\mu^*) + q^{-d_k}\delta_{\alpha_k,\mu}\]
            and
            \[\Ext^{>0}_{A}(M^{k-1},K_\mu^*) \simeq \Ext^{>0}_{A}(M^k,K_\mu^*).\]
            Repeat this procedure, we get that
            \begin{align*}
                \gdim \Hom_{A}(M^1,K_\mu^*) &= \sum_{k=2}^{m} q^{-d_k}\delta_{\alpha_k,\mu}
            \end{align*}
            and
            \begin{equation}
                \label{eq:ext-M^1-banish}
                \Ext^{>0}_{A}(M^1,K_\mu^*) \simeq \Ext^{>0}_{A}(M^m,K_\mu^*)=0.    
            \end{equation}
            
            Since the number of factors isomorphic to $\widetilde{K}_{\mu}\gs{d}$ is equal to $[(K_{\mu})_{d}:\lambda_i]$ for each  integer $d$, and the first factor is isomorphic to $\widetilde{K}_{\lambda_i}$, then we have
            \begin{align}
                \sum_{k=2}^{m} q^{-d_k}\delta_{\alpha_k,\mu} &= \sum_{d\in\zz} q^{-d}[(K_\mu)_d:\lambda_i] - \delta_{\lambda_i,\mu} = [(K_\mu)^*:L_{\lambda_i}]_q - \delta_{\lambda_i,\mu}.
            \end{align}

            A short exact sequence
            \[\xymatrix{
                0\ar[r] &M^1 \ar[r] &M^0 = P_{\lambda_i} \ar[r] &M^0/M^1\simeq\widetilde{K}_{\lambda_i} \ar[r] &0
            }\]
            induces a long exact sequence

            \[\xymatrix@R=10pt{
                0\ar[r] &\Hom_{A}(\widetilde{K}_{\lambda_i},K_\mu^*) \ar[r] &\Hom_{A}(P_{\lambda_i},K_\mu^*) \ar[r] &\Hom_{A}(M^1,K_\mu^*)\\
                \ar[r] &\Ext^1_{A}(\widetilde{K}_{\lambda_i},K_\mu^*) \ar[r] &\Ext^1_{A}(P_{\lambda_i},K_\mu^*) \ar[r] &\Ext^1_{A}(M^1,K_\mu^*)\\
                \ar[r] &\Ext^2_{A}(\widetilde{K}_{\lambda_i},K_\mu^*) \ar[r] &\Ext^2_{A}(P_{\lambda_i},K_\mu^*)
            }\]
            Since $P_{\lambda_i}$ is projective, we have $\Ext^1_{A}(P_{\lambda_i},K_\mu^*) = \Ext^2_{A}(P_{\lambda_i},K_\mu^*) = 0$.
            Then we obtain
            \[
                \Ext^2_{A}(\widetilde{K}_{\lambda_i},K_\mu^*)\simeq\Ext^1_{A}(M^1,K_\mu^*) = 0
            \]
            by \eqref{eq:ext-M^1-banish}.

            If $\mu \precsim \lambda_i$, then we have
            \begin{align}
                \gdim \Hom_{A}(M^1,K_\mu^*)&= [(K_\mu)^*:L_{\lambda_i}]_q - \delta_{\lambda_i,\mu} =\delta_{\lambda_i,\mu} - \delta_{\lambda_i,\mu}=0.        
            \end{align}
            Therefore, we obtain $\gdim \Ext^1_{A}(\widetilde{K}_{\lambda_i},K_\mu^*) =0$ and
            \[
                \gdim \Hom_{A}(\widetilde{K}_{\lambda_i},K_\mu^*) = \gdim\Hom_{A}(P_{\lambda_i},K_\mu^*) = \delta_{\lambda_i,\mu}.
            \]

            If $\mu \not\precsim \lambda_i$, a $\widetilde{K}$-filtration of $P_\mu$ has no factors isomorphic to a grading shift of $\widetilde{K}_{\lambda_i}$ by Proposition \ref{prop:PKtilde-precsimK} and Proposition \ref{prop:lessrelations}.
            Therefore, we have $[K_{\lambda_i}:L_\mu]_q=0$.
            For every $\lambda'\in \irr{W}$ such that $\lambda'\sim \lambda_i$, we can show $[K_{\lambda'}:L_\mu]_q=0$ in a similar fashion.
            By Corollary \ref{cor:KtildeK-equiv}, $\widetilde{K}_{\lambda_i}$ is filtered by $\{K_{\lambda'}\}_{\lambda'\sim\lambda_i}$.
            Therefore, we obtain $[\widetilde{K}_{\lambda_i}:L_\mu]_q = 0$.
            For any homogeneous morphism $f\in \Hom_{A}(\widetilde{K}_{\lambda_i},K_\mu^*)$, the image $\im f$ is a graded submodule of $K_\mu^*$.
            Since $(K_\mu^*)_0$ is isomorphic to $\mu$, we have $(\im f)_0 = 0$.
            Then we have $\im f\cdot (K_\mu)_0 = 0$.
            Since $(K_\mu)_0$ generates $K_\mu$, we have $\im f\cdot K_\mu = 0$.
            Therefore, we obtain $\im f = 0$.
            It follows that $\Hom_{A}(\widetilde{K}_{\lambda_i},K_\mu^*) = 0$.
            Since there is an exact sequence
            \[
                \xymatrix@R=10pt{
                    &0=\Hom_{A}(\widetilde{K}_{\lambda_i},K_\mu^*) \ar[r] &\Hom_{A}(P_{\lambda_i},K_\mu^*) \ar[r] &\Hom_{A}(M^1,K_\mu^*)\\
                    \ar[r] &\Ext^1_{A}(\widetilde{K}_{\lambda_i},K_\mu^*) \ar[r] &\Ext^1_{A}(P_{\lambda_i},K_\mu^*)=0,}
            \]
            we obtain
            \begin{align}
                \gdim \Ext^1_{A}(\widetilde{K}_{\lambda_i},K_\mu^*) &= \gdim \Hom_{A}(P_{\lambda_i},K_\mu^*) -\gdim \Hom_{A}(M^1,K_\mu^*)\\
                    &= \gdim \Hom_{A}(P_{\lambda_i},K_\mu^*) -([(K_\mu)^*:L_{\lambda_i}]_q - \delta_{\lambda_i,\mu})\\
                    &= \gdim \Hom_{A}(P_{\lambda_i},K_\mu^*) -[(K_\mu)^*:L_{\lambda_i}]_q.
            \end{align}
            By Proposition \ref{prop:ghom-multi}, we have
            \[
                \gdim \Ext^1_{A}(\widetilde{K}_{\lambda_i},K_\mu^*)= [(K_\mu)^*:L_{\lambda_i}]_q-[(K_\mu)^*:L_{\lambda_i}]_q = 0. 
            \]

            Consequently, we completes the inductive step, and the result follows by induction.
        \end{proof}




    \section{Dihedral group $D_n$}\label{sec:dihedralgroup}
    Let $n\geq 3$ be an odd integer and $\zeta$ be a fixed primitive $n$-th root of unity.
    The dihedral group $D_n$ of order $2n$ is generated by two elements $r,s$ subject to the relations:
    \[
        r^n = s^2 = e, srs^{-1} = r^{-1}.   
    \]
    In the below, we set $W \coloneqq D_n$. We have a set of explicit representatives:
    \[
        W = \{r^0,r^1,r^2,\ldots,r^{n-2},r^{n-1},r^0s,r^1s,r^2s,\ldots,r^{n-2}s,r^{n-1}s\}.
    \]  
    For each $i \in \zz$, we have a $2$-dimensional $W$-representation $\chi_i$ with its basis $\{b_i^+,b_i^-\}$ such that:
    \[
        r\cdot b_i^\pm \coloneqq \zeta^{\pm i} b_i^\pm, \quad s\cdot b_i^\pm \coloneqq b_i^\mp,
    \]
    where multiple signs correspond to each other.
    Since we have $\zeta^n = 1$, we understand that $\chi_i$ is indexed by $i \in \zz/n\zz$.
    We define two $1$-dimensional $W$-representations $\triv$ and $\sgn$ with their bases $b_\triv$ and $b_\sgn$ such that their $W$-actions are given by
    \[    
        r\cdot b_\triv \coloneqq b_\triv,\, s\cdot b_\triv \coloneqq b_\triv,\, r\cdot b_\sgn  \coloneqq b_\sgn,\, \textrm{and}\,   s\cdot b_\sgn  \coloneqq -b_\sgn
    \]
    respectively.
    For convenience, we use the following notations:
    \[
        \chi_{(1,1)} \coloneqq \triv,\, b_{(1,1)} \coloneqq b_\triv,\, \chi_{(1,-1)} \coloneqq \sgn,\,  b_{(1,-1)} \coloneqq b_\sgn.\\
    \]
    In particular, we have
    \[
        r\cdot b_{(1,u)} = b_{(1,u)} \quad\textrm{and}\quad  s\cdot b_{(1,u)} = ub_{(1,u)} 
    \]
    for $u = 1,-1$.
 
    \begin{prop}\label{prop:relationofreps} 
        There are isomorphisms of $W$-representations
            \[\chi_i \simeq \chi_{n-i},\, \chi_0 \simeq \triv\oplus\sgn,\,\textrm{ and}\quad \cc W  \simeq \bigoplus_{i\in \zz/n\zz} \chi_i\simeq \triv \oplus \sgn \oplus \bigoplus_{i=1}^{\frac{n-1}{2}}\chi_i^{\oplus 2}. \]
                
    \end{prop}
    \begin{prop}
        A $W$-representation $V$ is isomorphic to its dual $V^*$.
    \end{prop}
    \begin{proof}
        These results can be easily checked by inspecting the character table (cf. ~\cite[Section 5.3]{Serre1977}).
    \end{proof}
    \begin{cor}

       We have $\irr{W} = \{\triv,\sgn\}\sqcup \{\chi_i\}_{i=1}^{\frac{n-1}{2}}$.
    \end{cor}
    \begin{proof}
        It follows from Proposition \ref{prop:relationofreps} and the algebraic Peter-Weyl theorem (see e.g.~\cite[Section 2.4]{Serre1977}).
    \end{proof}
    We set $\C \coloneq (n-1)/2$, that is the number of isomorphism classes of two-dimensional simple representations.
    
    \begin{prop}\label{prop:tensorofreps}       
        We have isomorphisms between $W$-representations
        \[
            \triv\otimes\triv \simeq\triv,\quad \triv\otimes\sgn \simeq \sgn,\quad \sgn\otimes\sgn \simeq \triv,  
        \]
        \[
            \triv\otimes\chi_i \simeq \chi_i,\quad \sgn\otimes\chi_i \simeq \chi_i,\quad \textrm{and} \quad\chi_i\otimes\chi_j \simeq \chi_{i+j}\oplus\chi_{i-j}.
        \]
    \end{prop}
    \begin{proof}
        This can be checked by calculating the characters (see e.g.~\cite[Section 5.3]{Serre1977}).
    \end{proof}

    Let $S$ be a polynomial ring $\cc\left[ X,Y\right]$ in two variables $X,Y$.
    We equip $S$ with a grading by setting $\deg X = \deg Y = 1$.
    We define the $W$-action on $S$ by 
    \[ 
        r\cdot X\coloneqq \zeta X, \quad r\cdot Y\coloneqq \zeta^{-1} Y, \quad s\cdot X \coloneqq Y, \textrm{ and }  s\cdot Y \coloneqq X
    \]
    extended to the whole $S$.
    We have a graded $W$-representation isomorphism $S_1\simeq \chi_1\gs{1}$.
    
    By Proposition \ref{prop:gldim-skewgrpalg}, a skew group algebra $S*W$ has global dimension two.
    Set $A \coloneqq S*W$ and $R\coloneqq S^W$.
    \begin{prop}
        \label{prop:gch-S}
        We have 
        \[
            \gch P_\triv = \gch S = \frac{1}{(1-q^2)(1-q^n)}\cdot\left([L_\triv]+q^n[L_\sgn]+\sum_{i=1}^C(q^i+q^{n-i})[L_{\chi_i}]\right).
        \]
    \end{prop}
    \begin{proof}
        Since $S$ is an integral domain, there exists an injective homomorphism $S\gs{2}\rightarrow S$ given by the multiplication by $XY\in S_2$.
        By Lemma \ref{lem:repeatingfilt}, we have 
        \[
            S\filt \frac{1}{1-q^2} \cdot [M],
        \]
        where $M = S/(S\cdot XY)$, and hence $\gch S = \frac{1}{1-q^2}\cdot\gch M$.
        For each $d\geq 2$, a quotient vector space $M_d = S_d/(S_{d-2}\cdot XY)$ has a basis $\{\overline {X^d}, \overline {Y^d}\}$.
        We have an isomorphism of $W$-representation:
        \[\begin{matrix}
            \chi_d &\simeq &M_d\\
            b_d^+  &\longmapsto & \overline{X^d}\\
            b_d^-  &\longmapsto & \overline{Y^d}.
        \end{matrix}\]
        It is straightforward to see that this isomorphism preserves $W$-action by inspection.
        We have $M_0\simeq \triv$ and $M_1\simeq\chi_1$.
        Therefore, we have
        \[
            M_d \simeq \begin{cases}
                \triv &(d=0)\\
                \triv\oplus\sgn &(d>0 \textrm{ and } d\equiv 0 \textrm{ mod }n)\\
                \chi_i  &(d\equiv i \textrm{ or } d\equiv n-i \textrm{ mod }n \textrm{ for }1\leq i \leq \C).
            \end{cases}                
        \]
        We obtain
        \[
            \gch M = \frac{1}{1-q^n}\left([L_\triv]+q^n[L_\sgn] + \sum_{i=1}^C (q^i+q^{n-i})[L_{\chi_i}]\right).
        \]
        Consequently, we have
        \[
            \gch S = \frac{1}{1-q^2}\gch M = \frac{1}{(1-q^2)(1-q^n)}\left([L_\triv]+q^n[L_\sgn] + \sum_{i=1}^C (q^i+q^{n-i})[L_{\chi_i}]\right).  
        \]
    \end{proof}

    \begin{cor}
        \label{cor:gch-Psgn}
        We have 
        \[
            \gch P_\sgn = \frac{1}{(1-q^2)(1-q^n)}\cdot\left([L_\sgn]+q^n[L_\triv]+\sum_{j=1}^{n-1}q^j[L_{\chi_j}]\right).
        \]

    \end{cor}
    \begin{proof}
        This is immediate from Proposition  \ref{prop:tensorofreps} and \ref{prop:gch-S}. 
    \end{proof}
        Note that $L_{\chi_j}$ is isomorphic to $L_{\chi_{n-j}}$ for $1\leq j\leq C$ by Proposition \ref{prop:relationofreps}.
        Hence, we have $[P_\triv:L_{\chi_j}]_q = [P_\sgn:L_{\chi_j}]_q = (q^j+q^{n-j})/(1-q^2)(1-q^n)$.

    \begin{cor}
        \label{cor:gch-Pi}
        For each $1\leq i\leq \C$, we have
        \[
            \gch P_{\chi_i} = \frac{1}{(1-q^2)(1-q^n)}\cdot\left(\sum_{j=0}^{n-1} q^j\gch L_{\chi_{i+j}}+ \sum_{j=1}^{n}q^{j}\gch L_{\chi_{i-j}}\right).
        \]
    \end{cor}
    \begin{proof}
        By Proposition \ref{prop:tensorofreps}, we have
        \begin{align}
            \gch P_{\chi_i} &= \gch P_\triv\otimes\chi_i \\
            &= \frac{1}{(1-q^2)(1-q^n)}\cdot\left([L_{\chi_i}]+q^n[L_{\chi_i}]+\sum_{j=1}^{n-1}q^j(\gch L_{\chi_{i+j}}+\gch L_{\chi_{i-j}})\right)\\
            &= \frac{1}{(1-q^2)(1-q^n)}\cdot\left(\sum_{j=0}^{n-1} q^j\gch L_{\chi_{i+j}}+ \sum_{j=1}^{n}q^{j}\gch L_{\chi_{i-j}}\right).
        \end{align}
    \end{proof}
    Note that $L_{\chi_0}$ is not simple, and we used here that $\gch L_{\chi_0}=\gch L_\triv + L_\sgn$.

    \begin{lem}
        \label{lem:resofsimple}
        For a $W$-representation $V$, a graded semisimple module $L_V$ admits a graded projective resolution of the form
        \[\xymatrix{
            0 \ar[r] &P_{\sgn\otimes V}\gs{2} \ar[r] &P_{\chi_1\otimes V}\gs{1} 
            \ar[r]  &P_{\triv\otimes V} \ar[r] &L_V \ar[r] &0.  
        }
        \]
    \end{lem}
    \begin{proof}
        We have the Koszul resolution:
        \begin{align}\label{eq:res-Ltriv}\vcenter{\xymatrix{
            0\ar[r] &P_\sgn\gs{2} \ar[rrr]^{\{1\otimes b_\sgn\mapsto Y\otimes b_1^+ - X\otimes b_1^-\}} &&&P_{\chi_1}\gs{1} \\
            \ar[rrr]^{\left\{\begin{matrix}1\otimes b_1^+ \mapsto X\otimes b_\triv\\1\otimes b_1^- \mapsto Y\otimes b_\triv  \end{matrix}\right\}} &&&P_\triv \ar[r] &L_\triv \ar[r] &0  
        }}\end{align}
        (see e.g.~\cite{weibel_1994}).
        By applying a functor $(-\otimes_\cc L_V)$ to \eqref{eq:res-Ltriv}, we obtain an exact sequence
        \[\xymatrix{
            0 \ar[r] &P_{\sgn\otimes V}\gs{2} \ar[r] &P_{\chi_1\otimes V}\gs{1} 
            \ar[r]  &P_{\triv\otimes V} \ar[r] &L_{\triv\otimes V}\simeq L_V \ar[r] &0 
        }\]
        as desired.

    \end{proof}

    \begin{cor}
        \label{cor:Ext-LL}
        Let $d$ be an integer and $i$ be a non-negative integer.
        For $\lambda,\mu\in\irr{W}$, we have
        \[
            \dim\Ext_{A\mathrm{\mathchar`-gmod}}^i(L_{\lambda},L_\mu\gs{d}) = 
                \begin{cases}
                    [\triv\otimes\lambda:\mu]       &(i=d=0),\\
                    [\chi_1\otimes\lambda:\mu]      &(i=d=1),\\
                    [\sgn\otimes\lambda:\mu]      &(i=d=2),\\
                    0 &otherwise.
                \end{cases} 
        \]

    \end{cor}

    \begin{proof}
        Since we have the projective resolutions in Lemma \ref{lem:resofsimple}, which are minimal resolutions, the assertion can be easily checked.

    \end{proof}

    \begin{lem}
        \label{lem:gch-diff}
        Let $i$ be an integer.
        Then we have
        \[
            \gch P_{\chi_i} -\gch P_{\chi_{i+1}}\gs{1} = \frac{1}{1-q^n}\sum_{k=0}^{n-1} q^k\cdot\gch L_{\chi_{i-k}}.
        \]
    \end{lem}
    \begin{proof}
        Let $j$ be an integer.
        By Lemma \ref{lem:resofsimple} and Proposition \ref{prop:tensorofreps}, we have an exact sequence
        \[\xymatrix{
            0 \ar[r] &P_{\chi_j}\gs{2} \ar[r] &P_{\chi_{j+1}}\gs{1} \oplus P_{\chi_{j-1}}\gs{1} 
            \ar[r]  &P_{\chi_j} \ar[r] &L_{\chi_j} \ar[r] &0.
        }
        \]
        Then, we have
        \begin{align}
            \gch P_{\chi_j} -q^1\gch P_{\chi_{j+1}} &= \gch L_{\chi_j} + q^1\gch P_{\chi_{j-1}} -q^2\gch P_{\chi_j}\\
                &=  q(\gch P_{\chi_{j-1}} -q^1\gch P_{\chi_j}) + \gch L_{\chi_j}
        \end{align}
        Applying this equality repeatedly, we obtain
        \begin{align}
            \gch P_{\chi_i} -q^1\gch P_{\chi_{i+1}} &= q^1(\gch P_{\chi_{i-1}} -q^1\gch P_{\chi_i})+ \gch L_{\chi_i} \\
            &= q^2 (\gch P_{\chi_{i-2}} -q^1\gch P_{\chi_{i-1}})+\gch L_{\chi_i} + q^1\gch L_{\chi_{i-1}} \\
            &\qquad \vdots\\
            &= q^n (\gch P_{\chi_{i-n}} -q^1\gch P_{\chi_{i-(n-1)}})+ \sum_{k=0}^{n-1} q^k \gch L_{\chi_{i-k}}.
        \end{align}
        By Proposition \ref{prop:relationofreps}, we have $P_{\chi_{i-n}} \simeq P_{\chi_{i}}$ and $P_{\chi_{i-(n-1)}}\simeq P_{\chi_{i+1}}$.
        Consequently, we obtain
        \[
            \gch P_{\chi_i} -q^1\gch P_{\chi_{i+1}} = \frac{1}{1-q^n}\sum_{k=0}^{n-1} q^k\gch L_{\chi_{i-k}}.
        \]
    \end{proof}

    \begin{prop}
        \label{prop:gEP-gch-Lmu}
        Let $M$ be a finite dimensional graded $A$-module.
        For $\mu\in\irr{W}$, we have
        \[
            \gep{M}{L_\mu} = \sum_{d\in \zz}q^{-d}\left([M_d:\mu]+[M_{d-2}:\sgn\otimes\mu] -\sum_{\lambda\in\irr{W}}[\chi_1\otimes\mu:\lambda][M_{d-1}:\lambda]\right). 
        \] 
    \end{prop}
    \begin{proof}
        Note that $\sgn\otimes\mu$ is an irreducible representation by Proposition \ref{prop:tensorofreps}.
        We have
        \[
            \gep{M}{L_\mu} = \sum_{\lambda\in\irr{W}}\sum_{d\in\zz} q^{-d}[M_d:\lambda]\gep{L_\lambda}{L_\mu}
        \]
        by Corollary \ref{cor:epdectosimple}.
        By Corollary \ref{cor:Ext-LL}, we have 
        \[
            \gep{L_\lambda}{L_\mu}=  [\triv\otimes\lambda:\mu]-q^{-1}[\chi_1\otimes\lambda:\mu]+q^{-2}[\sgn\otimes\lambda:\mu].
        \]
        Thus, we obtain
        \begin{align}
            &\gep{M}{L_\mu} \\
            &= \sum_{\lambda\in\irr{W}}\sum_{d\in\zz} q^{-d}[M_d:\lambda]([\triv\otimes\lambda:\mu]-q^{-1}[\chi_1\otimes\lambda:\mu]+q^{-2}[\sgn\otimes\lambda:\mu])\\
            &=\sum_{d\in\zz} \sum_{\lambda\in\irr{W}}q^{-d}([M_d:\lambda][\triv\otimes\lambda:\mu]-[M_{d-1}:\lambda][\chi_1\otimes\lambda:\mu]+[M_{d-2}:\lambda][\sgn\otimes\lambda:\mu]).
        \end{align}
        Since $[\triv\otimes\lambda:\mu]\neq 0$ implies $\lambda = \mu$, and $[\sgn\otimes\lambda:\mu]\neq 0$ implies $\lambda = \sgn\otimes\mu$, we have
        \[
            \gep{L_\lambda}{L_\mu} = \sum_{d\in\zz} q^{-d}\left([M_d:\mu]+[M_{d-2}:\sgn\otimes\mu]-\sum_{\lambda\in\irr{W}}[M_{d-1}:\lambda][\chi_1\otimes\lambda:\mu]\right).    
        \]
        By Proposition \ref{prop:adj-sym}, we have $[\chi_1\otimes\lambda:\mu] = [\chi_1\otimes\mu:\lambda]$.
        These complete the proof.
    \end{proof}

    \begin{prop}
        \label{prop:Ext1-gch}
        Let $M$ be a finitely generated graded $A$-module.
        Let $\mu\in\irr{W}$ and let $d$ be an integer.
        If we have $\Ext^1_{A\hgmod}(M,L_\mu\gs{d})\neq 0$, then there exists $\lambda \in \irr{W}$ such that $[\mu\otimes\chi_1:\lambda]\neq 0$ and $[M_{d-1}:\lambda]\neq 0$.
    \end{prop}
    \begin{proof}
        Since $M$ is bounded below, we can suppose $M=M_{\geq 0}$ with out loss of generality.
        We have a short exact sequence
        \[
            \xymatrix{0 \ar[r] & M_{\geq d+1} \ar[r] &M \ar[r] & M/M_{\geq d+1} \ar[r] &0.}    
        \]
        Therefore, we obtain
        \begin{multline}
            \dim \Ext^1_{A\hgmod}(M,L_\mu\gs{d})\\
                 \leq \dim \Ext^1_{A\hgmod}(M_{\geq d+1},L_\mu\gs{d}) + \dim \Ext^1_{A\hgmod}(M/M_{\geq d+1},L_\mu\gs{d}).     
        \end{multline}
        
        Since $A$ is non-negatively graded, we have $\Ext^1_{A\hgmod}(M_{\geq d+1},L_\mu\gs{d}) = 0$.
        Hence, we have
        \[
            0 < \dim \Ext^1_{A\hgmod}(M,L_\mu\gs{d})\leq \dim \Ext^1_{A\hgmod}(M/M_{\geq d+1},L_\mu\gs{d}).
        \]
        For each non-negative integer $i$, a quotient module $M_{\geq i}/M_{\geq i+1}$ is a semisimple graded $A$-module.
        In particular, we have a direct sum decomposition
        \[
            M_{\geq i}/M_{\geq i+1} \simeq \bigoplus_{\lambda\in\irr{W}} L_\lambda\gs{i}^{\oplus[M_i:\lambda]}.
        \]
        By Corollary \ref{cor:Ext-LL}, we have
        \[
            \dim \Ext^1_{A\hgmod}(M_{\geq i}/M_{\geq i+1},L_\mu\gs{d}) = \begin{cases}
                    [M_{d-1}\otimes \chi_1:\mu] &(i=d-1)\\
                    0       &(i\neq d-1).
            \end{cases}
        \]
        We have a short exact sequence
        \[
            \xymatrix{0 \ar[r] & M_{\geq i}/M_{\geq i+1} \ar[r] &M/M_{\geq i+1} \ar[r] & M/M_{\geq i} \ar[r] &0}.
        \]
        From this, we obtain an inequality
        \begin{multline}
            \dim \Ext^1_{A\hgmod}(M/M_{\geq i+1},L_\mu\gs{d})\\
                \leq \dim \Ext^1_{A\hgmod}(M_{\geq i}/M_{\geq i+1},L_\mu\gs{d}) + \dim \Ext^1_{A\hgmod}(M/M_{\geq i},L_\mu\gs{d}).
        \end{multline}
        
        Consequently, we have
        \begin{align}
            0 &< \dim \Ext^1_{A\hgmod}(M/M_{\geq d+1},L_\mu\gs{d})\\
                 &= \dim \Ext^1_{A\hgmod}(M/M_{\geq d+1},L_\mu\gs{d}) -\Ext^1_{A\hgmod}(M/M_{\geq 0},L_\mu\gs{d})\\
                &= \sum_{i=0}^d \dim \Ext^1_{A\hgmod}(M/M_{\geq i+1},L_\mu\gs{d}) -\dim \Ext^1_{A\hgmod}(M/M_{\geq i},L_\mu\gs{d})\\
                &\leq \sum_{i=0}^d \dim \Ext^1_{A\hgmod}(M_{\geq i}/M_{\geq i+1},L_\mu\gs{d})\\
                &= [M_{d-1}\otimes \chi_1:\mu].
        \end{align}
        Therefore, there exists $\lambda\in\irr{W}$ such that $[M_{d-1}:\lambda]\neq 0$ and $[\lambda\otimes \chi_1:\mu]\neq 0$.
        By Proposition \ref{prop:adj-sym}, we have $[\mu\otimes \chi_1:\lambda]=[\lambda\otimes \chi_1:\mu]\neq 0$.
    \end{proof}

    The main goal of this paper is to classify the families of modules $K=\{K_\lambda\}_{\lambda\in\irr{W}}$ and $\widetilde{K} = \{\widetilde{K}_\lambda\}_{\lambda\in\irr{W}}$ which satisfy following conditions:
        \begin{dfn}\label{dfn:setting-problem}\mbox{}
        \begin{enumerate}[\upshape(i)\itshape]
            \item \label{item:preorder-compatibility}
                There exists a preorder $\precsim$ on $\irr{W}$ which satisfies 
                \[
                    \widetilde{K}_\lambda \simeq P_\lambda \left/ \sum_{\mu\succ\lambda,f\in\Hom_{A}(P_\mu,P_\lambda)_{>0}}\mathrm{Im}\,f \right.,
                \]
                and
                \[
                    K_\lambda \simeq P_\lambda \left/ \sum_{\mu\succsim\lambda,f\in\Hom_{A}(P_\mu,P_\lambda)_{>0}}\mathrm{Im}\,f \right. .  
                \]
                In other word, $\widetilde{K}=\{\widetilde{K}_\lambda\}_{\lambda\in\irr{W}}$ is the trace quotiet modules and $K=\{K_\lambda\}_{\lambda\in\irr{W}}$ is the trace quotient modules with respect to $\precsim$.
            \item \label{item:orthogonality}
                For each $\lambda,\mu\in\irr{W}$, it holds that
                \begin{equation}\Ext_{A}^i(\widetilde{K}_\lambda,K_\mu^*)_d = \cc^{\delta_{i,0}\delta_{\lambda,\mu}\delta_{d,0}}.\end{equation}
            \item \label{item:PKtildefilter}
                For each $\lambda\in\irr{W}$, a projective module $P_\lambda$ is filtered by $\{\widetilde{K}_\mu\}_{\mu\in\irr{W}}$. 
            \item \label{item:KtildeKfilter}
                For each $\lambda\in\irr{W}$, a graded $A$-module $\widetilde{K}_\lambda$ is filtered by $\{K_\mu\}_{\mu\in\irr{W}}$.
        \end{enumerate}
        \end{dfn}

        By the conditions \eqref{item:preorder-compatibility},\eqref{item:orthogonality},\eqref{item:PKtildefilter}, and \eqref{item:KtildeKfilter}, a preorder $\precsim$ satisfies the setting of Section \ref{sec:reciprocity}.
        Due to Proposition \ref{prop:recoverpreorder}, we can replace the condition \eqref{item:preorder-compatibility} with following:\\
        \begin{enumerate}[(i)$'$]
            \item \label{item:precsimK-compatibility}
                For the preorder $\precsim_K$ defined in Definition \ref{dfn:precsimK}, it holds that
        \end{enumerate}
        \begin{align}
            \widetilde{K}_\lambda \simeq P_\lambda \left/ \sum_{\mu\succ_K\lambda,f\in\Hom_{A}(P_\mu,P_\lambda)_{>0}}\mathrm{Im}\,f \right. \textrm{ and}\\ 
            K_\lambda \simeq P_\lambda \left/ \sum_{\mu\succsim_K\lambda,f\in\Hom_{A}(P_\mu,P_\lambda)_{>0}}\mathrm{Im}\,f \right. .  
        \end{align}

        We call a pair of families of modules $K=\{K_\lambda\}_{\lambda\in\irr{W}}$ and $\widetilde{K} = \{\widetilde{K}_\lambda\}_{\lambda\in\irr{W}}$ a Springer correspondence if it satisfies the conditions \eqref{item:preorder-compatibility}, \eqref{item:orthogonality},\eqref{item:PKtildefilter}, and \eqref{item:KtildeKfilter} (or equivalently \eqref{item:precsimK-compatibility}$'$,\eqref{item:orthogonality},\eqref{item:PKtildefilter}, and \eqref{item:KtildeKfilter}).
        By abuse of terminology, we also say $K$ is a Springer correspondence if there exists $\widetilde{K}$ such that the pair $(K,\widetilde{K})$ is a Springer correspondence.
        By the condition \eqref{item:precsimK-compatibility}$'$, one can recover $\widetilde{K}$ from a Springer correspondence $K$.
        
        We define $D_\lambda \coloneqq \{\mu \mid \Ext^1_A(K_\lambda,L_\mu)\neq 0\}$ for each $\lambda\in\irr{W}$.
        The preorder $\precsim_K$ defined in Definition \ref{dfn:precsimK} is determined by the datum of $\{D_\lambda\}$.

        Let $\lambda\in \irr{W}$ and let $D=\{\lambda_1,\lambda_2,\ldots,\lambda_l\}$ be a subset of $\irr{W}$.  
        We define a graded $A$-module $F_\lambda^D$ and $F_\lambda^{\lambda_1,\lambda_2,\ldots,\lambda_l}$ by
        \[
            F_\lambda^D = F_\lambda^{\lambda_1,\lambda_2,\ldots,\lambda_l}\coloneqq P_\lambda \left/ \sum_{\mu \in\{\lambda_1,\lambda_2,\ldots,\lambda_l\}, f\in\Hom_A(P_\mu,P_\lambda)_{>0}} \im\,f\right. .
        \]
        If $K=\{K_\lambda\}_{\lambda\in\irr{W}}$ is a Springer correspondence, then $K_\lambda=F_\lambda^{D_\lambda}$ by Proposition \ref{prop:I-Ext1}.
        Since a preorder is a reflexive relation, we have $[(K_\lambda)_{>0}:L_\lambda]_q = 0$ for $\lambda\in\irr{W}$. 

        \begin{lem}
            \label{lem:F-gch-identify}
            Let $\lambda\in \irr{W}$ and $D \subset \irr{W}$.
            Let $M$ be a graded quotient module of $P_\lambda$.
            If $\gch M = \gch F_\lambda^D$ holds, then $M\simeq F_\lambda^D$.
        \end{lem}
        \begin{proof}
            Let $N$ be a graded submodule of $P_\lambda$ such that $M\simeq P_\lambda/ N$.
            It suffices to show that 
            \[
                \sum_{\mu \in D, f\in\Hom_A(P_\mu,P_\lambda)_{>0}} \im\,f = N.
            \]
            Since $M$ is a nonzero module, then $N$ is contained in $(P_\lambda)_{>0}$.
            For every $\mu \in D$, we have $[M_{>0}:L_\mu]_q = [(F_\lambda^D)_{>0}:L_\mu]_q = 0$.
            Therefore, $N$ contains the isotypic component of $(P_\lambda)_{>0}$ of type $\mu$ for every $\mu\in\irr{W}$.
            Hence, we obtain
            \[
                \sum_{\mu \in D, f\in\Hom_A(P_\mu,P_\lambda)_{>0}} \im\,f \subset N.
            \]
            By the definition of $F_\lambda^D$, we have
            \[
                \gch F_\lambda^D = \gch P_\lambda -\gch\biggl(\,\sum_{\mu \in D, f\in\Hom_A(P_\mu,P_\lambda)_{>0}} \im\,f\biggr).
            \]
            Consequently, we have
            \begin{align}
                \gch N = \gch P_\lambda - \gch M = \gch P_\lambda - \gch F_\lambda^D \\
                    = \gch(\sum_{\mu \in D, f\in\Hom_A(P_\mu,P_\lambda)_{>0}} \im\,f).
            \end{align}
            Hence, we obtain
            \[
                \sum_{\mu \in D, f\in\Hom_A(P_\mu,P_\lambda)_{>0}} \im\,f = N.
            \]
            as desired.
        \end{proof}

        \begin{lem}
            \label{lem:F_triv^sgn-resolution}
            A graded module $F_\triv^\sgn$ has a following projective resolution:
            \[    \xymatrix{
                    0 \ar[r] &P_\sgn\gs{n} \ar[rrr]^{\{b_\sgn \mapsto (X^n-Y^n)\otimes b_\triv \}} &&&P_\triv \ar[r] & F_\triv^\sgn \ar[r] &0.
                }
            \]
            In particular, we have
            \[
                \gch F_\triv^\sgn = \gch P_\triv -q^n\gch P_\sgn.
            \]
        \end{lem}
        \begin{proof}
            Let $M$ be the submodule of $P_\triv$ generated by $(X^n-Y^n)\otimes b_\triv$.
            Since $(X^n-Y^n)\otimes b_\triv$ is a homogeneous element of degree $n$ which belongs to the isotypic component of type $\sgn$, then $M$ is contained in $\sum_{f\in\Hom_A(P_\sgn,P_\triv)_{>0}}\im\, f$.
            Then, we have a following projective resolution:
            \begin{equation} \label{eq:res-F_triv^sgn}
                \xymatrix{
                    0 \ar[r] &P_\sgn\gs{n} \ar[r]^{p_1} &P_\triv \ar[r] & P_\triv/M \ar[r] &0,
                }
            \end{equation}
            where $p_0$ is a graded $A$-module homomorphism sending $b_\sgn \in P_\sgn\gs{n}$ to $(X^n-Y^n)\otimes b_\triv \in P_\triv$.
            Indeed, the image of $p_1$ is generated by $(X^n-Y^n)\otimes b_\triv$, and thus $\im\,p_1 = M$.
            Therefore, the cokernel of $p_1$ is $P_\triv/M$.
            An $A$-module homomorphism $p_1$ is also $S$-module homomorphism between $P_\sgn\gs{n}$ and $P_\triv$, which are free $S$-modules of rank $1$.
            Since $S$ is an integral domain, any $S$-moudule homomorphism from $S$ to itself is injective or zero.
            Thus $p_1$ is injective.
            Consequently, a chain complex \eqref{eq:res-F_triv^sgn} is exact.
            Therefore, we obtain $\gch P_\triv /M = \gch P_\triv - q^n \gch P_\sgn$.
            In particular, we have $[P_\triv/M:L_\sgn]_q = [P_\triv:L_\sgn]_q - q^n[P_\sgn:L_\sgn]_q$.
            By Proposition \ref{prop:gch-S} and Corollary \ref{cor:gch-Psgn}, we have $[P_\triv:L_\sgn]_q = q^n/(1-q^2)(1-q^n)$ and $[P_\sgn:L_\sgn]_q = 1/(1-q^2)(1-q^n)$.
            Hence, we obtain $[P_\triv/M:L_\sgn]_q = 0$.
            It follows that $M$ contains the isotypic component of $P_\triv$ corresponding to $\sgn$.
            Thus, we have 
            \[
                M = \sum_{f\in\Hom_A(P_\sgn,P_\triv)_{>0}}\im\, f. 
            \] 
            Hence, $P_\triv/M\simeq F_\triv^\sgn$.
            This completes the proof.
        \end{proof}

        \begin{lem}
            \label{lem:F_triv^j-resolution}
            Let $j$ be a positive integer such that $1\leq j\leq C$.
            There exists an exact sequence
            \[
                \xymatrix{
                    0 \ar[r] &P_\sgn\gs{2j} \ar[rrr]^{\{b_\sgn \mapsto Y^j\otimes b_j^+ -X^j\otimes b_j^- \}} &&&P_{\chi_j}\gs{j}\\
                    && \ar[rrr]^{\left\{\begin{matrix}b_j^+ \mapsto X^j\otimes b_\triv\\ b_j^- \mapsto Y^j\otimes b_\triv \end{matrix}\right\}}  &&&P_\triv \ar[r] & F_\triv^{\chi_j} \ar[r] &0.
                }  
            \]
            In particular, we have \[
                \gch F_\triv^{\chi_j} = \gch P_\triv -q^j \gch P_{\chi_j} + q^{2j} \gch P_\sgn. 
            \]  
        \end{lem}
        \begin{proof}
            Let $M$ be the submodule of $P_\triv$ generated by $X^j\otimes b_\triv$ and $Y^j\otimes b_\triv$.
            It is straightforward to see that these generators belong to the isotypic component of type $\chi_j$.
            Therefore, we have $M\subset  \sum_{f\in\Hom_A(P_{\chi_j},P_\triv)_{>0}}\im\, f$.
            We show that $P_\triv/M$ is isomorphic to $F_\triv^{\chi_j}$.
            A graded $A$-module $P_\triv/M$ has a following projective resolution:
            \begin{equation}
                \label{eq:res-F_triv^j}
                 \xymatrix{
                0 \ar[r] &P_\sgn\gs{2j} \ar[r]^{p_2} &P_{\chi_j}\gs{j} \ar[r]^{p_1}  &P_\triv \ar[r] & P_\triv/M \ar[r] &0,
                }
            \end{equation} 
            where $p_1$ sending $b_j^+$ and $b_j^-$ to $X^j\otimes b_\triv$ and $Y^j\otimes b_\triv$ respectively, and $p_2$ sending $b_\sgn$ to $Y^j\otimes b_j^+ -X^j\otimes b_j^-$.
            Since the image of $p_1$ is $M$, a chain complex
            \[
                \xymatrix{P_{\chi_j}\gs{j} \ar[r]^{p_1}  &P_\triv \ar[r] & P_\triv/M \ar[r] &0}
            \]
            is exact.
            We have bases of free $S$-modules $\{1\otimes b_\sgn\}\subset P_\sgn\gs{2j}$, $\{1\otimes b_j^+,1\otimes b_j^-\}\subset P_{\chi_j}\gs{j}$, and $\{1\otimes b_\triv\}\subset P_\triv$.
            Hence, $p_1$ and $p_2$ admit matrix representations with respect to the above bases as follows:
            \[  \xymatrix{
                P_\sgn\gs{2j} \ar[r]^{\begin{pmatrix}Y^j\\-X^j\end{pmatrix}} &P_{\chi_j}\gs{j} \ar[rr]^{\begin{pmatrix}X^j &Y^j\end{pmatrix}}  &&P_\triv.
            } 
            \]
            By Lemma \ref{lem:detinjective}, $p_2\simeq \begin{pmatrix}Y^j\\-X^j\end{pmatrix}$ is injective.
            Let $a$ and $b$ be polynomials such that $aX^j+ bY^j = 0$.
            Since $X^j$ does not have irreducible factor $Y$, thus $a$ is divisible by $Y^j$.
            Hence, we have $a = uY^j$ for some $u\in S$.
            Then we obtain $b=-uX^j$.
            Therefore, the kernel of $p_1\simeq \begin{pmatrix}X^j&Y^j\end{pmatrix}$ is generated by $(Y^j\otimes b_j^+ - X^j \otimes b_j^-)$.
            We obtain $\im\,p_2 = \ker\,p_1$, and hence, a chain complex \eqref{eq:res-F_triv^j} is exact.
            Therefore, we obtain 
            \[
                \gch P_\triv /M = \gch P_\triv - q^j \gch P_{\chi_j}+ q^{2j}\gch P_\sgn.
            \]
            By Proposition \ref{prop:gch-S} and Corollary \ref{cor:gch-Psgn}, we have $[P_\triv:L_{\chi_j}]_q = [P_\sgn:L_{\chi_j}]_q =(q^j+q^{n-j})/(1-q^2)(1-q^n)$. 
            By Corollary \ref{cor:gch-Pi}, we also have $[P_{\chi_j}:L_{\chi_j}]_q =(1+q^{n-2j}+q^{2j}+q^n)/(1-q^2)(1-q^n)$.
            Therefore, we have $[P_\triv/M:L_{\chi_j}]_q = 0$.
            It follows that $M$ contains the isotypic component of $P_\triv$ corresponding to $\chi_j$.
            Hence, we obtain 
            \[
                M = \sum_{f\in\Hom_A(P_{\chi_j},P_\triv)_{>0}}\im\, f.  
            \] 
            Therefore, we conclude $P_\triv/M\simeq F_\triv^{\chi_j}$.
            A projective resolution of $P_\triv/M$ is given as \eqref{eq:res-F_triv^j}.
        \end{proof}

        \begin{lem}\label{lem:F_sgn^sgn-resolution}
            A graded module $F_\sgn^\sgn$ has a following projective resolution:
            \[\xymatrix{
                0 \ar[r] &P_\sgn\gs{n+2} \ar[rrrr]^{\left\{b_\sgn\mapsto \begin{pmatrix}(X^n+Y^n)\otimes b_\sgn\\ -XY\otimes b_\sgn\end{pmatrix}\right\}} &&&&{\begin{matrix} P_\sgn\gs{2}\\ \oplus\\ P_\sgn\gs{n} \end{matrix}}\\ 
                \ar[rrrr]^{\begin{pmatrix}\{b_\sgn \mapsto XY\otimes b_\sgn\}\\ \{b_\sgn \mapsto (X^n+Y^n)\otimes b_\sgn\}\end{pmatrix}} &&&&P_\sgn \ar[r] &F_\sgn^\sgn \ar[r] &0.
            }\]
        \end{lem}
        \begin{proof}
            We have a following chain complex:            
            \begin{equation}\label{eq:res-F_sgn^sgn}\vcenter{\xymatrix{
                0 \ar[r] &P_\sgn\gs{n+2} \ar[rrr]^{\left\{b_\sgn\mapsto \begin{pmatrix}(X^n+Y^n)\otimes b_\sgn\\ -XY\otimes b_\sgn\end{pmatrix}\right\}} &&&{\begin{matrix} P_\sgn\gs{2}\\ \oplus\\ P_\sgn\gs{n} \end{matrix}}\\ 
                \ar[rrr]^{\begin{pmatrix}\{b_\sgn \mapsto XY\otimes b_\sgn\}\\ \{b_\sgn \mapsto (X^n+Y^n)\otimes b_\sgn\}\end{pmatrix}} &&&P_\sgn \ar[r] &P_\sgn/ \langle XY\otimes b_\sgn,(X^n+Y^n)\otimes b_\sgn \rangle_A\\
                \ar[r] &0.
            }}
            \end{equation}
            We show that this complex is exact.
            Clearly, the kernel of the quotient map $P_\sgn\rightarrow P_\sgn/ \langle XY\otimes b_\sgn,(X^n+Y^n)\otimes b_\sgn \rangle_A$ is $\langle XY\otimes b_\sgn,(X^n+Y^n)\otimes b_\sgn \rangle_A$.
            It is straightforward to see that the left map of the second row surjects onto $\langle XY\otimes b_\sgn,(X^n+Y^n)\otimes b_\sgn \rangle_A$.
            The map of the first row is a $S$-module homomorphism between free $S$-modules $P_\sgn\gs{n+2}$ and $P_\sgn\gs{2}\oplus P_\sgn\gs{n}$,
            and it admits a matrix representation
            \[
                \begin{pmatrix}
                    X^n+Y^n\\
                    -XY
                \end{pmatrix}.
            \]
            By Lemma \ref{lem:detinjective}, this morphism is injective.
            Next, we verify the exactness of following diagram:
            \[
                \xymatrix{P_\sgn\gs{n+2} \ar[rrrrr]^{\left\{b_\sgn\mapsto \begin{pmatrix}(X^n+Y^n)\otimes b_\sgn\\ -XY\otimes b_\sgn\end{pmatrix}\right\}} &&&&&{\begin{matrix} P_\sgn\gs{2}\\ \oplus\\ P_\sgn\gs{n} \end{matrix}}
                \ar[rrrrr]^{\begin{pmatrix}\{b_\sgn \mapsto XY\otimes b_\sgn\}\\ \{b_\sgn \mapsto (X^n+Y^n)\otimes b_\sgn\}\end{pmatrix}} &&&&&P_\sgn}.
            \]
            Since $S$ is an unique factorization domain and $XY,X^n+Y^n$ are coprime, then the kernel of the right map is generated by
            \[
                \begin{pmatrix}(X^n+Y^n)\otimes b_\sgn\\ -XY\otimes b_\sgn\end{pmatrix}.
            \]
            Consequently, a chain complex \eqref{eq:res-F_sgn^sgn} is exact.
            Therefore, we have
            \begin{multline}
                \gch P_\sgn/ \langle XY\otimes b_\sgn,(X^n+Y^n)\otimes b_\sgn \rangle_A\\
                    =\gch P_\sgn\gs{n+2} - \gch (P_\sgn\gs{2} \oplus P_\sgn\gs{n}) + \gch P_\sgn.
            \end{multline}
            By Proposition \ref{prop:tensorofreps} and Proposition \ref{prop:gch-S}, we obtain
            \begin{align}
                &\gch P_\sgn\gs{n+2} - \gch (P_\sgn\gs{2} \oplus P_\sgn\gs{n}) + \gch P_\sgn\\
                    &=(q^{n+2}-q^2-q^n+1)\gch P_\sgn\\
                    &=\frac{q^{n+2}-q^2-q^n+1}{(1-q^2)(1-q^n)}\cdot\left([L_\sgn]+q^n[L_\triv]+\sum_{i=1}^C(q^i+q^{n-i})[L_{\chi_i}]\right)\\
                    &=[L_\sgn]+q^n[L_\triv]+\sum_{i=1}^C(q^i+q^{n-i})[L_{\chi_i}].
            \end{align}
            Then, we have $[(P_\sgn/ \langle XY\otimes b_\sgn,(X^n+Y^n)\otimes b_\sgn \rangle_A)_{>0}:L_\sgn]_q = 0$.
            The isotypic component of $P_\sgn$ corresponding to $\sgn$ is contained in $\langle XY\otimes b_\sgn,(X^n+Y^n)\otimes b_\sgn\rangle_A$.
            Since $XY\otimes b_\sgn$ and $(X^n+Y^n)\otimes b_\sgn$ belong to the isotypic component of type $\sgn$, we obtain
            \[
                \langle XY\otimes b_\sgn,(X^n+Y^n)\otimes b_\sgn \rangle_A = \sum_{f\in\Hom_{A}(P_\sgn,P_\sgn)_{>0}}\mathrm{Im}\,f.
            \]
            Hence, we conclude that $P_\sgn/ \langle XY\otimes b_\sgn,(X^n+Y^n)\otimes b_\sgn \rangle_A = F_\sgn^\sgn$.
            Its projective resolution is given as \eqref{eq:res-F_sgn^sgn}.
        \end{proof}

        \begin{lem}
            \label{lem:F_i^sgn-resolution}
            For $1\leq i \leq C$, a graded module $F_{\chi_i}^\sgn$ has a following projective resolution: 
            \[
                \xymatrix{
                    0 \ar[r] &{\begin{matrix}P_\sgn\gs{i}\\ \oplus\\ P_\sgn\gs{n-i} \end{matrix}} \ar[rrrrrr]^{\begin{pmatrix}\{b_\sgn \mapsto Y^i\otimes b_i^+ - X^i\otimes b_i^- \}\\ \{b_\sgn \mapsto X^{n-i}\otimes b_i^+ - Y^{n-i}\otimes b_i^- \}\end{pmatrix}} &&&&&&P_{\chi_i} \ar[r] & F_{\chi_i}^\sgn \ar[r] &0.
                }   
            \]
            In particular, we have
            \[
                \gch F_{\chi_i}^\sgn = \gch P_{\chi_i} -(q^i+q^{n-i})\gch P_\sgn \qquad(1\leq i\leq C).
            \]
                
        \end{lem}
        \begin{proof}
            Let $M$ be the submodule of $P_{\chi_i}$ generated by $(Y^i\otimes b_i^+ - X^i\otimes b_i^-)$ and $(X^{n-i}\otimes b_i^+ - Y^{n-i}\otimes b_i^-)$.
            These generators belong to the isotypic component of type $\sgn$.
            Thus $M$ is contained in $\sum_{f\in\Hom_A(P_{\sgn},P_{\chi_i})_{>0}}\im\, f$.
            We show that $P_{\chi_i}/M$ is isomorphic to $F_{\chi_i}^{\sgn}$.
            A graded $A$-module $P_\triv/M$ has a following projective resolution:
            \begin{equation}  \label{eq:res-F_i^sgn}
                \xymatrix{
                    0 \ar[r] &{\begin{matrix}P_\sgn\gs{i}\\ \oplus\\ P_\sgn\gs{n-i} \end{matrix}} \ar[r]^{p_1} &P_{\chi_i} \ar[r] & F_{\chi_i}^\sgn \ar[r] &0.
                }   
            \end{equation}
            where $p_1$ sending $1\otimes b_\sgn \in P_\sgn\gs{i}$ to $Y^i\otimes b_i^+ - X^i\otimes b_i^-$, and $1\otimes b_\sgn \in P_\sgn\gs{n-i}$ to $X^{n-i}\otimes b_i^+ - Y^{n-i}\otimes b_i^-$.
            Since the image of $p_1$ is $M$, a complex
            \[
                \xymatrix{
                    {\begin{matrix}P_\sgn\gs{i}\\ \oplus\\ P_\sgn\gs{n-i} \end{matrix}} \ar[r]^{p_1}  &P_\triv \ar[r] &P_\triv/M \ar[r] &0
                }
            \]
            is exact.
            We prove that  $p_1$ is injective.
            We have bases of free $S$-modules $\{(1\otimes b_\sgn,0),(0,1\otimes b_\sgn)\}\subset P_\sgn\gs{i} \oplus P_\sgn\gs{n-i}$ and  $\{1\otimes b_i^+,1\otimes b_i^-\}\subset P_{\chi_i}$.
            Then, $p_1$ have a matrix representation $\begin{pmatrix}Y^i & X^{n-i}\\-X^i &Y^{n-i}\end{pmatrix}$ with respect to the above bases.
            Since the determinant of this matrix is $X^n+Y^n\neq 0$, a homomorphism $p_1$ is injective by Lemma \ref{lem:detinjective}.
            Consequently, a chain complex \eqref{eq:res-F_i^sgn} is exact.
            Therefore, we obtain $[P_{\chi_i} /M:L_\sgn]_q = [P_{\chi_i}:L_\sgn]_q - (q^i+q^{n-i}) [P_\sgn:L_\sgn]_q$.
            By Corollary \ref{cor:gch-Psgn}, we have $[P_\sgn:L_\sgn]_q = 1/(1-q^2)(1-q^n)$. 
            By Corollary \ref{cor:gch-Pi}, we also have $[P_{\chi_i}:L_\sgn]_q =(q^i+q^{n-i})/(1-q^2)(1-q^n)$.
            Therefore, we have $[P_{\chi_i}/M:L_\sgn]_q = 0$.
            It follows that $M$ contains the isotypic component of $P_{\chi_i}$ corresponding to $\sgn$.
            Hence, we obtain 
            \[
                M = \sum_{f\in\Hom_A(P_{\sgn},P_{\chi_i})_{>0}}\im\, f. 
            \] 
            Therefore $P_{\chi_i}/M\simeq F_{\chi_i}^{\sgn}$.
            Its projective resolution is given as \eqref{eq:res-F_i^sgn}.    
        \end{proof}

        \begin{lem}
            \label{lem:F_i^sgn^j-resolution}
            Let $i,j$ be positive integers such that $1\leq i < j \leq C$.
            Then, a graded module $F_{\chi_i}^{\sgn,\chi_j}$ has a following projective resolution:
            \[
                \xymatrix{
                    0 \ar[r] &P_\sgn\gs{2j-i} \ar[rrrrr]^{\left\{b_\sgn \mapsto \begin{pmatrix}(XY)^{j-i}\otimes b_\sgn\\ -Y^j\otimes b_j^+ + X^j \otimes b_j^-\end{pmatrix}\right\}} &&&&&{\begin{matrix}P_\sgn\gs{i}\\ \oplus\\ P_{\chi_j}\gs{j-i} \end{matrix}}\\ 
                    \ar[rrrrr]^{\begin{pmatrix}\{b_\sgn \mapsto Y^i\otimes b_i^+ - X^i\otimes b_i^- \}\\ \left\{\begin{matrix} b_j^+ \mapsto X^{j-i}\otimes b_i^+\\ b_j^- \mapsto Y^{j-i}\otimes b_i^- \end{matrix}\right\}\end{pmatrix}} &&&&&P_{\chi_i} \ar[r] & F_{\chi_i}^{\sgn,\chi_j} \ar[r] &0.
                }   
            \]
            In particular, we have
            \[
                \gch F_{\chi_i}^{\sgn, \chi_j} = \gch P_{\chi_i} - q^i \gch P_\sgn -q^{j-i}\gch P_{\chi_j} + q^{2j-i} \gch P_\sgn.
            \]
        \end{lem}
        \begin{proof}
            Let $M$ be the submodule of $P_{\chi_i}$ generated by $X^{j-i}\otimes b_i^+, Y^{j-i}\otimes b_i^-$, and $Y^i\otimes b_i^+ - X^i\otimes b_i^-$.
            The vectors $X^{j-i}\otimes b_i^+$ and $Y^{j-i}\otimes b_i^-$ belong to the isotypic component of type $\chi_j$, and $(Y^i\otimes b_i^+ - X^i\otimes b_i^-)$ belongs to the component of type $\sgn$.
            Thus, we have $M$ is contained in 
            \[\sum_{f\in\Hom_A(P_{\sgn},P_{\chi_i})_{>0}}\im\, f+\sum_{f\in\Hom_A(P_{\chi_j},P_{\chi_i})_{>0}}\im\, f.\]
            We show that $P_{\chi_i}/M$ is isomorphic to $F_{\chi_i}^{\sgn,\chi_j}$.
            There exists a following chain complex:
            \begin{equation}\label{eq:res-F_i^sgn^j}
            \xymatrix{
                0 \ar[r] &P_\sgn\gs{2j-i} \ar[r]^{p_2} &{\begin{matrix}P_\sgn\gs{i}\\ \oplus\\ P_{\chi_j}\gs{j-i} \end{matrix}} \ar[r]^{p_1} &P_{\chi_i} \ar[r] & P_{\chi_i}/M \ar[r] &0,
            }   
            \end{equation}
            where $p_1$ and $p_2$ given by
            \begin{align}
                p_1 :\begin{matrix} b_\sgn &\mapsto &Y^i\otimes b_i^+ - X^i\otimes b_i^- \\
                                    b_j^+  & \mapsto &X^{j-i}\otimes b_i^+\\ 
                                    b_j^-  & \mapsto &Y^{j-i}\otimes b_i^-
                                \end{matrix} \textrm{ and } 
                p_2: b_\sgn \mapsto \begin{pmatrix}(XY)^{j-i}\otimes b_\sgn\\ -Y^j\otimes b_j^+ + X^j \otimes b_j^-\end{pmatrix}.
              \end{align}
              Since the image of $p_1$ is $M$, a chain complex
            \[
                \xymatrix{{\begin{matrix}P_\sgn\gs{i}\\ \oplus\\ P_{\chi_j}\gs{j-i} \end{matrix}} \ar[r]^{p_1} &P_{\chi_i} \ar[r] & P_{\chi_i}/M \ar[r] &0,
                }
            \]
            is exact.
            We have bases of free $S$-modules $\{1\otimes b_\sgn\}\subset P_\sgn\gs{2j-i}$, $\{1\otimes b_\sgn, 1\otimes b_j^+,1\otimes b_j^-\}\subset P_\sgn\gs{i}\oplus P_{\chi_j}\gs{j-i}$, and $\{1\otimes b_i^+,1\otimes b_i^- \}\subset P_{\chi_i}$.
            Then, $p_1$ and $p_2$ admit matrix representations with respect to the above bases as follows:
            \[  \xymatrix{
                P_\sgn\gs{2j-i} \ar[rr]^{\begin{pmatrix}(XY)^{j-i}\\ -Y^j\\X^j \end{pmatrix}} &&{\begin{matrix}P_\sgn\gs{i}\\ \oplus\\ P_{\chi_j}\gs{j-i} \end{matrix}} \ar[rrrr]^{\begin{pmatrix}Y^i &X^{j-i} & 0\\ -X^i &0 & Y^{j-i}\end{pmatrix}} &&&&P_{\chi_i}
            }\]
            By Lemma \ref{lem:detinjective}, $p_2\simeq \begin{pmatrix}Y^j\\-X^j\end{pmatrix}$ is injective.
            Let $a,b$ and $c$ be polynomials such that $aY^i+bX^{j-i}= -aX^i+cY^{j-i}=0$.
            Since we have $aY^i=-bX^{j-i}$ and $Y^i$ does not have irreducible factor $X$, thus $a$ is divisible by $X^{j-i}$.
            Since we also have $aX^i=cY^{j-i}$, $a$ is divisible by $Y^{j-i}$.
            Hence, we have $a = u(XY)^{j-i}$ for some $u\in S$.
            Then we obtain $b=-uY^j$ and $c=uX^j$.
            Therefore, the kernel of $p_1\simeq \begin{pmatrix}Y^i &X^{j-i} & 0\\ -X^i &0 & Y^{j-i}\end{pmatrix}$ is generated by $\begin{pmatrix}(XY)^{j-i}\otimes b_\sgn\\ -Y^j\otimes b_j^+ + X^j \otimes b_j^-\end{pmatrix}$.
            We obtain $\im\,p_2 = \ker\,p_1$, and hence, a chain complex \eqref{eq:res-F_i^sgn^j} is exact.
            Therefore, we obtain 
            \[
                \gch P_{\chi_i} /M = \gch P_{\chi_i} + (q^{2j-i}-q^i)\gch P_\sgn  -q^{j-i} \gch P_{\chi_j}.
            \] 
            
            By Lemma \ref{lem:F_i^sgn-resolution}, we have $[P_{\chi_i}:L_\sgn]_q-(q^i+q^{n-i})[P_\sgn:L_\sgn]_q = [P_{\chi_j}:L_\sgn]_q-(q^j+q^{n-j})[P_\sgn:L_\sgn]_q=0$.
            Therefore, we obtain 
            \begin{align}
                [P_{\chi_i} /M:L_{\sgn}]_q&=[P_{\chi_i}:L_\sgn]_q + (q^{2j-i}-q^i)[P_\sgn:L_\sgn]_q  -q^{j-i} [P_{\chi_j}:L_\sgn]_q\\
                    &= [P_{\chi_i}:L_\sgn]_q-(q^i+q^{n-i})[P_\sgn:L_\sgn]_q\\
                    &\qquad -q^{j-i}([P_{\chi_j}:L_\sgn]_q-(q^j+q^{n-j})[P_\sgn:L_\sgn]_q)\\
                    &=0.
            \end{align}
            By Corollary \ref{cor:gch-Psgn}, we have $[P_\sgn:L_{\chi_j}]_q =(q^j+q^{n-j})/(1-q^2)(1-q^n)$. 
            By Corollary \ref{cor:gch-Pi}, we also have $[P_{\chi_i}:L_{\chi_j}]_q =(q^{j-i}+q^{n-i-j}+q^{i+j}+q^{n+i-j})/(1-q^2)(1-q^n)$ and $[P_{\chi_j}:L_{\chi_j}]_q =(1+q^{n-2j}+q^{2j}+q^n)/(1-q^2)(1-q^n)$.
            Therefore, we have 
            \begin{align}
                [P_\triv/M:L_{\chi_j}]_q &= [P_{\chi_i}:L_{\chi_j}]_q + (q^{2j-i}-q^i)[P_\sgn:L_{\chi_j}]_q  -q^{j-i} [P_{\chi_j}:L_{\chi_j}]_q\\
                    &=\frac{(q^{j-i}+q^{n-i-j}+q^{i+j}+q^{n+i-j})}{(1-q^2)(1-q^n)}+\frac{(q^{2j-i}-q^i)(q^j+q^{n-j})}{(1-q^2)(1-q^n)}\\
                        &\qquad - \frac{q^{j-i}(1+q^{n-2j}+q^{2j}+q^n)}{(1-q^2)(1-q^n)}\\
                    &=0.
                \end{align}
            It follows that $M$ contains the isotypic components of $P_{\chi_i}$ corresponding to $\chi_j$ and $\sgn$.
            Hence, we obtain 
            \[
                M = \sum_{f\in\Hom_A(P_{\sgn},P_{\chi_i})_{>0}}\im\, f+\sum_{f\in\Hom_A(P_{\chi_j},P_{\chi_i})_{>0}}\im\, f.  
            \] 
            Therefore $P_\triv/M\simeq F_\triv^{\sgn,\chi_j}$.
            A projective resolution of $P_\triv/M$ is given as \eqref{eq:res-F_i^sgn^j}.
        \end{proof}
    
        \begin{lem}
            \label{lem:F_C^sgn^C-resolution}
            A graded $A$-module $F_{\chi_C}^{\sgn,\chi_C}$ has a following projective resolution:
            \[
                \xymatrix{
                    0 \ar[r] &P_\sgn\gs{C+2} \ar[rrrrr]^{\left\{b_\sgn \mapsto \begin{pmatrix}XY\otimes b_\sgn\\ X^{C+1}\otimes b_C^+ - Y^{C+1} \otimes b_C^-\end{pmatrix}\right\}} &&&&&{\begin{matrix}P_\sgn\gs{C}\\ \oplus\\ P_{\chi_C}\gs{1} \end{matrix}}\\ 
                    \ar[rrrrr]^{\begin{pmatrix}\{b_\sgn \mapsto Y^C\otimes b_C^+ - X^C\otimes b_C^- \}\\ \left\{\begin{matrix} b_C^+ \mapsto Y\otimes b_C^-\\ b_C^- \mapsto X\otimes b_C^+ \end{matrix}\right\}\end{pmatrix}} &&&&&P_{\chi_C} \ar[r] & F_{\chi_C}^{\sgn,\chi_C} \ar[r] &0.
                }   
            \]
            In particular, we have 
            \[
                \gch F_{\chi_C}^{\sgn,\chi_C} = (1-q)\gch P_{\chi_C} + (-q^C+q^{C+2})\gch P_\sgn.
            \]
        \end{lem}
        \begin{proof}
            We set $i = C=(n-1)/2$ and $j= C+1=(n+1)/2$. 
            Let $M$ be the submodule of $P_{\chi_C}$ generated by $X\otimes b_C^+, Y\otimes b_C^-$, and $(Y^C\otimes b_C^+ - X^C\otimes b_C^-)$.
            Then $M$ has a projective resolution of the same form to \eqref{eq:res-F_i^sgn^j} in the proof of Lemma \ref{lem:F_i^sgn^j-resolution}.
            Hence, we obtain $\gch P_{\chi_C}/M = \gch P_{\chi_C} + (-q^C+q^{C+2})\,\gch P_\sgn - q\,\gch P_{\chi_{C+1}}$.
            By Corollary \ref{cor:gch-Psgn}, we have $[P_\sgn:L_\sgn]_q = 1/(1-q^2)(1-q^n)$ and $[P_\sgn:L_{\chi_C}]_q = (q^C+q^{C+1})/(1-q^2)(1-q^n)$.
            By Lemma \ref{lem:gch-diff}, we have $\gch P_{\chi_C}- q\,\gch P_{\chi_{C+1}} = 1/(1-q^n)\sum_{k=0}^{n-1} q^k\,\gch L_{\chi_{i-k}}$.
            Therefore, we obtain $[P_{\chi_C}:L_\sgn]_q- q\,[P_{\chi_{C+1}}:L_\sgn]_q = q^C/(1-q^n)$ and $[P_{\chi_C}:L_{\chi_C}]_q- q\,[P_{\chi_{C+1}}:L_{\chi_C}]_q = (1+q^{2C})/(1-q^n)$.
            Consequently, we have 
            \[
                [P_{\chi_C}/M:L_\sgn]_q= q^C/(1-q^n)+ (-q^C+q^{C+2})/(1-q^2)(1-q^n) = 0
            \] 
            and
            \begin{align}
                [P_{\chi_C}/M:L_{\chi_C}]_q &=  (1+q^{2C})/(1-q^n) + (-q^C+q^{C+2})(q^C+q^{C+1})/(1-q^2)(1-q^n)\\ 
                &= (1+q^{2C})/(1-q^n) + (-q^{2C} - q^{2C+1})/(1-q^n) \\
                &= (1-q^n)/(1-q^n) =1.    
            \end{align}
            Therefore, $M$ contains the isotypic components of $(P_{\chi_C})_{>0}$ corresponding to $\sgn$ and $\chi_C$.
            It follows that 
            \[
                M= \sum_{f\in\Hom_A(P_{\sgn},P_{\chi_C})_{>0}}\im\, f+\sum_{f\in\Hom_A(P_{\chi_C},P_{\chi_C})_{>0}}\im\, f,
            \]
            and hence, $P_{\chi_C}/M$ is isomorphic to $F_{\chi_C}^{\sgn,\chi_C}$.
            Its projective resolution is given as \eqref{eq:res-F_i^sgn^j}.
        \end{proof}

        \begin{cor}
            \label{cor:gch-standard-K_i}
            Let $i$ be an integer such that $1\leq i\leq C$.
            Then we have
            \[
                \gch F_{\chi_i}^{\sgn,\chi_{i+1}} = q^i[L_\triv] +\sum_{j=1}^i q^{i-j}[L_{\chi_j}].
            \]
        \end{cor}
        \begin{proof}
            When $i<C$, we have $\gch F_{\chi_i}^{\sgn, \chi_{i+1}} = \gch P_{\chi_i} - q^i \gch P_\sgn -q\,\gch P_{\chi_{i+1}} + q^{i+2} \gch P_\sgn$ by Lemma \ref{lem:F_i^sgn^j-resolution}.
            When $i=C$, we have $\chi_C\simeq \chi_{C+1}$ by Proposition \ref{prop:relationofreps}.
            Therefore, we have $F_{\chi_C}^{\sgn,\chi_{C+1}} = F_{\chi_C}^{\sgn,\chi_C}$ and $\gch P_{\chi_C} = \gch P_{\chi_{C+1}}$.
            By Lemma \ref{lem:F_C^sgn^C-resolution}, we obtain 
            \begin{align}
                \gch F_{\chi_C}^{\sgn,\chi_C} &= (1-q^1)\gch P_{\chi_C} + (-q^C-q^{C+2})\gch P_\sgn\\
                &=\gch P_{\chi_C} -q^1 \gch P_{\chi_{C+1}} + (-q^C-q^{C+2})\gch P_\sgn.
            \end{align}
            Hence, regardless $i=C$ or not, we have 
            \begin{equation}\label{eq:gch-F_i^sgn^i+1}
                \gch F_{\chi_i}^{\sgn, \chi_{i+1}} = \gch P_{\chi_i} -q\,\gch P_{\chi_{i+1}}  + (-q^i+q^{i+2})\,\gch P_\sgn.    
            \end{equation}

            By Proposition \ref{prop:relationofreps} and Lemma \ref{lem:gch-diff}, we obtain 
            \begin{align}
                \gch& P_{\chi_i}-q\,\gch P_{\chi_{i+1}} = \frac{1}{1-q^n}\sum_{j=0}^{n-1} q^j\cdot\gch L_{\chi_{i-j}}\\
                &=\frac{1}{1-q^n} \left( \sum_{j=0}^{i-1} q^j\cdot[L_{\chi_{i-j}}]+ q^i[L_\triv]+ q^i[L_\sgn] + \sum_{j=1}^{n-i-1} q^{i+j}\cdot[ L_{\chi_{j}}] \right).
            \end{align}
            By Corollary \ref{cor:gch-Psgn}, we have 
            \begin{align}
                - q^i &\gch P_\sgn +  q^{i+2} \gch P_\sgn = \frac{-q^i}{(1-q^n)}\cdot\left([L_\sgn]+q^n[L_\triv]+\sum_{j=1}^{n-1}q^j[L_{\chi_j}]\right)\\
                    &=-\frac{1}{(1-q^n)}\cdot\left(q^i[L_\sgn]+q^{n+i}[L_\triv]+\sum_{j=1}^{n-1}q^{i+j}[L_{\chi_j}]\right)\\
                    &=-\frac{1}{(1-q^n)}\cdot\left(q^i[L_\sgn]+q^{n+i}[L_\triv]+\sum_{j=1}^{n-i-1}q^{i+j}[L_{\chi_j}]+\sum_{j=0}^{i-1}q^{n+j}[L_{\chi_{j-i}}]\right).\\
            \end{align}
            By \eqref{eq:gch-F_i^sgn^i+1}, we obtain
            \begin{align}
                \gch &F_{\chi_i}^{\sgn, \chi_{i+1}} = \gch P_{\chi_i} -q\,\gch P_{\chi_{i+1}}  + (-q^i+q^{i+2})\,\gch P_\sgn\\
                    &= \frac{1}{1-q^n} \left( \sum_{j=0}^{i-1} q^j\cdot[L_{\chi_{i-j}}]+ q^i[L_\triv]+ q^i[L_\sgn] + \sum_{j=1}^{n-i-1} q^{i+j}\cdot[ L_{\chi_{j}}] \right)\\
                    &\qquad -\frac{1}{1-q^n}\left(q^i[L_\sgn]+q^{n+i}[L_\triv]+\sum_{j=1}^{n-i-1}q^{i+j}[L_{\chi_j}]+\sum_{j=0}^{i-1}q^{n+j}[L_{\chi_{j-i}}]\right)\\
                    &=\frac{1}{1-q^n}\cdot\left((q^i-q^{n+i})[L_\triv]+\sum_{j=0}^{i-1} (q^j-q^{n+j})\cdot[L_{\chi_{i-j}}]\right)\\
                    &= q^i[L_\triv]+\sum_{j=0}^{i-1} q^j\cdot[L_{\chi_{i-j}}] = q^i[L_\triv]+\sum_{j=1}^{i} q^{i-j}\cdot[L_{\chi_{j}}].
            \end{align}
        \end{proof}

        \begin{prop}
            \label{prop:F_triv^sgn-filt-F_triv^j-F_j-sgn}
            Let $j$ be a positive integer such that $1\leq j\leq C$.
            We have
            \[
                F_\triv^\sgn \filt [F_\triv^{\chi_j}] + q^j [F_{\chi_j}^\sgn].    
            \]
        \end{prop}
        \begin{proof}
            By Lemma \ref{lem:F_triv^sgn-resolution}, we have $F_\triv^\sgn = P_\triv/\langle (X^n-Y^n)\otimes b_\triv\rangle_A$ and $\gch F_\triv^\sgn = \gch P_\triv -q^n \gch P_\sgn$.
            Let $f$ be a graded $A$-module homomorphism from $P_{\chi_j}\gs{j}$ to $F_\triv^\sgn$ given by
            \[  
            \begin{matrix}
                f: &P_{\chi_j}\gs{j} &\rightarrow &F_\triv^\sgn\\
                    &1\otimes b_j^+ &\mapsto & X^j \otimes b_\triv\\
                    &1\otimes b_j^- &\mapsto & Y^j \otimes b_\triv
            \end{matrix}.
            \]
            Then we have $F_\triv^\sgn/\im\, f = P_\triv/\langle (X^n-Y^n)\otimes b_\triv, X^j\otimes b_\triv, Y^j \otimes b_\triv\rangle_A$.
            Clearly, $(X^n-Y^n)\otimes b_\triv = X^{n-j}\cdot X^j\otimes b_\triv + Y^{n-j}\cdot Y^j\otimes b_\triv$ belongs to $\langle X^j\otimes b_\triv, Y^j \otimes b_\triv\rangle_A$.
            Therefore, we obtain $F_\triv^\sgn/\im\, f \simeq P_\triv/\langle X^j\otimes b_\triv, Y^j \otimes b_\triv\rangle_A$.
            By Lemma \ref{lem:F_triv^j-resolution}, we have $P_\triv/\langle X^j\otimes b_\triv, Y^j \otimes b_\triv\rangle_A \simeq F_\triv^{\chi_j}$, and 
            \[
                \gch F_\triv^{\chi_j} = \gch P_\triv -q^j \gch P_{\chi_j} + q^{2j}\gch P_\sgn.
            \]
            Hence, we obtain
            \begin{align}
                \gch \im\, f &= \gch F_\triv^\sgn -\gch F_\triv^\sgn/\im\, f\\
                    &= \gch F_\triv^\sgn -\gch \gch F_\triv^{\chi_j}\\
                    &= (\gch P_\triv -q^n \gch P_\sgn) - (\gch P_\triv -q^j \gch P_{\chi_j} + q^{2j}\gch P_\sgn)\\
                    &=  q^j(\gch P_{\chi_j} -(q^j+q^{n-j})\gch P_\sgn). 
            \end{align}
            By Lemma \ref{lem:F_i^sgn-resolution}, we have $\gch F_{\chi_j}^\sgn\gs{j} = q^j(\gch P_{\chi_j} -(q^j+q^{n-j})\gch P_\sgn)$.
            Since $\im\, f $ is a graded quotient module of $P_{\chi_j}\gs{j}$, we obtain $\im\, f \simeq \gch F_{\chi_j}^\sgn\gs{j}$ by Lemma \ref{lem:F-gch-identify}.
            Consequently, there exists a filtration
            \[
                F_\triv^\sgn \supset \im\,f \supset 0
            \]
            which have two factors $\im\,f \simeq \gch F_{\chi_j}^\sgn\gs{j}$ and $F_\triv^\sgn/\im\, f \simeq F_\triv^{\chi_j}$ as desired.
        \end{proof}

        \begin{prop}
            \label{prop:F_i^sgn-filt-F_i^sgn^j-F_j^sgn}
            Let $i,j$ be positive integers such that $1\leq i \leq j \leq C$.
            Then we have
            \[
                F_{\chi_i}^\sgn \filt [F_{\chi_i}^{\sgn,\chi_j}] + q^{j-i} [F_{\chi_j}^{\sgn}].    
            \]
        \end{prop}
        \begin{proof}
            By Lemma \ref{lem:F_i^sgn-resolution}, we have $F_{\chi_i}^\sgn = P_{\chi_i}/\langle Y^i\otimes b_i^+ - X^i \otimes b_i^-,X^{n-i}\otimes b_i^+ - Y^{n-i}\otimes b_i^- \rangle_A$.
            Let $f$ be a graded $A$-module homomorphism given by
            \[
            \begin{matrix}
                f: &P_{\chi_j}\gs{j-i} &\rightarrow &F_{\chi_i}^\sgn\\
                    &1\otimes b_j^+ &\mapsto & X^{j-i} \otimes b_i^+\\
                    &1\otimes b_j^- &\mapsto & Y^{j-i} \otimes b_i^-
            \end{matrix}.
            \]
            Then we have 
            \[
                F_{\chi_i}^\sgn / \im\,f \simeq P_{\chi_i} /\langle Y^i\otimes b_i^+ - X^i \otimes b_i^-,X^{n-i}\otimes b_i^+ - Y^{n-i}\otimes b_i^- ,X^{j-i} \otimes b_i^+,Y^{j-i} \otimes b_i^-\rangle_A.
            \]
            Since $X^{n-i}\otimes b_i^+ - Y^{n-i}\otimes b_i^-$ belongs to $\langle X^{j-i} \otimes b_i^+,Y^{j-i} \otimes b_i^- \rangle_A$, we obtain $ F_{\chi_i}^\sgn / \im\,f \simeq \langle Y^i\otimes b_i^+ - X^i \otimes b_i^-,X^{j-i} \otimes b_i^+,Y^{j-i} \otimes b_i^-\rangle_A$.
            By Lemma \ref{lem:F_i^sgn^j-resolution}, we find $P_{\chi_i}/\langle Y^i\otimes b_i^+ - X^i \otimes b_i^-,X^{j-i} \otimes b_i^+,Y^{j-i} \otimes b_i^-\rangle_A$ is isomorphic to $F_{\chi_i}^{\sgn,\chi_j}$.
            Therefore, we have $\gch \im\, f = \gch F_{\chi_i}^\sgn - \gch F_{\chi_i}^{\sgn,\chi_j}$.
            Again by Lemma \ref{lem:F_i^sgn^j-resolution}, we have 
            \[
                \gch F_{\chi_i}^{\sgn, \chi_j} = \gch P_{\chi_i} - q^i \gch P_\sgn -q^{j-i}\gch P_{\chi_j} + q^{2j-i} \gch P_\sgn.
            \]
            By Lemma \ref{lem:F_i^sgn-resolution}, we also have 
            \[    
                \gch F_{\chi_i}^\sgn = \gch P_{\chi_i} -(q^i+q^{n-i})\gch P_\sgn.
            \]
            Hence, we obtain
            \begin{align}
                \gch \im\, f &= (\gch P_{\chi_i} -(q^i+q^{n-i})\gch P_\sgn) \\
                            &\qquad- (\gch P_{\chi_i} - q^i \gch P_\sgn -q^{j-i}\gch P_{\chi_j} + q^{2j-i} \gch P_\sgn)\\
                &= q^{j-i}\gch P_{\chi_j} - q^{2j-i} \gch P_\sgn -q^{n-i}\gch P_\sgn\\
                &= q^{j-i}(\gch P_{\chi_j} - q^{j} \gch P_\sgn -q^{n-j}\gch P_\sgn ).
            \end{align}
            Since $\gch F_{\chi_j}^\sgn = \gch P_{\chi_j} - q^{j} \gch P_\sgn -q^{n-j}\gch P_\sgn $ by Lemma \ref{lem:F_i^sgn-resolution}, we obtain $\gch \im\, f = \gch F_{\chi_j}^\sgn\gs{j-i}$.
            Clearly, $\im\, f$ is a graded quotient module of $P_{\chi_j}\gs{j-i}$.
            Thus we have $\im\, f \simeq F_{\chi_j}^\sgn\gs{j-i}$ by Lemma \ref{lem:F-gch-identify}.
            Consequently, there exists a filtration
            \[
                F_{\chi_i}^\sgn \supset \im\,f \supset 0 
            \]
            with desired factors.
        \end{proof}

        \begin{prop}
            \label{prop:F_C^sgn-filt-F_C^sgn^C}
            We have
            \[
                F_{\chi_C}^\sgn \filt \frac{1}{1-q} [F_{\chi_C}^{\sgn,\chi_C}].
            \]
        \end{prop}
        \begin{proof}
            By Lemma \ref{lem:F_i^sgn-resolution}, we have $F_{\chi_C}^\sgn = P_{\chi_C}/\langle Y^C\otimes b_C^+ - X^C \otimes b_C^-,X^{n-C}\otimes b_C^+ - Y^{n-C}\otimes b_C^- \rangle_A$ and $\gch F_{\chi_C}^\sgn = \gch P_{\chi_C} -(q^C+q^{n-C})\gch P_\sgn$.
            Let $f$ be a graded $A$-module homomorphism given by
            \[
            \begin{matrix}
                f: &P_{\chi_C}\gs{1} &\rightarrow &F_{\chi_C}^\sgn\\
                    &1\otimes b_C^+ &\mapsto & Y \otimes b_C^-\\
                    &1\otimes b_C^- &\mapsto & X \otimes b_C^+
            \end{matrix}.
            \]
            Then we have 
            \[F_{\chi_C}^\sgn /\im\,f \simeq P_{\chi_C}/\bigl(\langle Y^C\otimes b_C^+ - X^C \otimes b_C^-,X^{n-C}\otimes b_C^+ - Y^{n-C}\otimes b_C^-\rangle_A + \langle Y \otimes b_C^-, X \otimes b_C^+ \rangle_A \bigr).\]
            Since $C=(n-1)/2$, we have $n-C = (n+1)/2 \geq 1$.
            Thus, $X^{n-C}\otimes b_C^+ - Y^{n-C}\otimes b_C^-$ belongs to $\langle Y \otimes b_C^-, X \otimes b_C^+ \rangle_A$.
            Therefore, we have $F_{\chi_C}^\sgn /\im\,f \simeq P_{\chi_C}/\langle Y^C\otimes b_C^+ - X^C \otimes b_C^-, Y \otimes b_C^-, X \otimes b_C^+ \rangle_A$.
            By Lemma \ref{lem:F_C^sgn^C-resolution}, we have $P_{\chi_C}/\langle Y^C\otimes b_C^+ - X^C \otimes b_C^-, Y \otimes b_C^-, X \otimes b_C^+ \rangle_A \simeq F_{\chi_C}^{\sgn,\chi_C}$ and 
            \begin{align}
                \gch F_{\chi_C}^{\sgn,C} = (1-q)\gch P_{\chi_C} + (-q^C+-q^{C+2})\gch P_\sgn.
            \end{align}
            Hence, we obtain
            \begin{align}
                \gch \im\,f &= \gch F_{\chi_C}^\sgn - \gch F_{\chi_C}^\sgn/\im\,f\\ 
                    &=\bigl(\gch P_{\chi_C} -(q^C+q^{n-C})\gch P_\sgn\bigr) -\bigl((1-q)\gch P_{\chi_C} + (-q^C+-q^{C+2})\gch P_\sgn\bigr)\\
                    &= q\,\gch P_{\chi_C}-q^{\frac{n+1}{2}}\gch P_\sgn +q^{\frac{n+3}{2}} \gch P_\sgn\\
                    &=q(\gch P_{\chi_C} -(q^C+q^{n-C})\gch P_\sgn)=\gch F_{\chi_C}^\sgn\gs{1}.
            \end{align}
            Clearly, $\im\,f$ is a quotient module of $P_{\chi_C}\gs{1}$, and hence, we have $\im\,f \simeq F_{\chi_C}^\sgn\gs{1}$.
            Hence, we have 
            \[
                F_{\chi_C}^\sgn \filt \frac{1}{1-q} [F_{\chi_C}^{\sgn,\chi_C}]
            \]
            by Lemma \ref{lem:repeatingfilt}.
        \end{proof}

        \begin{lem}
            \label{lem:triv-socle}
            Let $K$ be a Springer correspondence.
            Let $\lambda\in\irr{W}$ and $V$ be the maximal semisimple submodule of $(K_\lambda)_{>0}$.
            A representation $V$ is isotypic semisimple of type $\triv$ or $\sgn$ if $V\neq 0$.
        \end{lem}
        \begin{proof}
            Since $(K_\lambda)_{>0}$ is finite dimensional, we have $V=0$ if and only if $(K_\lambda)_{>0}=0$.
            Suppose that $V$ admits a direct summand $L_\alpha\gs{a}$ with $\alpha\in\irr{W}$ and $a>0$.
            By Corollary \ref{cor:K-socle-sgn-relation} and Corollary \ref{cor:Ext-LL}, for a simple representation $\beta$ which is a direct summand of $\alpha\otimes  \sgn$, we have $\lambda\precsim_K\beta$.
            In particular, the isotypic component of $(K_\lambda)_{>0}$ associated to $\beta$ is zero.
            Then, at least one of the isotypic components of $V$ associated to $\triv$ or $\sgn$ is zero.
            We have $\chi_i\otimes \sgn \simeq \chi_i$ for $1\leq i \leq C$ by Proposition\ref{prop:tensorofreps}, and so the isotypic component of $V$ associated to a two dimensional irreducible representation $\chi_i$ is also zero.
            Therefore, a representation $V$ is isotypic semisimple of type $\triv$ or $\sgn$ if $V\neq 0$.
        \end{proof}

        \begin{cor}
            \label{cor:modified-springerorder}
            Let $K$ be a Springer correspondence.
            We have at least one of the followings:
                \begin{itemize}
                    \item  $[(K_\triv)_{>0}:\lambda]_q = [(K_\lambda)_{>0}:\sgn]_q=0$ holds for all $\lambda$,
                    \item  $[(K_\sgn)_{>0}:\lambda]_q = [(K_\lambda)_{>0}:\triv]_q=0$ holds for all $\lambda$.
                \end{itemize}

        \end{cor}
        \begin{proof}
            We keep the notation in the proof of Lemma \ref{lem:triv-socle}.
            If $V$ is a nonzero isotypic semisimple module of type $\triv$, then we have $\triv \prec_K \lambda$ by Corollary \ref{cor:higherK-relation} and \ref{cor:ext-eliminated}.
            We also have $\lambda \precsim_K \sgn$ by Corollary \ref{cor:K-socle-sgn-relation}, and so $\triv \prec_K \sgn$ holds.
            By swapping the role of $\triv$ and $\sgn$, we can also show that $\sgn\prec_K\triv$ holds if $V$ is a nonzero isotypic semisimple module of type $\sgn$.
            Therefore, for $\lambda,\mu\in\irr{W}$, the maximal semisimple submodules of $(K_\lambda)_{>0}$ and $(K_\mu)_{>0}$ are isotypic of the same type or at least one of them is zero. 

            

            If $(K_\lambda)_{>0}$ is zero for all $\lambda\in \irr{W}$, there is nothing left to prove.
            Suppose that the maximal semisimple submodule $V$ of $(K_\lambda)_{>0}$ is nonzero for some $\lambda\in\irr{W}$.
            If $V$ is an isotypic semisimple module of type $\triv$, then $(K_\mu)_{>0}$ is zero or admits a simple submodule which is isomorphic to a grading shift of $L_\triv$ for all $\mu\in \irr{W}$.
            Therefore, we have $(K_\mu)_{>0}=0$ or $\mu\precsim_K\sgn$, and so $[(K_\mu)_{>0}:\sgn]_q=0$ holds for all $\mu\in\irr{W}$.
            By the reflexivity of a preorder, the isotypic component of $(K_\triv)_{>0}$ associated to $\triv$ is zero.
            Consequently, we obtain $(K_\triv)_{>0}=0$ and the desired equality $[(K_\triv)_{>0}:\mu]_q=0$ for all $\mu\in\irr{W}$.
            When $V$ is an isotypic semisimple module of type $\sgn$, we can show that $[(K_\mu)_{>0}:\triv]_q=[(K_\sgn)_{>0}:\mu]_q=0$ holds for all $\mu\in\irr{W}$ in a similar fashion.

            
        \end{proof}

        We define the additional property concerning $K$.
        \begin{dfn}
            \label{dfn:additional-relation}
            \mbox{}
            \begin{enumerate}[\upshape(i)\itshape]
                \setcounter{enumi}{4}
                \item \label{item:triv-minimum}
                    $[(K_\triv)_{>0}:\lambda]_q = [(K_\lambda)_{>0}:\sgn]_q=0$ holds for all $\lambda\in\irr{W}$.
            \end{enumerate}
        \end{dfn}
        Due to the Corollary \ref{cor:modified-springerorder}, it suffices to consider with the property of Definition \ref{dfn:additional-relation} because the rest case is obtained by applying an equivalence functor $-\otimes\sgn$ to swap the role of $\triv$ and $\sgn$.
        If a Springer correspondence $K$ satisfies the property of Definition \ref{dfn:additional-relation}, then $K_\lambda$ is a quotient module of $F_\lambda^{\lambda,\sgn}$ for $\lambda\in\irr{W}$.

        \begin{lem}
            \label{lem:gchK-upperbound}
            For each $\lambda \in \irr{W}$, we define a module $J_\lambda$ by
            \[
                J_\lambda \coloneqq F_\lambda^{\lambda,\sgn} = P_\lambda \left/ \sum_{\mu\in\{\lambda,\sgn\},f\in\Hom_{A}(P_\mu,P_\lambda)_{>0}}\mathrm{Im}\,f \right.
            \]
            We have
            \begin{align}
                &\gch J_\sgn = [L_\sgn]+q^n[L_\triv] +\sum_{j=1}^\C(q^j + q^{n-j})[L_{\chi_j}] \qquad \textrm{ and}\\
                &\gch J_{\chi_i} = q^i[L_\triv] +\sum_{j=1}^i q^{i-j}[L_{\chi_j}] + \sum_{j=i+1}^\C(q^{j-i} + q^{n-j-i})[L_{\chi_j}] \qquad(1\leq i\leq C).
            \end{align}
        \end{lem}
        \begin{proof}
            By Corollary \ref{cor:gch-Psgn} and Lemma \ref{lem:F_sgn^sgn-resolution}, we obtain
            \begin{align}
                \gch F_\sgn^\sgn &= (1-q^2-q^n +q^{n+2})\gch P_\sgn\\
                &= [L_\sgn]+q^n[L_\triv]+\sum_{j=1}^{n-1}q^j[L_{\chi_j}] = [L_\sgn]+q^n[L_\triv]+\sum_{j=1}^C(q^j+q^{n-j})[L_{\chi_j}].
            \end{align}

            We calculate $\gch J_{\chi_i}$ for $1\leq i\leq \C$.
            Set 
            \[
                M=P_{\chi_i}/ \langle Y^i\otimes b_i^+ -X^i\otimes b_i^-, XY\otimes b_i^+, XY\otimes b_i^-, Y^{n-2i}\otimes b_i^-, X^{n-2i}\otimes b_i^+  \rangle_A.
            \]
            We show that $J_{\chi_i} = M$.
            We have a following projective resolution:            
            \begin{equation}\label{eq:res-Fi^isgn}\vcenter{\xymatrix{
                0 \ar[rrrr] &&&&{\begin{matrix} P_\sgn\gs{i+2}\\ \oplus \\ P_{\chi_{i+1}}\gs{n-2i+1} \end{matrix}} \ar[rrrr]^{\left\{b_\sgn\mapsto \begin{pmatrix} -XY\otimes b_\sgn \\ Y^i\otimes b_i^+ - X^i\otimes b_i^-\\0\end{pmatrix}\right\}}_{\left\{\begin{matrix}b_{i+1}^+\mapsto \begin{pmatrix}0 \\Y^{n-2i-1}\otimes b_i^- \\ -X\otimes b_i^+\end{pmatrix}\\b_{i+1}^-\mapsto \begin{pmatrix}0 \\X^{n-2i-1}\otimes b_i^+ \\ -Y\otimes b_i^-\end{pmatrix}\end{matrix}\right\}}  &&&&{\begin{matrix} P_\sgn\gs{i}\\ \oplus\\ P_{\chi_i}\gs{2} \\ \oplus \\ P_{\chi_i}\gs{n-2i} \end{matrix}}\\\\
                \ar[rrrr]^{\begin{pmatrix}\{b_\sgn \mapsto Y^i\otimes b_i^+ -X^i\otimes b_i^-\}\\\left\{\begin{matrix} b_i^+ \mapsto XY\otimes b_i^+\\ b_i^- \mapsto XY\otimes b_i^- \end{matrix} \right\}\\ \left\{\begin{matrix} b_i^+ \mapsto Y^{n-2i}\otimes b_i^-\\ b_i^- \mapsto X^{n-2i}\otimes b_i^+ \end{matrix} \right\}\end{pmatrix}} &&&&P_{\chi_i} \ar[r] &M \ar[r] &0.&&
            }}\end{equation}
            We verify the exactness of this chain complex.
            The second row of this complex is exact by inspection.
            The left map of the second row is a $S$-module homomorphism between free $S$-modules which has a matrix representation
            \begin{align}
                \label{eq:matrixrep}
                \begin{pmatrix}
                    Y^i   &XY &0  &0          &X^{n-2i}\\
                    -X^i  &0  &XY  &Y^{n-2i} &0\\
                \end{pmatrix}
            \end{align}
            with respect to the bases
            \begin{multline}
            \left\{\begin{pmatrix}1\otimes b_\sgn\\0\\0\end{pmatrix},\begin{pmatrix}0\\1\otimes b_{\chi_i}^+\\0\end{pmatrix},\begin{pmatrix}0\\1\otimes b_{\chi_i}^-\\0\end{pmatrix},\begin{pmatrix}0\\0\\1\otimes b_{\chi_i}^+\end{pmatrix},\begin{pmatrix}0\\0\\1\otimes b_{\chi_i}^-\end{pmatrix}\right\} \subset \begin{matrix} P_\sgn\gs{i}\\ \oplus\\ P_{\chi_i}\gs{2} \\ \oplus \\ P_{\chi_i}\gs{n-2i} \end{matrix} \\
            \textrm{ and } \left\{1\otimes b_i^+,1\otimes b_i^-\right\} \subset P_{\chi_i}.\\
            \end{multline}
            Suppose that a column vector
            \[
                \begin{pmatrix} f_1\\ f_2\\ f_3\\f_4\\ f_5 \end{pmatrix} \qquad (f_1,f_2,f_3,f_4, f_5 \in S)
            \]
            belongs to the kernel of the linear map \eqref{eq:matrixrep}.
            Then we have 
            \[
                f_1Y^i + f_2 XY + f_5 X^{n-2i} = -f_1X^i + f_3XY + f_4 Y^{n-2i} = 0.    
            \]
            Since $XY$ and $X^{n-2i}$ are divisible by $X$, thus $f_1Y^i = -f_2 XY - f_5 X^{n-2i}$ is also divisible by $X$.
            A polynomial $Y^i$ does not have an irreducible factor $X$.
            Therefore $f_1$ has an irreducible factor $X$.
            Similarly, $f_1X^i= f_3XY + f_4 Y^{n-2i}$ is divisible by $Y$.
            Since $X^i$ does not have an irreducible factor $Y$, thus $f_1$ also has an irreducible factor $Y$.
            Therefore, $f_1$ can be written as $f_1 = f_1'XY$ with $f_1'\in \cc[X,Y]$.
            Then we obtain $f_5 X^{n-2i} = XY(- f_1' Y^i- f_2)$ and  $f_4 Y^{n-2i} = XY(f_1'X^i -f_3)$.
            In particular, $f_5X^{n-2i}$ is divisible by $Y$, and $f_4 Y^{n-2i}$ is divisible by $X$.
            Since $X^{n-2i}$ does not have an irreducible factor $Y$, a polynomial $f_5$ can be written as $f_5 = f_5'Y$ with $f_5'\in \cc[X,Y]$.
            Similarly, $f_4$ can be written as $f_4 = f_4'X$ with $f_4'\in \cc[X,Y]$ because $Y^{n-2i}$ does not have an irreducible factor $X$.
            Then, we obtain $f_2 = -f_1'Y^i  -f_5'X^{n-2i-1}$ and $f_3= f_1' X^i -f_4'Y^{n-2i-1}$.
            Hence, we have 
            \[
                \begin{pmatrix} f_1\\ f_2\\ f_3\\f_4\\ f_5 \end{pmatrix} = f_1' \begin{pmatrix} XY\\ -Y^i\\ X^i\\0\\0 \end{pmatrix} + f_4'\begin{pmatrix} 0\\ 0\\ -Y^{n-2i-1}\\X\\ 0 \end{pmatrix} + f_5'\begin{pmatrix} 0\\ -X^{n-2i-1}\\ 0\\0\\Y \end{pmatrix}.
            \]
            Therefore, the kernel of the left map of the second row in a complex \eqref{eq:res-Fi^isgn} is generated by
            \[
                \begin{pmatrix} XY\otimes b_\sgn \\ -Y^i\otimes b_i^+ + X^i\otimes b_i^-\\0\end{pmatrix}, \begin{pmatrix}0 \\-Y^{n-2i-1}\otimes b_i^- \\ X\otimes b_i^+\end{pmatrix}, \textrm{ and }\begin{pmatrix}0 \\-X^{n-2i-1}\otimes b_i^+ \\ Y\otimes b_i^-\end{pmatrix}.
            \]
            Clearly, the image of the right map of the first row contains these generators.
            This implies the surjectivity onto the kernel.

            The right map of the first row is a $S$-module homomorphism between free $S$-modules which admits a matrix representation
            \[
                \begin{pmatrix}
                    -XY     &0          &0\\
                    Y^i     &0          &X^{n-2i-1}\\
                    -X^i    &Y^{n-2i-1} &0\\
                    0       &-X         &0\\
                    0       &0          &Y
                \end{pmatrix}
            \]
            with respect to the bases
            \begin{multline}
            \left\{\begin{pmatrix}1\otimes b_\sgn\\0\end{pmatrix},\begin{pmatrix}0\\1\otimes b_{\chi_{i+1}}^+\end{pmatrix},\begin{pmatrix}0\\1\otimes b_{\chi_{i+1}}^-\end{pmatrix}\right\} \textrm{ and }\\
            \left\{\begin{pmatrix}1\otimes b_\sgn\\0\\0\end{pmatrix},\begin{pmatrix}0\\1\otimes b_{\chi_i}^+\\0\end{pmatrix},\begin{pmatrix}0\\1\otimes b_{\chi_i}^-\\0\end{pmatrix},\begin{pmatrix}0\\0\\1\otimes b_{\chi_i}^+\end{pmatrix},\begin{pmatrix}0\\0\\1\otimes b_{\chi_i}^-\end{pmatrix}\right\}.  
            \end{multline}
            This matrix representation has a $(3\times 3)$-submatrix 
            \[
            \begin{pmatrix}
                -X^i    &Y^{n-2i-1} &0\\
                0       &-X         &0\\
                0       &0          &Y
            \end{pmatrix}\]
            obtained by deleting the first row and the second row.
            Clearly, the determinant of this submatrix is $-X^{i+1}Y$.
            By Lemma \ref{lem:detinjective}, the right map of the first row of \eqref{eq:res-Fi^isgn} is injective.

            Therefore, a chain complex \eqref{eq:res-Fi^isgn} is exact.
            Then, we have
            \begin{align}
                \gch M &= \gch P_{\chi_i} - \gch (P_\sgn\gs{i}\oplus P_{\chi_i}\gs{2} \oplus  P_{\chi_i}\gs{n-2i})   \\
                &\qquad+ \gch (P_\sgn\gs{i+2} \oplus P_{\chi_{i+1}}\gs{n-2i+1})\\
                &= (1-q^2)\gch P_{\chi_i} -q^i(1-q^2)\gch P_\sgn - q^{n-2i}(\gch P_{\chi_i} -\gch \chi_{i+1}\gs{1})
            \end{align}
            By Proposition \ref{prop:gch-S}, Corollary \ref{cor:gch-Psgn}, Corollary \ref{cor:gch-Pi} and Lemma \ref{lem:gch-diff}, we obtain
            \begin{align}\scriptsize
                \gch M &=
                \frac{1}{1-q^n} \left(\sum_{j=0}^{n-1} q^j\gch L_{\chi_{i+j}}+ \sum_{j=1}^n q^j\gch L_{\chi_{i-j}}\right)\\
                &\quad -\frac{q^i}{1-q^n}\left([L_\sgn]+q^n[L_\triv]+\sum_{j=1}^{n-1}q^j[L_{\chi_j}]\right)-\frac{q^{n-2i}}{1-q^n}\sum_{j=0}^{n-1} q^j\gch L_{\chi_{i-j}}\\
                &=\frac{1}{1-q^n} \left(\sum_{j=0}^{n-1}q^j\gch L_{\chi_{i+j}} + \sum_{j=1}^{i-1} q^j\gch L_{\chi_{i-j}}+ q^i([L_\triv]+[L_\sgn]) +\sum_{j=i+1}^n q^j\gch L_{\chi_{i-j}}\right)\\
                &\quad -\frac{1}{1-q^n}\left(q^i[L_\sgn]+q^{n+i}[L_\triv]+\sum_{j=i+1}^{n+i-1}q^{j}[L_{\chi_{j-i}}]\right) -\frac{1}{1-q^n}\sum_{j=n-2i}^{2n-2i-1} q^{j}\gch L_{\chi_{-i-j}}\\
                &=\frac{1}{1-q^n} \left(\sum_{j=0}^{n-2i-1}q^j\gch L_{\chi_{i+j}} - \sum_{j=n}^{2n-2i-1} q^j \gch L_{\chi_{i+j}} + \sum_{j=1}^{i-1} q^j\gch L_{\chi_{i-j}}\right.\\
                &\qquad\qquad\qquad\qquad\qquad\qquad \left.+  q^i[L_\triv]-q^{n+i}[L_\triv] -\sum_{j=n+1}^{n+i-1}q^{j}[L_{\chi_{i-j}}]\right)\\
                &= \sum_{j=0}^{n-2i-1}q^j [L_{\chi_{i+j}}] + q^i[L_\triv] + \sum_{j=1}^{i-1} q^j [L_{\chi_{i-j}}].
            \end{align}
            From this, we obtain
            \[
                [M:L_\mu]_q = \begin{cases}
                    1 &(\mu = \chi_i),\\
                    q^{i-k} &(\mu = \chi_k \text{ for } 1\leq k < i),\\
                    q^{k-i} + q^{n-k-i} &(\mu = \chi_k \text{ for } i< k \leq C),\\
                    q^i &(\mu = \triv), \\
                    0 &(\mu=\sgn).
                \end{cases}    
            \]
            In particular, we have $[M_{>0}:L_\sgn] = [M_{>0}:L_{\chi_i}] = 0$.
            Therefore, we conclude that $\langle Y^i\otimes b_i^+ -X^i\otimes b_i^-, XY\otimes b_i^+, XY\otimes b_i^-, Y^{n-2i}\otimes b_i^-, X^{n-2i}\otimes b_i^+  \rangle_A$ contains the isotypic components of $(P_{\chi_i})_{>0}$ associated to $\sgn$ and $\chi_i$,
            and hence we have $J_{\chi_i} = M$.
            Consequently, we obtain
            \[
                \gch J_{\chi_i} = \gch M = q^i[L_\triv] +\sum_{j=1}^i q^{i-j}[L_{\chi_j}] + \sum_{j=i+1}^\C(q^{j-i} + q^{n-j-i})[L_{\chi_j}]
            \]
            by Proposition \ref{lem:gch-diff} as required.
        \end{proof}

        \begin{lem}
            \label{lem:chain-2dimsimple}
            Let $K$ be a Springer correspondence.
            Let $i$ be a positive integer such that $1\leq i \leq C$ and $\lambda \in \irr{W}$.
            If $[(K_\lambda)_{>0}:L_{\chi_i}]_q \neq 0$ holds, then we have $[(K_\lambda)_{>0}:L_{\chi_j}]_q\neq 0$ for $1\leq j < i$.
        \end{lem}
        \begin{proof}
            It suffices to consider the case $i\geq 2$. 
            Assume to the contrary that there is a positive integer $j$ such that $1\leq j < i$ and $[(K_\lambda)_{>0}:L_{\chi_j}]_q\neq 0$.
            Let $d$ be the largest integer such that $[(K_\lambda)_d:\chi_k]\neq 0$ for some $k=j+1,j+2, \ldots ,C-1,C$.
            Since $K_\lambda$ is finite dimensional and $[(K_\lambda)_{>0}:L_{\chi_i}]_q \neq 0$ holds, we indeed have $d$ as above.
            Fix $k$ satisfying $[(K_\lambda)_d:\chi_k]\neq 0$ and $j < k \leq C$.
            By Corollary \ref{cor:appear-gep-vanish}, we have $\gep{K_\lambda}{L_{\chi_k}} = 0$.
            By Proposition \ref{prop:gEP-gch-Lmu}, the coefficient of $q^{-(d+2)}$ in a Laurent series $\gep{K_\lambda}{L_{\chi_k}}$ is 
            \[
                [(K_\lambda)_{d+2}:\chi_k]+[(K_\lambda)_{d}:\sgn\otimes\chi_k] -\sum_{\mu\in\irr{W}}[\chi_1\otimes\chi_k:\mu][(K_\lambda)_{d+1}:\mu]=0.
            \]
            Since $\sgn\otimes\chi_k \simeq \chi_k$ by Proposition \ref{prop:tensorofreps}, we have $[(K_\lambda)_{d}:\sgn\otimes\chi_k]>0$.
            Again by Proposition \ref{prop:tensorofreps}, we have $\chi_1\otimes\chi_k \simeq \chi_{k-1}\oplus \chi_{k+1}$.
            By Proposition \ref{prop:relationofreps}, we have $\chi_{k+1}\simeq\chi_C$ if $k=C$.
            If $k\neq C$ holds, then we have $j<k+1\leq C$.
            Therefore, we obtain $[(K_\lambda)_{d+1}:\chi_{k+1}] = 0$ by the maximality of $d$ as it holds when $k=\C$ and $k\neq \C$.
            If $k=j+1$, then we have $\chi_{k-1} = \chi_j$.
            In this case, we have $[(K_\lambda)_{d+1}:\chi_{k-1}] = 0$ by assumption $[(K_\lambda)_{>0}:L_{\chi_j}]_q\neq 0$.
            If $k\neq j+1$, then $j+1\leq k-1 \leq C$.
            Thus, we obtain $[(K_\lambda)_{d+1}:\chi_{k-1}] = 0$ by the maximality of $d$.
            Whether $k = j+1$ or not, we have $[(K_\lambda)_{d+1}:\chi_{k-1}] = 0$.
            Consequently, we obtain $\sum_{\mu\in\irr{W}}[\chi_1\otimes\chi_k:\mu][(K_\lambda)_{d+1}:\mu] = 0$.
            Then, the coefficient of $q^{-(d+2)}$ in $\gep{K_\lambda}{L_{\chi_k}}$ is nonzero.
            This contradicts $\gep{K_\lambda}{L_{\chi_k}} = 0$.
            Therefore, we conclude that $[(K_\lambda)_{>0}:L_{\chi_j}]_q = 0$ for $1\leq j < i$.
        \end{proof}

        \begin{prop}
            \label{prop:possible-D}
            Let $K$ be a Springer correspondence.
            Suppose that that the conditions in Definition \ref{dfn:additional-relation} is satisfied.
            We set $D_\lambda = \{\mu\in\irr{W} \mid \Ext^1_A(K_\lambda,L_\mu)\neq 0\}$ for each $\lambda\in\irr{W}$.
            Then we have
            \[
                D_\triv =  \{ \chi_1\},\quad D_\sgn =  \{\sgn\}\text{ or } \{\chi_1\},\quad   D_{\chi_1} = \{\chi_2,\sgn\} \text{ or } \{\chi_2,\triv,\sgn\} ,
            \]
            \[
                \text{and }  D_{\chi_i} = \{\chi_{i+1},\sgn\} \text{ or }\{\chi_{i+1},\chi_{i-1}\}  \qquad(2\leq i \leq C).
            \]
        \end{prop}
        \begin{proof}
            We have $K_\lambda = F_\lambda^{D_\lambda}$ for each $\lambda\in\irr{W}$ by Proposition \ref{prop:I-Ext1}.
            Since we assume the condition \eqref{item:triv-minimum} in Definition \ref{dfn:additional-relation}, we have $K_\triv = L_\triv$.
            It follows from Proposition \ref{prop:tensorofreps} and Corollary \ref{cor:Ext-LL} that $D_\triv =  \{ \chi_1\}$.
            Then, we obtain $\triv\precsim_K \chi_1$.
            Recall that $K_\lambda$ is a quotient module of $F_\lambda^{\lambda,\sgn}$ for each $\lambda\in\irr{W}$.
            Then we have an inequality
            \begin{align}
                \label{eq:quotient-ineq}
                [(K_\lambda)_d:\mu] \leq [(F_\lambda^{\lambda,\sgn})_d:\mu]    
            \end{align}
            for any $\lambda,\mu\in\irr{W}$ and $d\in\zz$.

            If $\chi_1 \in D_\sgn$, then we have $K_\sgn = L_\sgn$ because $(P_\sgn)_1$ is isomorphic to $\chi_1$ and generates $(P_\sgn)_{\geq 1}$.
            In this case, it follows from Corollary \ref{cor:Ext-LL} that $D_\sgn = \{\chi_1\}$.
            
            Suppose that $\chi_1 \not\in D_\sgn$ holds.
            Then we have $(K_\sgn)_1 \simeq \chi_1$.
            We prove by induction on $j$ that if $(K_\sgn)_1 \simeq \chi_1$ holds, then $[(K_\sgn)_j:\chi_j] \neq 0$ for $1\leq j \leq C$.
            The case $j=1$ is obvious.
            We prove the case $2\leq j \leq C$.
            Suppose that $[(K_\sgn)_{j-1}:\chi_{j-1}] \neq 0$ holds.
            Then, we have $\gep{K_\sgn}{L_{\chi_{j-1}}} =0$ by Corollary \ref{cor:appear-gep-vanish}.
            By Proposition \ref{prop:gEP-gch-Lmu} and Proposition \ref{prop:tensorofreps}, the term of $\gep{K_\sgn}{L_{\chi_{j-1}}}$ of degree $-(j+1)$ is 
            \[
                q^{-(j+1)}\left([(K_\sgn)_{j+1}:\chi_{j-1}]+[(K_\sgn)_{j-1}:\chi_{j-1}] -\sum_{\lambda\in\irr{W}}[\chi_{j-2}\oplus\chi_j:\lambda][(K_\sgn)_{j}:\lambda]\right).
            \]
            Therefore, we obtain
            \[
                [(K_\sgn)_{j+1}:\chi_{j-1}]+[(K_\sgn)_{j-1}:\chi_{j-1}] = \sum_{\lambda\in\irr{W}}[\chi_{j-1}\otimes\chi_1:\lambda][(K_\sgn)_{j}:\lambda].
            \]
            Since $[(K_\sgn)_{j-1}:\chi_{j-1}] \neq 0$ by the induction hypothesis, the left-hand side of the equality above is positive.
            Hence, there exists $\lambda\in\irr{W}$ such that $[\chi_{j-1}\otimes\chi_1:\lambda]\neq 0$ and $[(K_\sgn)_{j}:\lambda]\neq 0$.
            By Proposition \ref{prop:relationofreps} and \ref{prop:tensorofreps}, we obtain following semisimple decompositions:
            \[
                \chi_{j-1}\otimes\chi_1\simeq \begin{cases}
                    \triv \oplus \sgn \oplus \chi_j &(j=2),\\
                    \chi_{j-2}\oplus\chi_j &(2<j\leq C).
                \end{cases}
            \]
            By Lemma \ref{lem:gchK-upperbound}, we have 
            \[\gch F_\sgn^\sgn = [\sgn]+q^n[\triv] +\sum_{j=1}^\C(q^j + q^{n-j})[\chi_j].\]
            Clearly we have $j\neq 0$ and $j\neq n$.
            Hence, we obtain $[(K_\sgn)_{j}:\triv]=[(K_\sgn)_{j}:\sgn]=0$.
            Since we have $j\neq j-2$ and $j\leq C =\frac{n-1}{2} < n-(j-2) $, then it follows from \eqref{eq:quotient-ineq} that $[(K_\sgn)_{j}:\chi_{j-2}]\leq [(F_\sgn^\sgn)_j:\chi_{j-2}]=0$ if $j>2$ .
            Consequently, we have $[(K_\sgn)_{j}:\chi_j] \neq 0$ for $j=2,3,\ldots,C-1,C$.
            Then, we have $\chi_j\not\in D_\sgn$ for $j=2,3,\ldots,C-1,C$ by Corollary \ref{cor:ext-eliminated}.
            The coefficient of $q^{-2}$ in $\gep{K_\sgn}{L_\sgn}$ is 
            \[
                [(K_\sgn)_2:\sgn]+[(K_\sgn)_0:\triv]-[(K_\sgn)_1:\chi_1]
            \]  
            by Proposition \ref{prop:gEP-gch-Lmu} and Proposition \ref{prop:tensorofreps}.
            Since we have $[(K_\sgn)_2:\sgn]+[(K_\sgn)_0:\triv] = 0$ and $[(K_\sgn)_1:\chi_1]=1$, we obtain $[(K_\sgn)_2:\sgn]+[(K_\sgn)_0:\triv]-[(K_\sgn)_1:\chi_1]=-1<0$.
            In particular, we have $\Ext^1_A(K_\sgn,L_\sgn)_{-2}\neq 0$, and hence $\sgn\in D_\sgn$.
            By Lemma \ref{lem:triv-socle}, we have $[(K_\sgn)_{>0}:L_\triv]_q\neq 0$.
            Then, we conclude that $\triv$ does not belong to $D_\sgn$ by Corollary \ref{cor:ext-eliminated}.
            Consequently, we have $D_\sgn = \{\sgn\}$ if $\chi_1 \not\in D_\sgn$ holds.

            Fix an integer $i$ such that $1\leq i \leq \C$.
            We consider $D_{\chi_i}$ and $K_{\chi_i}$.
            If $(K_{\chi_i})_{>0}$ is zero, then we have $K_{\chi_i} = L_{\chi_i}$.
            In this case, $D_{\chi_i}$ is the set of irreducible additive summands of $\chi_1\otimes \chi_i$ by Corollary \ref{cor:Ext-LL}.
            By Proposition \ref{prop:relationofreps} and \ref{prop:tensorofreps}, we have
            \[
                \chi_1\otimes \chi_i = \begin{cases}
                    \triv\oplus\sgn\oplus\chi_2 &(i=1),\\
                    \chi_{i-1}\oplus\chi_{i+1} &(2\leq i \leq C).
                \end{cases}
            \]
            Therefore, we have
            \[
                D_{\chi_i} = \begin{cases}
                    \{\triv,\sgn,\chi_2\} &(i=1),\\
                    \{\chi_{i-1},\chi_{i+1}\} &(2\leq i \leq C).
                \end{cases}
            \]

            Next, we assume that a graded module $(K_{\chi_i})_{>0}$ is nonzero.
            Since $(P_{\chi_i})_{>0}$ is generated by the homogeneous part $(P_{\chi_i})_1$ of degree one, we have $(K_{\chi_i})_1 \neq 0$.
            Since the isotypic component of $(K_{\chi_i})_{>0}$ corresponding to $\chi_i$ is zero, we obtain $[(K_{\chi_i})_{>0}:L_{\chi_j}]_q=0$ for $i<j\leq C$ by Lemma \ref{lem:chain-2dimsimple}.
            By the condition \ref{item:triv-minimum} in Definition \ref{dfn:additional-relation}, we also have $[(K_{\chi_i})_{>0}:L_\sgn]_q=0$.
            Therefore, $(K_{\chi_i})_{>0}$ consists of its isotypic components of type $\chi_1,\chi_2,\ldots,\chi_{i-2},\chi_{i-1},$ and $\triv$.
            Moreover, these isotypic components are all nonzero.
            Indeed, we have $[(K_{\chi_i})_{>0}:L_\triv]_q\neq 0$ by assumption $(K_{\chi_i})_{>0}\neq 0$ and Lemma \ref{lem:triv-socle}.
            We see $[(K_{\chi_i})_{>0}:L_{\chi_j}]_q\neq 0$ for $j = 1,2,\ldots, i-1$.
            There is nothing to show for $i=1$, so let $i>1$.
            Then we have $(P_{\chi_i})_1\simeq \chi_1\otimes\chi_i$ and its semisimple decomposition $\chi_1\otimes\chi_i\simeq \chi_{i-1}\oplus\chi_{i+1}$. 
            If $i\neq C$, then we have $i<i+1\leq C$.
            In this case, we have already seen $[(K_{\chi_i})_{>0}:L_{\chi_{i+1}}]_q\neq 0$.
            If $i=C$, then $\chi_{i+1}=\chi_{C+1}\simeq \chi_C=\chi_i$ by Proposition \ref{prop:relationofreps}.
            Thus, whether $i=C$ or not, we have $[(K_{\chi_i})_{>0}:L_{\chi_{i+1}}]_q\neq 0$.
            Therefore, nonzero representation $(K_{\chi_i})_1$ is isomorphic to $\chi_{i-1}$.
            By Lemma \ref{lem:chain-2dimsimple}, we obtain $[(K_{\chi_i})_{>0}:L_{\chi_j}]_q\neq 0$ for $j = 1,2,\ldots, i-1$.
            Now we have 
            \[ 
                \gch F_{\chi_i}^{\sgn,\chi_i} = q^i[L_\triv] +\sum_{j=1}^i q^{i-j}[L_{\chi_j}] + \sum_{j=i+1}^\C(q^{j-i} + q^{n-j-i})[L_{\chi_j}] 
            \]
            by Lemma \ref{lem:gchK-upperbound}.
            In particular, each of irreducible representations $\chi_1,\chi_2,\ldots,\chi_{i-2},\chi_{i-1},$ and $\triv$ appears in $F_{\chi_i}^{\sgn,\chi_i}$ with multiplicity one.
            Since $K_{\chi_i}$ has nonzero isotypic component of type $\mu$ for each $\mu =\chi_1,\chi_2,\ldots,\chi_{i-2},\chi_{i-1},$ and $\triv$, we obtain $[K_{\chi_i}:L_\mu]_q = [F_{\chi_i}^{\sgn,\chi_i}:L_\mu]_q$.
            Clearly, we also have $[K_{\chi_i}:L_{\chi_i}]_q = 1$.
            Hence, we have $\gch K_{\chi_i} = q^i[L_\triv] +\sum_{j=1}^i q^{i-j}[L_{\chi_j}]$. 
            By Corollary \ref{cor:gch-standard-K_i} and Proposition \ref{lem:F-gch-identify}, we obtain that $K_{\chi_i}\simeq F_{\chi_i}^{\sgn,\chi_{i+1}}$.
            Therefore, $D_{\chi_i}$ be a subset of $\{\sgn,\chi_{i+1}\}$.
            By \ref{prop:gEP-gch-Lmu}, the coefficient of $q^{-i}$ in $\gep{F_{\chi_i}^{\sgn,\chi_{i+1}}}{L_\sgn}$ is $[(F_{\chi_i}^{\sgn,\chi_{i+1}})_i:\sgn] - [(F_{\chi_i}^{\sgn,\chi_{i+1}})_{i-1}:\chi_i] + [(F_{\chi_i}^{\sgn,\chi_{i+1}})_{i-2}:\triv] = 0-1+0=-1<0$.
            Thus, $\Ext^1_A (F_{\chi_i}^{\sgn,\chi_{i+1}},L_\sgn)_{-i}$ is nonzero, and hence, $\sgn$ belongs to $D_{\chi_i}$.
            Again by Proposition \ref{prop:gEP-gch-Lmu}, the coefficient of $q^{-1}$ in $\gep{F_{\chi_i}^{\sgn,\chi_{i+1}}}{L_\sgn}$ is $[(F_{\chi_i}^{\sgn,\chi_{i+1}})_1:\chi_{i+1}] - [(F_{\chi_i}^{\sgn,\chi_{i+1}})_0:\chi_i] - [(F_{\chi_i}^{\sgn,\chi_{i+1}})_0:\chi_{i+2}] + [(F_{\chi_i}^{\sgn,\chi_{i+1}})_{-1}:\chi_{i+1}] = 0-1-0+0=-1<0$.
            Then, $\Ext^1_A (F_{\chi_i}^{\sgn,\chi_{i+1}},L_{\chi_{i+1}})_{-1}$ is nonzero, and hence, $\chi_{i+1}$ belongs to $D_{\chi_i}$.
            Consequently, we have $D_{\chi_i}=\{\sgn,\chi_{i+1}\}$ if $(K_{\chi_i})_{>0}\neq 0$.
        \end{proof}

        \begin{prop}
            \label{prop:F-destination}
            Let $\lambda\in \irr{W}$ and let $D$ be a subset of $\irr{W}$ which is one on the list of $D_\lambda$ in Proposition \ref{prop:possible-D}.
            For $\mu\in \irr{W}$, we have $\Ext^1_A(F_\lambda^D,L_\mu)\neq 0$ if and only if $\mu\in D$ holds.
        \end{prop}
        \begin{proof}
            When $(\lambda,D) = (\triv,\{\chi_1\}), (\sgn,\{\chi_1\}) , (\chi_1,\{\chi_2,\triv,\sgn,\}),$ or $(\chi_i,\{\chi_{i+1},\chi_{i-1}\})$ for $i\leq 2\leq C$, a graded $A$-module $F_\lambda^D$ is isomorphic to $L_\lambda$.
            Then the assertion follows from Corollary \ref{cor:Ext-LL} and Proposition \ref{prop:tensorofreps}.
            When $(\lambda,D) = (\sgn,\{\sgn\})$, we obtain a projective resolution of $F_\sgn^D = F_\sgn^\sgn$ by Lemma \ref{lem:F_sgn^sgn-resolution}.
            By inspecting the form of this projective resolution, we can see that $\Ext^1_A(F_\sgn^\sgn,L_\mu)\neq 0$ if and only if $\mu=\sgn$.
            If $(\lambda,D) = (\chi_i,\{\chi_{i+1},\sgn\})$ for $1\leq i < C$, we obtain the projective resolution of $F_\lambda^D$ by Lemma \ref{lem:F_i^sgn^j-resolution}.
            When $(\lambda,D) = (\chi_C,\{\chi_{C+1},\sgn\})$ for $1\leq i < C$, we have $F_\lambda^D = F_{\chi_C}^{\chi_{C+1},\sgn}=F_{\chi_C}^{\chi_C,\sgn}$.
            Then we obtain the projective resolution of $F_{\chi_C}^{\chi_C,\sgn}$ by Lemma \ref{lem:F_C^sgn^C-resolution}.
            We can show the assertion by inspecting the forms of these projective resolutions.
        \end{proof}
        \begin{prop}
            \label{prop:F_sgn^sgn-repeat}
            A projective module $P_\sgn$ has a filtration which has the multiplicity of factors
            \[
                P_\sgn\filt \frac{1}{1-q^2}\cdot\frac{1}{1-q^n}[F_\sgn^\sgn].
            \]
        \end{prop}
        \begin{proof}
        There is an injective map
            \[
                \begin{matrix}
                    P_\sgn\gs{2} &\longrightarrow &P_\sgn\\
                        1\otimes b_\sgn &\longmapsto &XY\otimes b_\sgn.
                \end{matrix}
            \]
            This map is $A$-module homomorphism because $XY\in R$ is commutative with $A$-action.
            By Lemma \ref{lem:repeatingfilt}, we have 
            \[P_\sgn \filt \frac{1}{1-q^2} [P_\sgn/XY\cdot P_\sgn].\]
            In particular, we have $\gdim P_\sgn/(XY\cdot P_\sgn) = (1-q^2)\gdim P_\sgn = (1+q)/(1-q)$.

            We have a $A$-module homomorphism
            \[
                \begin{matrix}
                    (P_\sgn/(XY\cdot P_\sgn)) \gs{n} &\longrightarrow &P_\sgn/(XY\cdot P_\sgn)\\
                        1\otimes b_\sgn &\longmapsto &(X^n+Y^n)^j\otimes b_\sgn.
                \end{matrix}
            \]
            Since $S$ is an unique factorization domain and $XY,X^n+Y^n\in S$ are coprime, this homomorphism is injective.
            Thus, we obtain
            \[
                P_\sgn/(XY\cdot P_\sgn) \filt \frac{1}{1-q^n} \cdot [P_\sgn/(XY\cdot P_\sgn + (X^n+Y^n)\cdot P_\sgn)].     
            \]
            Therefore, we have $\gdim P_\sgn/(XY\cdot P_\sgn + (X^n+Y^n)\cdot P_\sgn) = (1-q^n)\gdim P_\sgn/XY\cdot P_\sgn = (1+q)(1+q+q^2+\cdots+q^{n-1})$.
            In particular, $P_\sgn/(XY\cdot P_\sgn + (X^n+Y^n)\cdot P_\sgn)$ is finite dimensional.
            By Lemma \ref{lem:cnctfilts}, we have
            \[
                P_\sgn\filt \frac{1}{(1-q^2)(1-q^n)}\cdot[P_\sgn/(XY\cdot P_\sgn + (X^n+Y^n)\cdot P_\sgn)].    
            \]
            Then we have 
            \begin{align}
                \gch P_\sgn/(XY\cdot P_\sgn + (X^n+Y^n)\cdot P_\sgn) &= (1-q^2)(1-q^n)\gch P_\sgn   \\
                    &= [L_\sgn]+q^n[L_\triv]+\sum_{i=1}^{n-1}q^i[L_{\chi_i}]
            \end{align}
            by Corollary \ref{cor:gch-Psgn}.
            Since $[P_\sgn/(XY\cdot P_\sgn + (X^n+Y^n)\cdot P_\sgn):L_\sgn]_q = 1$ holds, a submodule $(XY\cdot P_\sgn + (X^n+Y^n)\cdot P_\sgn)$ of $P_\sgn$ contains isotypic component of $(P_\sgn)_{\geq 1}$ associated to $\sgn$.
            Consequently, we have $P_\sgn/(XY\cdot P_\sgn + (X^n+Y^n)\cdot P_\sgn) \simeq F_\sgn^\sgn$ and 
            \[
                P_\sgn\filt \frac{1}{(1-q^2)(1-q^n)}\cdot[F_\sgn^\sgn].    
            \]
        \end{proof}

        \begin{thm}
            \label{thm:mainthm}
            Let $K=\{K_\lambda\}_{\lambda\in\irr{W}}$ be a family of graded $A$-modules.
            Suppose that $K$ satisfies the property in Definition \ref{dfn:additional-relation}.
            A family of graded modules $K$ is a Springer correspondence if and only if $K$ is one of the following:
            \begin{enumerate}
                \item $K_\triv= F_\triv^{\chi_1}$, $K_\sgn = F_\sgn^{\chi_1}$, $K_{\chi_1} = F_{\chi_1}^{\triv,\sgn,\chi_2}$, and $K_{\chi_i} = F_{\chi_i}^{\chi_{i-1},\chi_{i+1}}$ for $2\leq i\leq C$.
                \item $K_\triv= F_\triv^{\chi_1}$, $K_\sgn = F_\sgn^{\sgn}$, $K_{\chi_1} = F_{\chi_1}^{\triv,\sgn,\chi_2}$ or $F_{\chi_1}^{\sgn,\chi_2}$, and $K_{\chi_i} = F_{\chi_i}^{\chi_{i-1},\chi_{i+1}}$ or $F_{\chi_i}^{\sgn,\chi_{i+1}}$ for $2\leq i\leq C$.
            \end{enumerate}
        \end{thm}
        \begin{proof}
            At the beginning, we prove the ``only if'' part.
            Suppose that the conditions in Definition \ref{dfn:setting-problem} are satisfied.
            By Proposition \ref{prop:possible-D}, we have $K_\triv= F_\triv^{\chi_1}$, $K_\sgn = F_\sgn^{\sgn}$ or $F_\sgn^{\chi_1}$, $K_{\chi_1} = F_{\chi_1}^{\triv,\sgn,\chi_2}$ or $F_{\chi_1}^{\sgn,\chi_2}$, and $K_{\chi_i} = F_{\chi_i}^{\chi_{i-1},\chi_{i+1}}$ or $F_{\chi_i}^{\sgn,\chi_{i+1}}$ for $2\leq i\leq C$.
            Thus, it suffices to show that if $K_\sgn = F_\sgn^{\chi_1}$ then $K_{\chi_1}$ cannot be $F_{\chi_1}^{\sgn,\chi_2}$ and $K_{\chi_i}$ cannot be $F_{\chi_i}^{\sgn,\chi_{i+1}}$ for $2\leq i\leq C$.
            Since $F_\sgn^{\chi_1} = L_\sgn$, we have $\Ext^2_A(F_\sgn^{\chi_1},L_\triv)\neq 0$ by Corollary\ref{cor:Ext-LL}.
            Therefore, we have $\sgn \precsim_K \triv$ if $K_\sgn = F_\sgn^{\chi_1}$.
            
            Assume to the contrary that $K_{\chi_1} = F_{\chi_1}^{\sgn,\chi_2}$.
            Then, we have $\chi_1\precsim_K \sgn$ by Proposition \ref{prop:F-destination}, and hence, $\chi_1\precsim_K \triv$ holds.
            Therefore, the isotypic component of $K_{\chi_1}$ corresponding to $\triv$ must be zero.
            However, it follows from Corollary \ref{cor:gch-standard-K_i} that $[F_{\chi_1}^{\sgn,\chi_2}:L_\triv]_q = q$.
            This is a contradiction.
            Consequently, $K_{\chi_1}$ cannot be $F_{\chi_1}^{\sgn,\chi_2}$ if $K_\sgn = F_\sgn^{\chi_1}$.
            
            Fix an integer $i$ with $2\leq i\leq C$ and assume $K_{\chi_i} = F_{\chi_i}^{\sgn,\chi_{i+1}}$.
            In a manner similar to the case $K_{\chi_1} = F_{\chi_1}^{\sgn,\chi_2}$, we can show $\chi_i\precsim_K \triv$ and $[F_{\chi_i}^{\sgn,\chi_{i+1}}:L_\triv]_q\neq 0$.
            Therefore, $K_{\chi_i} \neq F_{\chi_i}^{\sgn,\chi_{i+1}}$ if $K_\sgn = F_\sgn^{\chi_1}$.

            In the second place, to show the ``if'' part, we verify that the conditions in Definition \ref{dfn:setting-problem} are indeed satisfied if $K=\{K_\lambda\}_{\lambda\in\irr{W}}$ is one in the above list.
            We define a graded $A$-module $\widetilde{K}_\lambda$ for each $\lambda\in\irr{W}$ by 
            \[
                \widetilde{K}_\lambda \simeq P_\lambda \left/ \sum_{\mu\succ_K\lambda,f\in\Hom_{A}(P_\mu,P_\lambda)_{>0}}\im\,f \right..
            \]

            We assume that $K_\triv= F_\triv^{\chi_1}$, $K_\sgn = F_\sgn^{\chi_1}$, $K_{\chi_1} = F_{\chi_1}^{\triv,\sgn,\chi_2}$, and $K_{\chi_i} = F_{\chi_i}^{\chi_{i-1},\chi_{i+1}}$ for $2\leq i\leq C$.
            By Proposition \ref{prop:relationofreps} and \ref{prop:tensorofreps}, it is straight forward to see that $(K_\lambda)_1 = 0$ for each $\lambda\in\irr{W}$.
            Thus, we have $K_\lambda = L_\lambda$ for $\lambda\in\irr{W}$.
            By Proposition \ref{prop:F-destination}, all irreducible representations are equivalent with respect to the preorder $\precsim_K$.
            Therefore, $\widetilde{K}_\lambda$ is isomorphic to a projective module $P_\lambda$ for each $\lambda\in\irr{W}$.
            In this case, the conditions in Definition \ref{dfn:setting-problem} are clearly satisfied.

            Let $l$ be a non-negative integer and let $0=a_0 < a_1 < a_2 < \cdots < a_l \leq C$ be an increasing sequence of integers.
            We set $K_\triv = F_\triv^{\chi_1}$, $K_\sgn = F_\sgn^{\sgn}$,
            \begin{align}
                K_{\chi_1} = \begin{cases}
                    F_{\chi_1}^{\sgn ,\chi_2} &(1\in\{a_1,a_2,\ldots,a_l\}),\\
                    F_{\chi_1}^{\triv,\sgn,\chi_2} &(1\not\in\{a_1,a_2,\ldots,a_l\}),
                \end{cases}
            \end{align}
            and 
            \begin{align}
                K_{\chi_i} = \begin{cases}
                    F_{\chi_i}^{\sgn ,\chi_{i+1}} &(i\in\{a_1,a_2,\ldots,a_l\}),\\
                    F_{\chi_i}^{\chi_{i-1},\chi_{i+i}} &(i\not\in\{a_1,a_2,\ldots,a_l\})
                \end{cases}
            \end{align}
            for $2\leq i\leq C$.
            By Proposition \ref{prop:F-destination}, we can illustrate the preorder $\precsim_K$ as follows:
            \[
                \begin{pmatrix} (\triv)  \\\prec_K (\chi_{a_1} \sim_K \chi_2 \sim_K \cdots \sim_K \chi_{a_2-1}) \\
                                     \prec_K (\chi_{a_2} \sim_K \chi_{a_2+1} \sim_K \cdots \sim_K \chi_{a_3-1})\\
                                    \vdots\\
                                     \prec_K (\chi_{a_{l-1}} \sim_K \chi_{a_{l-1}+1} \sim_K \cdots \sim_K \chi_{a_l-1})\\
                                     \prec_K (\chi_{a_l} \sim_K \chi_{a_l+1} \sim_K \cdots \sim_K \chi_{C}) \\
                                     \prec_K (\sgn)
                \end{pmatrix} \qquad\qquad(a_1 = 1),\\
            \]   
            \[  
            \begin{pmatrix}(\triv \sim_K \chi_1 \sim_K \chi_2 \sim_K \cdots \sim_K \chi_{a_1-1}) \\
                \prec_K (\chi_{a_1} \sim_K \chi_{a_1+1} \sim_K \cdots \sim_K \chi_{a_2-1})\\
                \vdots\\
                \prec_K (\chi_{a_{l-1}} \sim_K \chi_{a_{l-1}+1} \sim_K \cdots \sim_K \chi_{a_l-1})\\
                \prec_K (\chi_{a_l} \sim_K \chi_{a_l+1} \sim_K \cdots \sim_K \chi_{C}) \\\prec_K (\sgn)
            \end{pmatrix} \qquad(a_1 \neq 1 \textrm{ or } l=0).
            \]
    
            Therefore, we have $\widetilde{K}_\sgn = P_\sgn$, 
            \[
                \widetilde{K}_\triv = \begin{cases}
                    F_\triv^{\chi_{a_1}} &(l\geq 1),\\
                    F_\triv^{\sgn}      &(l=0),
                \end{cases}
                \textrm{ and }
                \widetilde{K}_{\chi_i} = \begin{cases}
                    F_{\chi_i}^{\sgn,\chi_{a_k}} &(a_{k-1}\leq i < a_k),\\
                    F_{\chi_i}^{\sgn} &( a_l\leq i \textrm{ or } l=0).
                \end{cases}
            \]

            Clearly, $P_\sgn = \widetilde{K}_\sgn$ is filtered by $\widetilde{K}=\{\widetilde{K}_\mu\}_{\mu\in\irr{W}}$.
            By Lemma \ref{lem:F_triv^sgn-resolution}, we have $P_\triv \filt  [F_\triv^\sgn] + q^n[P_\sgn]$.
            When $l=0$, we have $\widetilde{K}_\triv = F_\triv^\sgn$ and $\widetilde{K}_\sgn = P_\sgn$.
            Hence, $P_\triv$ is filtered by $\widetilde{K}$ if $l=0$.
            If $l\geq 1$, we obtain $P_\triv \filt  [F_\triv^{\chi_{a_1}}]+ q^{a_1}[F_{\chi_{a_1}}^\sgn] + q^n[P_\sgn]$ by Lemma \ref{lem:cnctfilts} and Proposition \ref{prop:F_triv^sgn-filt-F_triv^j-F_j-sgn}.
            Applying Lemma \ref{lem:cnctfilts} and Proposition \ref{prop:F_i^sgn-filt-F_i^sgn^j-F_j^sgn} repeatedly, we obtain
            \[
                P_\triv \filt  [F_\triv^{\chi_{a_1}}]+ \sum_{j=1}^{l-1} q^{a_j}[F_{\chi_{a_j}}^{\sgn,\chi_{a_{j+1}}}] + q^{a_l} [F_{\chi_{a_l}}^\sgn] + q^n[P_\sgn].
            \]
            Thus, $P_\triv$ is filtered by $\widetilde{K}$ even if $l\geq 1$.
            
            By Lemma \ref{lem:F_i^sgn-resolution}, we have $P_{\chi_i}\filt [F_{\chi_i}^{\sgn}]+(q^i+q^{n-i})[P_\sgn]$ for $1\leq i\leq C$.
            Therefore, $P_{\chi_i}$ is filtered by $\widetilde{K}$ if $l=0$ or $a_l\leq i$.
            When $a_l > i$, there exists some positive integer $k$ such that $a_{k-1}\leq i<a_k$.
            Then, we obtain 
            \[
                P_{\chi_i}\filt [F_{\chi_i}^{\sgn,\chi_{a_k}}] + \sum_{j=k}^{l-1} q^{a_j-i}[F_{\chi_{a_j}}^{\sgn,\chi_{a_{j+1}}}]+ q^{a_l-i}[F_{\chi_{a_l}}^\sgn]+(q^i+q^{n-i})[P_\sgn]
            \]
            by applying Lemma \ref{lem:cnctfilts} and Proposition \ref{prop:F_i^sgn-filt-F_i^sgn^j-F_j^sgn} repeatedly.
            Therefore, $P_{\chi_i}$ is filtered by $\widetilde{K}$ whether there exists $k$ as above or not.
            Consequently, the condition \eqref{item:PKtildefilter} is satisfied.

            By Corollary \ref{cor:gch-standard-K_i} we have  
            \[
                \gch K_\triv = [L_\triv], \quad \gch K_{\chi_1} =  \begin{cases}
                        [L_{\chi_1}]+q^1[L_\triv] &(1\in\{a_1,a_2,\ldots,a_l\}),\\
                        [L_{\chi_1}] &(1\not\in\{a_1,a_2,\ldots,a_l\}),
                    \end{cases}
            \]
            and
            \[
                \gch K_{\chi_i} = \begin{cases}
                    q^i[L_\triv] + \sum_{j=1}^i q^{i-j}[L_{\chi_j}] &(i\in\{a_1,a_2,\ldots,a_l\}),\\
                    [L_{\chi_i}] &(i\not\in\{a_1,a_2,\ldots,a_l\})
                \end{cases}
            \]
            for $2\leq i\leq C$.
            We also have  
            \[
                \gch K_\sgn = \gch F_\sgn^\sgn = [L_\sgn]+q^n[L_\triv]+\sum_{j=1}^{C}(q^j+q^{n-j})[L_{\chi_j}]
            \]
            by Lemma \ref{lem:gchK-upperbound}.
            Hence, we obtain
            \[
              [K_\triv:L_\mu]_q = \begin{cases} 
                1 &(\mu=\triv),\\
                 0 &(\mu\neq \triv),\end{cases}  \qquad
              [K_\sgn:L_\mu]_q = \begin{cases} 
                1 &(\mu=\sgn),\\
                q^j+q^{n-j} &(\mu = \chi_j \textrm{ for } 1\leq j\leq C),\\
                q^n &(\mu=\triv),\end{cases}
            \]
            \[
            \textrm{and } [K_{\chi_i}:L_\mu]_q = \begin{cases} 
                1 &(\mu=\chi_i),\\ 
                q^{i-j} &(\mu = \chi_j \textrm{ for } 1\leq j < i \textrm{ and } i\in\{a_1,a_2,\ldots,a_l\}),\\
                q^i &(\mu = \triv \textrm{ and } i\in \{a_1,a_2,\ldots,a_l\}),\\
                0 &(otherwise)\end{cases}
            \]
            for $1\leq i\leq C$.
            Then, for each $\lambda,\mu\in\irr{W}$, the number of factors in $\widetilde{K}$-filtration of $P_\lambda$ which is isomorphic to $\widetilde{K}_\mu\gs{d}$ is the same as the multiplicity of $\lambda$ in $(K_{\mu})_d$ by inspection.
            Hence, by Proposition \ref{prop:verify-orth}, $\widetilde{K}$ and $K$ satisfies the orthogonality condition.
            This completes the proof.
        \end{proof}
        In view of Proposition \ref{prop:lessrelations}, a family of modules $K=\{K_\lambda\}_{\lambda\in\irr{W}}$ may be obtained from different preorders; that is, there may exist a preorder $\precsim$ containing much more relations than $\precsim_K$ such that its trace quotient modules are $\{K_\lambda\}$.
        However, the following Corollary shows that $\precsim_K$ is a unique preorder whose trace quotient module is $K$ if $K$ is the one given in Theorem \ref{thm:mainthm}. 
        \begin{cor}
            \label{cor:preorder-uniqueness}
            Let $K=\{K_\lambda\}_{\lambda\in\irr{W}}$ be one of the families of modules given in Theorem \ref{thm:mainthm}.
            Let $\lambda,\mu \in \irr{W}$ such that $\lambda \not\precsim_K \mu$ holds.
            There exists $\lambda',\mu'\in\irr{W}$ such that $\lambda'\sim_K \lambda$, $\mu'\sim_K \mu$ and $[K_{\lambda'}:L_{\mu'}]_q\neq 0$ hold. 
        \end{cor}
        \begin{proof}
            If $K=\{L_\lambda\}_{\lambda\in\irr{W}}$, then we have $\lambda \precsim_K \mu$ for all $\lambda,\mu\in\irr{W}$.
            In this case, there is nothing to prove.
            Let $1 \leq a_1 < a_2 < \cdots < a_l \leq C$ be an increasing sequence of integers borrowed from the proof of Theorem \ref{thm:mainthm}.
            We have $K_\triv = F_\triv^{\chi_1}$, $K_\sgn = F_\sgn^{\sgn}$,
            \begin{align}
                K_{\chi_1} = \begin{cases}
                    F_{\chi_1}^{\sgn ,\chi_2} &(1\in\{a_1,a_2,\ldots,a_l\}),\\
                    F_{\chi_1}^{\triv,\sgn,\chi_2} &(1\not\in\{a_1,a_2,\ldots,a_l\}),
                \end{cases}
            \end{align}
            and 
            \begin{align}
                K_{\chi_i} = \begin{cases}
                    F_{\chi_i}^{\sgn ,\chi_{i+1}} &(i\in\{a_1,a_2,\ldots,a_l\}),\\
                    F_{\chi_i}^{\chi_{i-1},\chi_{i+i}} &(i\not\in\{a_1,a_2,\ldots,a_l\})
                \end{cases}
            \end{align}
            for $2\leq i\leq C$.
            There is a complete system of representatives of an equivalent relation $\sim_K$ given by
            \[
                \triv \succ_K \chi_{a_1} \succ_K \chi_{a_2} \succ_K \cdots \succ_K \chi_{a_l} \succ_K \sgn.    
            \]
            We choose $\lambda'$ and $\mu'$ from the set $\mathcal{S} = \{\triv, \chi_{a_1}, \chi_{a_2},\ldots, \chi_{a_l}, \sgn\}$.
            It suffices to show that $[K_{\lambda'}:L_{\mu'}]_q\neq 0$ if $\lambda' \not\precsim_K \mu'$.
            Since we have formulas of graded characters
            \[
                \gch F_{\chi_i}^{\sgn,\chi_{i+1}} = q^i[L_\triv] +\sum_{j=1}^i q^{i-j}[L_{\chi_j}]
            \]
            and 
            \[    
                \gch F_\sgn^\sgn = [L_\sgn]+q^n[L_\triv]+\sum_{j=1}^C(q^j+q^{n-j})[L_{\chi_j}],
            \]
            we find that $[K_{\lambda'}:L_{\mu'}]_q\neq 0$ for $\lambda' \not\precsim_K \mu'$ by inspection.
            This completes the proof.
        \end{proof}
        By Corollary \ref{cor:preorder-uniqueness}, we cannot add relations to the preorder $\precsim_K$ without changing the trace quotient modules.
        In particular, a preorder $\precsim_K$ is a unique preorder whose trace quotient modules are $K=\{K_\lambda\}_{\lambda\in\irr{W}}$.

        \bibliographystyle{plain}
        \bibliography{reflist}
    
\end{document}